\documentclass[aap,preprint]{imsart}

\RequirePackage[OT1]{fontenc}
\RequirePackage{amsthm,amsmath}
\RequirePackage[authoryear]{natbib}
\RequirePackage[colorlinks,citecolor=blue,urlcolor=blue]{hyperref}

\usepackage{enumitem}

\usepackage{amssymb}
\usepackage{amsmath}
\usepackage{mathrsfs}

\usepackage{tikz}

\usepackage[normalem]{ulem}

\usepackage{graphicx}

\usepackage{color}
\usepackage{amsthm}
\usepackage{bm}
\usepackage{bbm}
\usepackage{mathtools}

\usepackage{mdframed}

\mdfdefinestyle{MyFrame}{%
    linecolor=purple,
    outerlinewidth=2pt,
    roundcorner=20pt,
    innertopmargin=\baselineskip,
    innerbottommargin=\baselineskip,
    innerrightmargin=20pt,
    innerleftmargin=20pt,
    backgroundcolor=gray!50!white}

\usepackage{hyperref}
\usepackage{xcolor}\hypersetup{allcolors=blue,colorlinks=true}

\newcommand{\R}{\mathbb{R}}

\newcommand{\var}{\mathbf{var\,}}
\newcommand{\E}{\mathbf{E}}
\renewcommand{\P}{\mathbf{P}}
\newcommand{\I}{\mathbbm{1}}
\newcommand{\D}{\mathbb{D}}

\newcommand{\one}[1]{\mathbbm{1}_{\{#1\}}}

\renewcommand{\d}{\mathrm{d}}

\newcommand{\stcomp}{\mathsf{c}}
\newcommand{\cl}{-}
\newcommand{\lip}{\mathrm{Lip}}

\newcommand{\chck}{{\color{red} [check again] }}
\renewcommand{\chck}{}

\newcommand{\rvout}[1]{{\color{red}{\sout{#1}}}}
\renewcommand{\rvout}[1]{}

\newcommand{\cf}[1]{$ $\\\fbox{\begin{minipage}{0.95\linewidth}{\large\textup{\texttt{[SUPPLEMENT] !!!THESE BOXES ARE ONLY FOR THE PURPOSE OF FACILITATING THE INSPECTION AND WILL BE DELETED BEFORE WE DISTRIBUTE THE PAPER!!!\newline}}\smallskip} #1\end{minipage}}\\} 
\renewcommand{\cf}[1]{} 

\newcommand\hcancel[2][red]{\setbox0=\hbox{$#2$}%
\rlap{\raisebox{.25\ht0}{\textcolor{#1}{\rule{\wd0}{1pt}}}}#2}
\renewcommand\hcancel[2][red]{}

\newcommand{\vertiii}[1]{{\left\vert\kern-0.25ex\left\vert\kern-0.25ex\left\vert #1
    \right\vert\kern-0.25ex\right\vert\kern-0.25ex\right\vert}}

\numberwithin{equation}{section}
\theoremstyle{plain}
\newtheorem{definition}{Definition}
\newtheorem{theorem}{Theorem}[section]
\newtheorem{lemma}{Lemma}[section]
\newtheorem{proposition}{Proposition}[section]
\newtheorem{corollary}{Corollary}[section]

\newtheorem{remark}{Remark}[section]

\startlocaldefs
\numberwithin{equation}{section}
\theoremstyle{plain}

\endlocaldefs

\begin{document}

\begin{frontmatter}
\title{Sample path large deviations for L\'evy processes\\ and random walks with Weibull increments}
\runtitle{Sample path large deviations: Weibull case}

\begin{aug}
\author{\fnms{Mihail} \snm{Bazhba}
\thanksref{m1}
\ead[label=e1]{bazhba@cwi.nl}},
\author{\fnms{Jose} \snm{Blanchet}
\thanksref{m2}
\ead[label=e2]{jblanche@stanford.edu}},
\author{\fnms{Chang-Han} \snm{Rhee}
\thanksref{m3}
\ead[label=e3]{chang-han.rhee@northwestern.edu}}
\and
\author{\fnms{Bert} \snm{Zwart}
\thanksref{m1}
\ead[label=e4]{bert.zwart@cwi.nl}
}

\runauthor{Bazhba et al.}

\affiliation{Centrum Wiskunde \& Informatica \thanksmark{m1}, Stanford University \thanksmark{m2}, and Northwestern University \thanksmark{m3}}

\address{Science Park 123, Amsterdam, The Netherlands\\
\printead{e1,e4}\\
\phantom{E-mail:\ }
}

\address{475 Via Ortega, Stanford, California, USA\\
\printead{e2}\\
}

\address{2145 Sheridan Road, Evanston, Illinois, USA\\
\printead{e3}\\
\phantom{E-mail:\ }
}

\end{aug}

\begin{abstract}
We study sample path large deviations for L\'evy processes and random walks with heavy-tailed jump-size distributions that are of Weibull type.
Our main results include an extended form of an LDP (large deviations principle) in the $J_1$ topology, and a full LDP in the $M_1'$ topology. The rate function can be represented as the solution to a quasi-variational problem.
The sharpness and applicability of these results are illustrated by a counterexample proving the nonexistence of a full LDP in the $J_1$ topology, and by an application to a first passage problem.
\end{abstract}



\end{frontmatter}

\section{Introduction}

In this paper, we develop sample path large deviations for
 L\'evy processes and random walks, assuming the jump sizes
have a semi-exponential distribution. Specifically, let $X(t),t\geq 0,$ be a
centered L\'evy process with positive jumps and L\'evy measure $\nu$ { which has non-negative support}. Assume
that $-\log \nu [x,\infty )$ is regularly varying of index $\alpha \in (0,1)$
and define $\bar X_n= \{\bar X_n(t), t\in [0,1]\}$, with $\bar X_n(t) =
X(nt)/n, t\geq 0$. We are interested in large deviations of $\bar X_n$.

The investigation of tail estimates of the one-dimensional distributions of $%
\bar X_n$ (or random walks with heavy-tailed step size distribution) was
initiated in \cite{Nagaev69,Nagaev77}. The state of the art of such results
is well summarized in \cite{BorovkovBorovkov, DiekerDenisovShneer,
EmbrechtsKluppelbergMikosch97, FKZ}. In particular, \cite%
{DiekerDenisovShneer} describe in detail how fast $x$ needs to grow with $n$
for the asymptotic relation
\begin{equation}  \label{onebigjump}
\P( X(n) > x) = n \P(X(1)>x)(1+o(1))
\end{equation}
to hold, as $n\rightarrow\infty$. If (\ref{onebigjump}) is valid, the
so-called \emph{principle of one big jump} is said to hold. It turns out
that, if $x$ increases linearly with $n$, this principle holds if $\alpha<1/2$ and
does not hold if $\alpha>1/2$, and the asymptotic behavior of $\P( X(n) > x)
$ becomes more complicated. When studying more general functionals of $X$ it
becomes natural to consider logarithmic asymptotics, as is common in large
deviations theory, cf.\ \cite{DemboZeitouni, Gantert1998, Gantert14}.

The study of large deviations of sample paths of processes with Weibullian
increments is relatively limited. The only paper we are aware of is \cite%
{Gantert1998}, where the inverse contraction principle is applied to obtain
a large deviation principle in the $L_1$ topology. As noted in \cite%
{DuffySapozhnikov} this topology is not suitable for many
applications---ideally one would like to work with the $J_1$ topology, see equation~\eqref{eq:skorokhod-J1-metric}.

Let us now describe precisely our results. We first develop an extended LDP
(large deviations principle) in the $J_1$ topology, i.e.\ there exists a rate function $I(\cdot )$
such that
\begin{equation}
\liminf_{n\rightarrow \infty }\frac{\log \P (\bar{X}_{n}\in A)}{%
L(n)n^{\alpha }}\geq -\inf_{x\in A}I(x).
\end{equation}%
if $A$ is open, and
\begin{equation}
\limsup_{n\rightarrow \infty }\frac{\log \P (\bar{X}_{n}\in A)}{%
L(n)n^{\alpha }}\leq -\lim_{\epsilon \downarrow 0}\inf_{x\in A^{\epsilon
}}I(x).
\end{equation}%
if $A$ is closed. Here $A^{\epsilon }=\{x:d(x,A)\leq \epsilon \}$. The rate
function $I$ is given by
\begin{equation*}
I(\xi )=%
\begin{cases}
\sum_{t:\xi (t)\neq \xi (t^{-})}\big(\xi (t)-\xi (t-)\big)^{\alpha } &
\text{if }\xi \in \mathbb{D}_{\leqslant \infty }[0,1], \\
\infty  & otherwise,%
\end{cases}
\end{equation*}%
where $\mathbb{D}_{\leqslant \infty }[0,1]$ is the subspace of $\mathbb{D}[0,1]
$ consisting of non-decreasing pure jump functions vanishing at the origin
and continuous at $1$. (As usual, $\D[0,1]$ is the space of cadlag functions from $[0,1]$ to $\R$.)

The notion of an extended large deviations principle has been introduced by
\cite{borovkov2010large}. We derive this result as follows: we use a
suitable representation for the L\'evy process in terms of Poisson random
measures, allowing us to decompose the process into the contribution
generated by the $k$ largest jumps, and the remainder. The contribution
generated by the $k$ largest jumps is a step function for which we obtain
the large deviations behavior by  Bryc's inverse Varadhan lemma (see e.g.
Theorem 4.4.13 of \citealp{dembo2010large}). The remainder term is tamed by
modifying a concentration bound due to \cite{jelenkovic2003large}.

To combine both estimates we need to consider the $\epsilon$-fattening $A^{\epsilon}$ of
the set $A$, which precludes us from obtaining a full LDP. To show that our
approach cannot be improved, we construct a set $A$ that is closed in the
Skorokhod $J_1$ topology for which the large deviation upper bound does not
hold. In this sense, our extended large deviations principle can be seen as
optimal.
This is in line with the observation made for the regularly varying L\'evy processes and random walks \citep{rheeblanchetzwart2016},
for which the full LDP w.r.t.\ $J_1$ topology in a classical sense is shown to be unobtainable as well.

Following a similar proof strategy, we also derive an extended sample path LDP for random walks in $\D[0,1]$. There are however also some differences.  The distributional representation of our random walk is different from the continuous-time case.  More importantly, the resulting rate function differs at time 1, since the rescaled random walk always has a jump at time 1. We present an exact formulation of this result in Theorem~\ref{thm:random-walk}.

We derive several implications of our extended LDP that facilitate its use
in applications. First of all, if a Lipschitz functional $\phi$ of $\bar X_n$
is chosen for which the function $I_\phi(y) = \inf_{x: \phi(x)=y} I(x)$ is a
good rate function, then $\phi(X_n)$ satisfies an LDP. We illustrate this
procedure by considering an example concerning the probability of ruin for
an insurance company where large claims are reinsured.

A second implication is a sample path LDP in the $M_1^{\prime }$ topology.
We show that the rate function $I$ is good in this topology, allowing us to
conclude $\lim_{\epsilon \downarrow 0} \inf_{x\in A^\epsilon} I(x) =
\inf_{x\in A} I(x)$, if $A$ is closed in the $M_1^{\prime }$ topology. An overview of the $M_1^{\prime}$ topology can be seen in Appendix~\ref{section:appendix-for-M_1p} or in \cite{puhalskii1997functional}. 
The LDP for the $M_1'$ topology in this paper is applied in \cite{BBRZ2} to obtain tail asymptotics of the queue length in the multiple server queue with heavy tailed Weibull service times. 
In that work, we show that the most likely number of big jobs that cause a large queue length is determined by the solution of an $L^\alpha$-norm minimization problem for $\alpha \in (0,1)$.
This result essentially answers a question posed by Sergey Foss at the
Erlang Centennial conference in 2009. For earlier conjectures on this
problem we refer to \cite{Whitt2000}.

We note that both implications of our extended LDP constitute two complementary tools, and that
the two aforementioned applications are dealt with using precisely one of
these tools. In particular, the functional in the reinsurance example is not
continuous in the $M_1^{\prime }$ topology, and the most likely paths in the
queueing application are discontinuous at time 0, rendering the $J_1$ or $M_1
$ topologies useless.


Another application of our results, which will be pursued in detail elsewhere, arises in the large deviations
analysis of  Markov random walks. More precisely, when
studying $\bar{X}_{n}\left( t\right) =\sum_{k=1}^{\left\lfloor nt \right\rfloor }f\left( Y_{k}\right) /n$%
, where $Y_{k},k\geq 0$ is a geometrically ergodic Markov chain and $f\left(
\cdot \right) $ is a given measurable function. Classical large deviations
results pioneered by Donsker and Varadhan on this topic (see, for example,
\citealp{DV_3}) and the more recent treatment in \cite{KM_2}, impose certain
Lyapunov-type assumptions involving the underlying function, $f\left( \cdot
\right) $.


These assumptions are not merely technical requirements, but are needed for
a large deviations theory with a linear (in $n$) speed function (as opposed to sublinear as we obtain here).
Even in simple cases (e.g.\ \citealp{BGM, DuffyMeyn}) the case of unbounded $f\left( \cdot \right) $ can result in a sublinear large
deviations scaling of the type considered here.
For Harris chains, this can be seen by splitting $\bar{X}_{n}\left( \cdot\right) $ into cycles. Each term corresponding to a cycle can be seen as the area under  a curve generated by $f\left( Y_{\cdot}\right)$.
For linear $f$, this results in a contribution towards $\bar{X}_{n}\left( \cdot \right) $ which often is roughly proportional to the square of the cycle.
Hence, the behavior of $\bar{X}_{n}\left( 1\right) $ is close to that of a sum of i.i.d.\ Weibull-type random variables.
To summarize, the main results of this paper can be applied to significantly
extend the classical Donsker-Varadhan theory to unbounded functionals of Markov chains.

This paper is organized as follows. Section 2 introduces notation and
presents extended LDP's. These are complemented in Section 3 by LDP's of
 Lipschitz functionals, LDP's in the $M_1^{\prime }$ topology, and a
counterexample. Section 4 is considering an application to boundary crossing
probabilities with moderate jumps. Additional proofs are presented
in Section 5. The appendix develops further details about the $M_1^{\prime }$
topology that are needed in the body of the paper.
	
%
%
%
%

\section{Sample path LDPs}\label{section:sample-path-ldps}
In this section, we discuss sample path large deviations for L\'evy processes and random walks.
Before presenting the main results, we start with general background.
Let $(\mathbb S, d)$ be a metric space, and $\mathcal T$ denote the topology induced by the metric $d$.
Let $X_n$ be a sequence of $\mathbb S$-valued random variables.
Let $A^{\epsilon} \triangleq \{\xi\in \mathbb S: d(\xi, A) \leq \epsilon \}$ where $d(\xi, A) = \inf_{\zeta \in A} d(\xi, \zeta)$, and $A^{-\epsilon} \triangleq \{\xi\in \mathbb S:\, d(\xi,\zeta) \leq \epsilon \text{ implies }\zeta\in A\}$.
Let $I$ be a non-negative lower semi-continuous function on $\mathbb S$, and 	$a_n$ be a sequence of positive real numbers that tends to infinity as $n\to \infty$.
We say that $X_n$ satisfies the LDP in $(\mathbb S, \mathcal T)$ with speed $a_n$ and the rate function $I$ if
$$
-\inf_{x\in A^\circ}I(x)
\leq \liminf_{n\to\infty} \frac{\log \P(X_n\in A)}{a_n}
\leq
\limsup_{n\to\infty} \frac{\log \P(X_n\in A)}{a_n}
\leq
-\inf_{x\in A^\cl}I(x)
$$
for any measurable set $A$. Here, $A^\circ$ and $A^\cl$ are respectively the interior and the closure of the set $A$.
We say that $X_n$ satisfies the \emph{extended} LDP in $(\mathbb S, \mathcal T)$ with speed $a_n$ and the rate function $I$ if
$$
-\inf_{x\in A^\circ}I(x)
\leq \liminf_{n\to\infty} \frac{\log \P(X_n\in A)}{a_n}
\leq
\limsup_{n\to\infty} \frac{\log \P(X_n\in A)}{a_n}
\leq
-\lim_{\epsilon\to 0}\inf_{x\in A^\epsilon}I(x)
$$
for any measurable set $A$.
The next proposition provides the key framework for proving our main results.

\begin{proposition}\label{theorem:extended-ldp-from-decomposition}
Let $I$ and $I_k$, $k\geq 1$ be rate functions.
Suppose that
for each $n$, $X_n$ has a sequence of approximations $\{Y_n^k\}_{k=1,\ldots}$ such that
\begin{itemize}
\item[(i)]
For each $k$, $Y_n^k$ satisfies the extended LDP in $(\mathbb S, \mathcal T)$ with speed $a_n$ and the rate function $I_k$;

\item[(ii)] For each closed set $F$,
$$
\lim_{k\to\infty} \inf_{x\in F}I_k(x) \geq \inf_{x\in F} I(x);
$$

\item[(iii)]
For each $\delta > 0$ and each open set $G$, there exist $\epsilon>0$ and $K\ge 0$ such that $k\ge K$ implies
\begin{equation*}
\inf_{x \in G^{-\epsilon}}I_k(x) \le \inf_{x \in G}I(x)+\delta;
\end{equation*}

\item[(iv)]
For every $\epsilon>0$ it holds that
\begin{equation}
\lim_{k \rightarrow \infty}\limsup_{n \rightarrow \infty}\frac{1}{a_n}\log\P\left(d(X_n, Y_n^k)>\epsilon\right)= -\infty.
\end{equation}
\end{itemize}
Then, $X_n$ satisfies the extended LDP in $(\mathbb S, \mathcal T)$ with speed $a_n$ and the rate function $I$.
\end{proposition}
The proof of this proposition is provided in Section~\ref{section:proofs}.
We conclude this section with an immediate implication of Proposition~\ref{theorem:extended-ldp-from-decomposition}.

\begin{corollary}\label{corollary-to-prop21}
Suppose that $Y_n$ satisfies the extended LDP in $(\mathbb S, \mathcal T)$ with speed $a_n$ and the rate function $I$. If for each $\epsilon> 0$,
$$
\limsup_{n\to\infty} \frac 1 {a_n} \log \P(d(X_n, Y_n) > \epsilon) = -\infty,
$$
then $X_n$ satisfies the extended LDP in $(\mathbb S, \mathcal T)$ with speed $a_n$ and the rate function $I$.
\end{corollary}
\begin{proof}
Let $Y_n^k \triangleq Y_n$ and $I_k \triangleq I$ for $k=1,2,\ldots$. 
Then, (i) and (ii) of Proposition~\ref{theorem:extended-ldp-from-decomposition} are trivially satisfied. 
For (iii), we note that by the definition of $G^{-\epsilon}$, for each $\delta>0$ and $G$ an open set, there exists $\epsilon>0$ such that 
$$
\inf_{x\in G^{-\epsilon}} I(x) \leq \inf_{x \in G} I(x) + \delta,
$$
and hence, (iii) is satisfied for $I_k = I$.
Since (iv) is also satisfied by the assumption, all the conditions of Proposition~\ref{theorem:extended-ldp-from-decomposition} is satisfied and the conclusion of the corollary follows.
\end{proof}

\subsection{Extended sample path LDP for L\'evy processes}\label{subsection:extended-ldp-1-d-levy-processes}
Throughout the rest of this paper we assume that
\begin{itemize} [leftmargin=1.5cm]
\item[A1.] $X$ is a real-valued L\'evy process with  L\'evy measure $\nu$ which has non-negative support satisfying $\nu[x,\infty) = \exp(-L(x)x^\alpha)$ where $\alpha \in (0,1)$ and $L(\cdot)$ is  slowly varying at infinity.
\item[A2.] The mapping $x \mapsto L(x)x^{\alpha-1}$ is non-increasing for sufficiently large $x$.
\end{itemize}

Let $\bar X_n(t)$ denote the centered and scaled process:
$$
\bar X_n(t) \triangleq \frac 1n X(nt) - t \E X(1).
$$
Let $\mathbb D[0,1]$ denote the Skorokhod space---space of c\`adl\`ag functions from $[0,1]$ to $\R$---and $\mathcal T_{J_1}$ denote the $J_1$ Skorokhod topology on $\mathbb D[0,1]$. 
That is, $\mathbb{D}[0,1]$ is metrized by the usual Skorokhod metric
\begin{equation}\label{eq:skorokhod-J1-metric}	 
	 d_{J_1}(\xi,\zeta) \triangleq \inf_{\lambda \in \Lambda}\left\{\max\{\|\xi(t)-\zeta(\lambda(t))\|_{\infty}, \|\lambda(t)-e(t)\|_{\infty}\}\right\},
	 \end{equation}
	  where $\xi, \zeta \in \mathbb{D}[0,1]$, $\lambda$ is a non-decreasing  homeomorphism  of $[0,1]$ onto itself, $\Lambda$ is the set of such homeomorphisms, $e(t)=t$ is the identity map, and $\|\xi\|_{\infty}=\sup_{t \in [0,1]}|\xi(t)|$ is the supremum norm.
We say that $\xi\in \D[0,1]$ is a pure jump function if $\xi = \sum_{i=1}^\infty x_i\I_{[u_i,1]}$ for some $x_i$'s and $u_i$'s such that $x_i\in \R$ and $u_i\in {(0,1)}$ for each $i$ and $u_i$'s are all distinct.
Let $\mathbb D_{\leqslant \infty}[0,1]$ denote the subspace of $\mathbb D[0,1]$ consisting of non-decreasing pure jump functions vanishing at the origin and continuous at $1$.
For the rest of the paper, if there is no confusion regarding the domain of a function space, we will omit the domain and simply write, for example, $\mathbb D_{\leqslant \infty}$ instead of $\mathbb D_{\leqslant \infty}[0,1]$.
The next theorem is the main result of this paper.

\begin{theorem}\label{theorem:1d-extended-ldp}
$\bar X_n$ satisfies the extended large deviation principle in $(\mathbb D, \mathcal T_{J_1})$ with speed $L(n)n^\alpha$ and rate function
\begin{equation}\label{eq:the-rate-function-1d}
I(\xi)=
\begin{cases}
\sum_{t: \xi(t) \neq \xi(t-)} \big(\xi(t)-\xi(t-)\big)^\alpha & \text{if }  \xi \in \mathbb D_{\leqslant \infty},\\
\infty, & otherwise.
\end{cases}
\end{equation}
That is, for any measurable $A$,
\begin{equation}\label{ineq:1d-extended-ldp}
-\inf_{\xi \in A^\circ} I(\xi)
\leq
\liminf_{n\to\infty} \frac{\log \P(\bar X_n \in A)}{L(n)n^{\alpha}}
\leq
\limsup_{n\to\infty} \frac{\log \P(\bar X_n \in A)}{L(n)n^{\alpha}}
\leq
-\lim_{\epsilon\to 0}\inf_{\xi \in A^\epsilon} I(\xi),
\end{equation}
where $A^{\epsilon} \triangleq \{\xi\in \mathbb D: d_{J_1}(\xi, \zeta) \leq \epsilon \text{ for some } \zeta\in A \}$.
\end{theorem}

\begin{remark}
Note that it is straightforward to extend Theorem~\ref{theorem:1d-extended-ldp} to spectrally two-sided L\'evy processes. For instance, suppose that the L\'evy measure $\nu$ of the process $X$ has  Weibull tail $\nu[x,\infty)=exp(-L(x)x^{\alpha})$ where $\alpha \in (0,1)$, $L(x)x^{\alpha-1}$ satisfies assumption $A2$, and  $\nu(-\infty,x]$ is light-tailed. 
We can consider a decomposition of $\bar X_n$ into a spectrally positive part $\bar Y_n$ and a spectrally negative part $\bar X_n -\bar Y_n$. 
Then, $\bar Y_n$ satisfies the LDP in Theorem~\ref{theorem:1d-extended-ldp}. 
On the other hand, since $\bar X_n - \bar Y_n$ is light-tailed, it satisfies a sample path LDP with linear speed---see \cite{mogulskii1993large}. 
This implies that the condition Corollary~\ref{corollary-to-prop21} is satisfied, which, allows one to apply Corollary~\ref{corollary-to-prop21} with $Y_n$, and we conclude that $\bar X_n$ satisfies the same LDP as the one in Theorem~\ref{theorem:1d-extended-ldp}.
\end{remark}

Recall that $X_n(\cdot)\triangleq X(n\cdot)$ has It\^{o} representation
\begin{equation}\label{eq:ito_rep_onesided}
X_n(s) = nsa + B(ns) + \int_{x<1} x[\hat N([0,ns]\times dx) - ns\nu(dx)] + \int_{x\geq 1} x\hat N([0,ns]\times dx),
\end{equation}
with $a$ a drift parameter, $B$ a Brownian motion, and $\hat N$ a Poisson random measure with mean measure Leb$\times\nu$ on $[0,n]\times(0,\infty)$; Leb here denotes the Lebesgue measure. All terms in \eqref{eq:ito_rep_onesided} are independent.
We will see that the large deviation behavior is dominated by the last term of \eqref{eq:ito_rep_onesided}.
It turns out to be convenient to consider the following distributional representation of the centered and scaled version of the last term:
\begin{align*}
\bar Y_n(\cdot)
\triangleq
\frac1n\sum_{l=1}^{N(n\cdot)}(Z_l - \E Z)
\stackrel{\mathcal D}{=}
\frac1n\int_{x\geq 1} x \hat N([0,n\cdot]\times dx) - \frac1n \,(\E Z)\,{\hat N([0,n\cdot]\times[1,\infty))} 
\end{align*}
where $N(t) \triangleq \hat N([0,t]\times[1,\infty))$ is a Poisson process with arrival rate ${\nu_1 \triangleq \nu[1,\infty)}$, and $Z_i$'s are i.i.d.\ copies of $Z$ such that $\P(Z\geq x) = \nu[x\vee 1, \infty)/\nu_1$, independent of $N$.
To facilitate the proof of Theorem~\ref{theorem:1d-extended-ldp}, we consider a further decomposition of $\bar Y_n$ into two pieces, one of which consists of the big increments, and the other one keeps the residual fluctuations.
To be more specific, we introduce an extra notation for the rank of the increments.
Given $N(n)$, define $\mathbf S_{N(n)}$ to be the set of all permutations of $\{1,...,N(n)\}$.
Let $R_n:\{1,\ldots,N(n)\}\to\{1,\ldots,N(n)\}$ be a random permutation of $\{1,\ldots,N(n)\}$ sampled uniformly from $\Sigma_n\triangleq\{\sigma \in \mathbf S_{N(n)}: Z_{\sigma^{-1}(1)} \ge \cdots \ge Z_{\sigma^{-1}(N(n))} \}$.
In words, $R_n(i)$ is the rank of $Z_i$ among $\{Z_1, \ldots, Z_{N(n)}\}$ when sorted in decreasing order with the ties broken uniformly randomly.
Now, we see that
$$
\bar Y_n{ (t)}
=
\underbrace{\frac1n \sum_{i=1}^{N(nt)} Z_i \one{R_n(i) \leq k}}_{\triangleq \bar J_n^k{(t)}}
+
\underbrace{\frac1n \sum_{i=1}^{N(nt)} (Z_i \one{R_n(i)>k} - \E Z)}_{\triangleq \bar H_n^k{(t)}}.
$$
The proof of Theorem~\ref{theorem:1d-extended-ldp} is straightforward once the following technical lemmas are in our hands; their proofs are provided in Section~\ref{section:proofs}.
Let $\mathbb D_{\leqslant k}$ denote the subspace of $\mathbb D_{\leqslant\infty}$ consisting of paths that have less than or equal to $k$ discontinuities.
\begin{lemma}\label{lemma:equivalence-J}
For each fixed $k$, $\bar J_n^k$ satisfies the LDP in $(\mathbb D, \mathcal T_{J_1})$ with speed $L(n)n^\alpha$ and the rate function
\begin{equation}\label{eq:rate-function-for-Ik}
I_k(\xi) =
\begin{cases}
\sum_{t\in[0,1]}\left(\xi(t)-\xi(t-)\right)^\alpha
&
\text{if \ $\xi\in \mathbb D_{\leqslant k}$},
\\
\infty
&
otherwise.
\end{cases}
\end{equation}

\end{lemma}
Recall that $A^{-\epsilon} \triangleq \{\xi\in \mathbb D:\, d_{J_1}(\xi,\zeta) \leq \epsilon \text{ implies }\zeta\in A\}$.
\begin{lemma}\label{lemma:Jk-approximates-J}
For each $\delta > 0$ and each open set $G$, there exist $\epsilon>0$ and $K\ge 0$ such that for any $k\ge K$
\begin{equation}\label{inf-G-epsilon-delta}
\inf_{\xi \in G^{-\epsilon}}I_k(\xi) \le \inf_{\xi \in G}I(\xi)+\delta.
\end{equation}
\end{lemma}

Let $B_{J_1}(\xi,\epsilon)$ be the open ball w.r.t.\ the $J_1$ Skorokhod metric centered at $\xi$ with radius $\epsilon$ and $B_\epsilon \triangleq B_{J_1}(0,\epsilon)$.
\begin{lemma}\label{lemma:negligence-K}
For every $\epsilon>0$ it holds that
\begin{equation}
\lim_{k \rightarrow \infty}\limsup_{n \rightarrow \infty}\frac{1}{L(n)n^\alpha}\log\P\left(\|\bar H_n^k\|_\infty >\epsilon\right)= -\infty.
\end{equation}
\end{lemma}

\begin{proof}[Proof of Theorem~\ref{theorem:1d-extended-ldp}]
For this proof, we use the following representation of $\bar X_n$:
\begin{equation}\label{eq:Poisson-Jump-representation}
\bar X_n \stackrel{\mathcal D}{=} \bar Y_n +   \bar R_n = \bar J_n^k + \bar H_n^k + \bar R_n,
\end{equation}
where $\bar R_n(s) = \frac1n B(ns) + \frac1n\int_{|x|\leq 1} x[N([0,ns]\times dx) - ns\nu(dx)] + \frac{1}n(\E Z) \hat N([0,n\cdot]\times[1,\infty))  - \nu_1t$.
Next, we verify the conditions of Proposition~\ref{theorem:extended-ldp-from-decomposition}.
Lemma~\ref{theorem:lower-semi-continuity} confirms that $I$ is lower semi-continuous.
Lemma~\ref{lemma:equivalence-J} verifies (i).
To see that (ii) is satisfied, note that $I_k(\xi) \geq I(\xi)$ for any $\xi \in \D$.
Lemma~\ref{lemma:Jk-approximates-J} verifies (iii).
Since $d_{J_1}(\bar X_n, \bar J_n^k) \leq \|\bar H_n^k\|_\infty + \|\bar R_n\|_\infty$, Lemma~\ref{lemma:negligence-K} and $\limsup_{n\to\infty} \frac{1}{L(n)n^\alpha}\log \P(\|\bar R_n \|_\infty >\epsilon) = -\infty$ implies (iv). Note that $\bar R_n$ is a L\'evy process whose moment generating function is finite everywhere, and hence, the LDP upper bound in Theorem~2.5 of \cite{mogulskii1993large} applies (with linear speed) to $\P(d_{J_1}(0,\bar R_n) > \epsilon)$. 
This, in turn, implies that $\limsup_{n\to\infty} \frac{1}{L(n)n^\alpha}\log \P(\|\bar R_n \|_\infty >\epsilon) = -\infty$.
Now, the conclusion of the theorem follows from Proposition~\ref{theorem:extended-ldp-from-decomposition}.
\end{proof}

\subsection{Extended LDP for random walks}\label{subsec:extended-ldp-for-random-walks}
Let $S_n \triangleq Z_1+\cdots+Z_n$ where $Z_i$'s are non-negative random variables.
Consider the centered and scaled process $\bar S_n = \{\bar S_n(t), t\in [0,1]\}$ where 
$\bar S_n(t) \triangleq  \frac1n \sum_{i=1}^{[nt]} (Z_i- \E Z),\, t\in[0,1].$
We assume that $\P(Z \geq x) = \exp(-L(x)x^\alpha)$  where $\alpha \in (0,1)$ and $L(\cdot)$ is a slowly varying function. 
We also assume A2 is in force as in Section~\ref{subsection:extended-ldp-1-d-levy-processes}.
The goal of this section is to prove an extended LDP for $\bar S_n$.
Let $\tilde \D_{\leqslant \infty}$ denote the subspace of $\D$ consisting of non-decreasing pure jump functions vanishing at the origin (not necessarily continuous at 1, though). 
Let $\tilde \D_{\leqslant k}$ denote the subspace of $\tilde \D_{\leqslant \infty}$ consisting of paths that have less than or equal to $k$ discontinuities. 
Define  $\tilde I$  as follows:
\begin{equation*}
\tilde I(\xi)=
\begin{cases}
\sum_{t: \xi(t) \neq \xi(t-)} \big(\xi(t)-\xi(t-)\big)^\alpha & \text{if }  \xi \in \tilde \D_{\leqslant \infty},\\
\infty, & otherwise.
\end{cases}
\end{equation*}

\begin{theorem}\label{thm:random-walk}
$\bar S_n$ satisfies the extended large deviation principle in $(\mathbb D, \mathcal T_{J_1})$ with speed $L(n)n^\alpha$ and  rate function $\tilde I$.
\end{theorem}

Similarly to the case of L\'evy processes, the proof of Theorem~\ref{thm:random-walk} is facilitated by Lemma~\ref{lemma:rw-lem1}, Lemma~\ref{lemma:rw-lem3}, and Lemma~\ref{lemma:rw-lem2}; their proofs are deferred to Section~\ref{sec:proof-for-rw-lemmas}.
Let $\tilde Q^\gets(x) = \inf\{y\geq 0: \P(Z \geq y) < x\}$. 
Set
$$
\tilde J_n^k(t) \triangleq \frac1n \sum_{i=1}^k \tilde Q^\gets (V_{(i)}) \I_{[U_i,1]}(t)+\frac1n Z \I_{\{1\}}(t) 
$$ 
and
$$
\tilde H_n^k(t) \triangleq \frac1n \sum_{i=k+1}^{n-1} \tilde Q^\gets (V_{(i)}) \I_{[U_i,1]}(t)
 -\frac 1n \E Z \sum_{i=1}^{n-1}\I_{[U_i,1]}(t) - \frac1n \E Z \I_{\{1\}}(t)
 $$
where $V_{(1)}, V_{(2)},\ldots, V_{(n-1)}$ are the order statistics (in ascending order) of $V_1,V_2,\ldots, V_{n-1}$, which are i.i.d.\ Uniform$[0,1]$ and independent of $Z$.
(In the proof of Theorem~\ref{thm:random-walk}, we will show that $\tilde S_n \triangleq \tilde J_n^k + \tilde H_n^k$ satisfies the extended LDP with speed $L(n)n^\alpha$  and is \chck exponentially equivalent to $\bar S_n$ so that Corollary~\ref{corollary-to-prop21} applies, and hence, in turn, $\bar S_n$ satisfies the same extended LDP.)
Let $\tilde I_k$ be defined as follows:
\begin{equation}\label{eq:rate-function-for-Ik-tilde}
\tilde I_k(\xi) =
\begin{cases}
\sum_{t: \xi(t) \neq  \xi(t-)}\left(\xi(t)-\xi(t-)\right)^\alpha
&
\text{if \ $\xi\in \tilde\D_{\leqslant k}$},
\\
\infty
&
otherwise.
\end{cases}
\end{equation}
\begin{lemma}\label{lemma:rw-lem1}
For each fixed $k$, $\tilde J_n^k$ satisfies the LDP in $(\mathbb D, \mathcal T_{J_1})$ with speed $L(n)n^\alpha$ and the rate function $\tilde I_k$.
\end{lemma}
\begin{lemma}\label{lemma:rw-lem3}
For each $\delta > 0$ and each open set $G$, there exist $\epsilon>0$ and $K\ge 0$ such that for any $k\ge K$
\begin{equation*}
\inf_{\xi \in G^{-\epsilon}}\tilde I_k(\xi) \le \inf_{\xi \in G}\tilde I(\xi)+\delta.
\end{equation*}
\end{lemma}
\begin{lemma}\label{lemma:rw-lem2}
For every $\epsilon>0$ it holds that
$$\lim_{k\to\infty}\limsup_{n\to\infty} \frac{1}{L(n)n^\alpha} \log\P(\| \tilde H_n^k\|_\infty > \epsilon) = -\infty.$$ 
\end{lemma}

With these lemmas in our hands, we are ready to prove Theorem~\ref{thm:random-walk}.
\begin{proof}[Proof of Theorem~\ref{thm:random-walk}]
Let $\tilde R_i \triangleq |\{j\in \mathbb N: U_j \leq U_i, \ 1\leq j \leq n-1\} |$. Then, we claim that
$$
\bar S_n \stackrel{\mathcal D}{=} \frac 1n \sum_{i=1}^{n-1}\big(\tilde Q^\gets(V_{(i)}) - \E Z\big) \I_{[\tilde R_i/n,1]} + \frac 1n (Z-\E Z) \I_{\{1\}}.
$$
To see why this distributional equality holds, note that $\{\tilde R_1,\ldots, \tilde R_{n-1}\}$ is a uniformly random permutation of $\{1,
\ldots, n-1\}$, and $\{\tilde Q^\gets(V_{(1)}),\ldots, \tilde Q^\gets(V_{(n-1)})\}$ has the same distribution as the order statistics (in descending order) of $Z_1,\ldots, Z_{n-1}$ since $\tilde Q^\gets(V_i)$ has the same distribution as $Z$ for each $i$.
Now, we move on to showing that $\bar S_n$ is close to $\tilde S_n$---i.e., $\P(d_{J_1}(\tilde S_n, \bar S_n) > \epsilon)$ is asymptotically negligible. 
Recall that
$$
\tilde S_n = \frac1n \sum_{i=1}^{n-1} (\tilde Q^\gets(V_{(i)})- \E Z)\I_{[U_i,1]} + \frac1n(Z-\E Z) \I_{\{1\}}.
$$
First observe that
$\tilde R_i$ is the rank of $U_i$ among $U_1,\ldots, U_{n-1}$,
and hence, $\tilde R_i$\textsuperscript{th} earliest jump of $\bar S_n$ and $\tilde S_n$ are both $\tilde Q^\gets(V_{(i)})$.
Therefore, the jumps associated with $\tilde Q^\gets(V_{(1)}),\ldots, \tilde Q^\gets(V_{(n-1)}), Z$ are arranged in the same order for $\bar S_n$ and $\tilde S_n$ with probability 1.
Moreover, the jump times of $\bar S_n$ and $\tilde S_n$ are $\frac1n, \frac2n, \ldots, \frac{n-1}n, \frac nn$ and $U_{(1)}, U_{(2)}, \ldots, U_{(n-1)}, 1$, respectively.
Since $0<U_{(1)}<\cdots<U_{(n-1)}<1$ with probability 1, the piecewise linear time change $\lambda:[0,1]\to[0,1]$ defined by the linear interpolation of $\lambda(0) = 0$, $\lambda(1) = 1$, and $\lambda(i/n) = U_{(i)}$ for $i=1,\ldots,n-1$ is a homeomorphism with probability 1. Therefore, the $J_1$ distance between $\tilde S_n$ and $\bar S_n$ is bounded by 
$$
 \sup_{1\leq i\leq n-1} |i/n -U_{(i)}|
$$
with probability 1.
%
%
%
On the other hand, by the sample path LDP for the order statistics of uniform random variables \citep{duffy2011sample}, \chck 
 $\P(\sup_{1\leq i \leq n-1} |i/n - U_{(i)}|>\epsilon)$ decays at an exponential rate, and
hence, 
$$\limsup_{n\to\infty} \frac{1}{L(n)n^\alpha} \log \P(d_{J_1}(\tilde S_n,\bar S_n)>\epsilon) \leq 
\limsup_{n\to\infty} \frac{1}{L(n)n^\alpha} \log \P(\sup_{1\leq i \leq n-1} |i/n - U_{(i)}|>\epsilon)= -\infty.$$
In view of Corollary~\ref{corollary-to-prop21}, the proof is done if we show that 
 $\tilde S_n$ satisfies the extended LDP with speed $L(n)n^\alpha$ and the rate function $\tilde I$.
To do so, we apply Proposition~\ref{theorem:extended-ldp-from-decomposition}.
Note that Lemma~\ref{lemma:rw-lem1} verifies condition (i) of Proposition~\ref{theorem:extended-ldp-from-decomposition}; (ii) is trivially satisfied since $\tilde I_k \geq \tilde I$; Lemma~\ref{lemma:rw-lem3} verifies (iii); Lemma~\ref{lemma:rw-lem2} verifies (iv). 
Therefore Proposition~\ref{theorem:extended-ldp-from-decomposition}  applies to $\tilde J_n^k + \tilde H_n^k$, and the proof of Theorem~\ref{thm:random-walk} is complete. 
\end{proof}

%
%
\subsection{Multi-dimensional processes}\label{subsec:multidimensional-processes}
Let $X^{(1)},\ldots,X^{(d)}$ be independent processes satisfying assumptions $A1$ and $A2$.
Let $\bar X_n^{(i)}(t)$ denote the centered and scaled processes:
$$
\bar X_n^{(i)}(t) \triangleq \frac 1n X^{(i)}(nt) - t \E X^{(i)}(1).
$$
The next theorem establishes the extended LDP for $(\bar X_n^{(1)},\ldots,\bar X_n^{(d)})$.

\begin{theorem}\label{theorem:multi-d-levy-processes}
$\left(\bar{X}^{(1)}_n,\bar{X}^{(2)}_n,\dots,\bar{X}^{(d)}_n\right)$ satisfies the extended LDP in $\big(\prod_{i=1}^{d}\mathbbm{D}\left([0,1],\R_+\right), \prod_{i=1}^d \mathcal T_{J_1}\big)$ with speed $L(n)n^\alpha$ and the rate function
\begin{equation}\label{eq:rate-function-Id}
I^d(\xi_{1},...,\xi_{d})=
\begin{cases}
\sum_{j=1}^{d}\sum_{t\in[0,1]}\left( \xi_{j}(t)- \xi_{j}(t-)\right)^\alpha  & \text{if } \xi_j \in \D_{\leqslant\infty}\ \text{ for each }j=1,\ldots,d,
 \\
\infty, & otherwise.
\end{cases}
\end{equation}
\end{theorem}

For each $i$, we consider the same distributional decomposition of $\bar X_n^{(i)}$ as in Section~\ref{subsection:extended-ldp-1-d-levy-processes}:
\begin{equation*}
\bar X_n^{(i)} \stackrel{\mathcal D}{=}  \bar J_n^{k\,(i)} + \bar H_n^{k\,(i)} + \bar R_n^{(i)}.
\end{equation*}
The proof of the theorem is immediate as in the one dimensional case, from Proposition~\ref{theorem:extended-ldp-from-decomposition}, Lemma~\ref{lemma:negligence-K}, and the following lemmas that parallel Lemma~\ref{lemma:equivalence-J} and Lemma~\ref{lemma:Jk-approximates-J}.

\begin{lemma}\label{lemma:multi-d-equivalence-J}
$\left(\bar J_n^{k\,(1)},\ldots,\bar J_n^{k\,(d)}\right)$ satisfy the LDP in $\big(\prod_{i=1}^d\mathbb D, \prod_{i=1}^d\mathcal T_{J_1}\big)$ with speed $L(n)n^\alpha$ and the rate function $I_k^d:\prod_{i=1}^d\mathbb D\to [0,\infty]$
\begin{equation}
I_k^d(\xi_1,\ldots,\xi_d)
\triangleq
\begin{cases}
\sum_{i=1}^d\sum_{t\in[0,1]}\left(\xi_i(t)-\xi_i(t-)\right)^\alpha
&
\text{if \ $\xi_i\in \mathbb D_{\leqslant k}$ for $i=1,\ldots,d$},
\\
\infty
&
otherwise.
\end{cases}
\end{equation}

\end{lemma}

\begin{lemma}\label{eq:extspldplbt}
For each $\delta > 0$ and each open set $G$, there exists $\epsilon >0$ and $K \ge 0$ such that for any $k \ge K$
\begin{equation}
\inf_{(\xi_{1},\ldots,\xi_{d}) \in G^{-\epsilon}}I_{k}^d(\xi_1,\ldots,\xi_d) \le \inf_{(\xi_1,\ldots,\xi_d) \in G} I^d(\xi_1,\ldots,\xi_d)+\delta.
\end{equation}
\end{lemma}

We conclude this section with the extended LDP for multidimensional random walks.
Let $S_n^{(i)} = Z_1^{(i)}+\cdots+Z_n^{(i)}$ be a random walk with non-negative increments for each $i=1,\ldots,d$.
Consider $\bar S_n^{(i)} = \{\bar S_n^{(i)}(t), t\in [0,1]\}$ where $\bar S_n^{(i)}(t) = \frac 1n\sum_{j=1}^{[nt]} \big(Z_j^{(i)} - \E Z^{(i)}\big)$.
We assume that $\P(Z_j^{(i)} \geq x) = \exp(-L(x)x^\alpha)$ where $\alpha \in (0,1)$ and $L(\cdot)$ is a slowly varying function, and $A2$ is in force.
The following theorem can be obtained by ``lifting'' Lemma~\ref{lemma:rw-lem1} and Lemma~\ref{lemma:rw-lem3}
to the multi-dimensional context---in the same way Lemma~\ref{lemma:equivalence-J} and Lemma~\ref{lemma:Jk-approximates-J}
were lifted to 
the multi-dimensional counterparts in the proof of Theorem~\ref{theorem:multi-d-levy-processes}---and then applying Proposition~\ref{theorem:extended-ldp-from-decomposition}.
Let 
\begin{equation}\label{eq:rate-function-tilde-Id}
\tilde I^d(\xi_{1},...,\xi_{d})=
\begin{cases}
\sum_{j=1}^{d}\sum_{t\in[0,1]}\left( \xi_{j}(t)- \xi_{j}(t-)\right)^\alpha  & \text{if } \xi_j \in \tilde \D_{\leqslant \infty} \text{ for each }j=1,\ldots,d,
 \\
\infty & otherwise.
\end{cases}
\end{equation}
\begin{theorem}\label{thm:d-dim-random-walk}
$\left(\bar S_n^{(1)},\bar S_n^{(2)},\ldots,\bar S_n^{(d)}\right)$ satisfies the extended LDP in $(\prod_{i=1}^d\mathbb D, \prod_{i=1}^d\mathcal T_{J_1})$ with speed $L(n)n^\alpha$ and the rate function $\tilde I^d$.
\end{theorem}

\begin{remark}
Note that Theorem~\ref{theorem:multi-d-levy-processes} and Theorem~\ref{thm:d-dim-random-walk} can be extended to heterogeneous processes.
For example, if the L\'evy measure $\nu^{(i)}$ of the process $X^{(i)}$ has Weibull tail distribution $\nu^{(i)}[x,\infty) = \exp(-c_i L(x) x^\alpha)$ where $c_i \in (0,\infty)$ for each $i \leq d_0  \leq d$, and all the other processes have lighter tails---i.e., $L(x)x^{\alpha} = o(L_i(x)x^{\alpha_i})$ for $i>d_0$---then it is straightforward to check that $(\bar X_n^{(1)},\ldots, \bar X_n^{(d)})$ satisfies the extended LDP with the rate function
\begin{equation*}
I^d(\xi_{1},...,\xi_{d})=
\begin{cases}
\sum_{j=1}^{d_0}c_j\sum_{t\in[0,1]}\left( \xi_{j}(t)- \xi_{_j}(t-)\right)^\alpha  
& \text{if } \xi_j \in \D_{\leqslant\infty}[0,1] \text{ for }j=1,\ldots,d_0
\\
&\hfill\text{and } \xi_j \equiv 0 \text{ for }j>d_0,
 \\
\infty & otherwise.
\end{cases}
\end{equation*}
Similarly, $(\bar S_n^{(1)},\ldots,\bar S_n^{(d)})$ satisfies the extended LDP with the rate function $\tilde I^d$ defined by replacing $\D_{\leqslant \infty}$ with $\tilde \D_{\leqslant \infty}$ in the definition of $I^d$ above under the corresponding conditions on the tail distribution of $Z_1^{(i)}$'s.
\end{remark}

%
%
%
%
\section{Implications and further discussions}

%
%
\subsection{LDP for Lipschitz functions of L\'evy processes}
Let $\bar{\mathbf  X}_n$ denote the scaled L\'evy processes $(\bar X_n^{(1)},\ldots,\bar X_n^{(d)})$, and $\bar {\mathbf S}_n$ denote the scaled random walks $(\bar S_n^{(1)},\ldots,\bar S_n^{(d)})$ as defined in Section~\ref{section:sample-path-ldps}. 
Recall also the rate functions $I^d$ defined in (\ref{eq:rate-function-Id}) and $\tilde I^d$ defined in (\ref{eq:rate-function-tilde-Id}).
\begin{corollary}\label{theorem:contraction-principle}
Let $(\mathbb S, d)$ be a metric space and $\phi: \prod_{i=1}^d \D \to \mathbb S$ be a Lipschitz continuous mapping w.r.t.\ the $J_1$ Skorokhod metric.
Set
$$I'(x) \triangleq \inf_{\phi(\xi) = x}I^d(\xi) \qquad \text{and} \qquad \tilde I'(x) \triangleq \inf_{\phi(\xi) = x}\tilde I^d(\xi) $$
and suppose that $I'$ (or $\tilde I'$) is a good rate function---i.e., $\Psi_{I'}(a)\triangleq \{x\in \mathbb S: I'(s) \leq a \}$ (or $ \Psi_{\tilde I'}(a)\triangleq \{x\in \mathbb S: \tilde I'(s) \leq a \}$) is compact for each $a\in [0,\infty)$.
Then, $\phi\left(\bar {\mathbf X}_n \right)$ (or $\phi\left(\bar {\mathbf S}_n\right)$) satisfies the large deviation principle in $(\mathbb S, d)$ with speed $L(n)n^\alpha$ and the good rate function $I'$ (or $\tilde I'$).
\end{corollary}
\begin{proof}
Since the argument for $\phi(\bar {\mathbf S}_n)$ is very similar, we only prove the LDP for $\phi(\bar {\mathbf X}_n)$.
We start with the upper bound.
Suppose that the Lipschitz constant of $\phi$ is $\|\phi\|_\lip$.
Note that since the $J_1$ distance is dominated by the supremum distance, $\|\bar {\mathbf H}_n^k\|_\infty\leq\epsilon$ and $\|\bar {\mathbf R}_n\|_\infty \leq \epsilon$ implies $d_{J_1}\big(\phi(\bar {\mathbf J}_n^{k}),\,\phi(\bar {\mathbf X}_n)\big) \leq 2\epsilon \|\phi\|_\lip$, where $\bar {\mathbf J}_n^k \triangleq  (\bar J_n^{k\,(1)}, \ldots, \bar J_n^{k\,(d)})$, $\bar {\mathbf H}_n^k \triangleq  (\bar H_n^{k\,(1)}, \ldots, \bar H_n^{k\,(d)})$, and $\bar {\mathbf R}_n \triangleq  (\bar R_n^{(1)}, \ldots, \bar R_n^{(d)})$.
Therefore, for any closed set $F$,
\begin{align}
&\P\left(\phi(\bar {\mathbf X}_n)\in F\right)
\nonumber
\\
&
\leq
\P\left(\phi\left( \bar{\mathbf X}_n\right)\in F,d_{J_1}\big(\phi(\bar {\mathbf J}_n^{k}),\,\phi(\bar {\mathbf X}_n)\big) \leq 2\epsilon \|\phi\|_\lip\right)+\P\left(d_{J_1}\big(\phi(\bar {\mathbf J}_n^{k}),\,\phi(\bar {\mathbf X}_n)\big) > 2\epsilon \|\phi\|_\lip\right)\nonumber
\\
&		
\leq
\P\left(\phi\left( \bar{\mathbf J}^{k}_n\right)\in F^{2\epsilon\|\phi\|_\lip}\right)
+
\P\left(d_{J_1}\big(\phi(\bar {\mathbf J}_n^{k}),\,\phi(\bar {\mathbf X}_n)\big) > 2\epsilon \|\phi\|_\lip\right)\nonumber
\\
&		
\leq
\P\left( \bar{\mathbf J}^{k}_n\in \phi^{-1}\big(F^{2\epsilon\|\phi\|_\lip}\big)\right)
+\P\left(\|\bar {\mathbf H}_n^{k}\|_\infty > \epsilon\right)
+\P\left(\|\bar {\mathbf R}_n\|_\infty > \epsilon\right)
.\nonumber
\end{align}
Since $\P\left(\|\bar {\mathbf R}_n\|_\infty > \epsilon\right)$ decays at an exponential rate and $\P\left(\|\bar {\mathbf H}_n^k\|_\infty > \epsilon\right) \leq \sum_{i=1}^d\P\big(\|\bar { H}_n^{k\,(i)}\|_\infty > \epsilon/d \big)$, we get the following bound by applying the principle of the maximum term and Theorem~\ref{theorem:multi-d-levy-processes}:
\begin{align*}
&
\limsup_{n \rightarrow \infty}\frac{1}{L(n)n^\alpha}\log\P\left(\phi\left(\bar {\mathbf X}_n \right)\in F\right)
\\
&
\le
\limsup_{n \rightarrow \infty}\frac{1}{L(n)n^\alpha}\log \Big\{
\P\left( \bar{\mathbf J}^{k}_n\in \phi^{-1}\big(F^{2\epsilon\|\phi\|_\lip}\big)\right)
+\P\left(\|\bar {\mathbf H}_n^{k}\|_\infty > \epsilon\right)
+\P\left(\|\bar {\mathbf R}_n\|_\infty > \epsilon\right)
\Big\} \nonumber
\\
&
=
\max\left\{
\limsup_{n \rightarrow \infty} \frac{\log  \P\left( \bar{\mathbf J}^{k}_n\in \phi^{-1}\big(F^{2\epsilon\|\phi\|_\lip}\big)\right)}{L(n)n^\alpha},\
\limsup_{n\rightarrow \infty} \frac{ \log \P\left(\|\bar {\mathbf H}_n^{k}\|_\infty > \epsilon\right)}{L(n)n^\alpha}
\right\}
\nonumber
\\
&
\le
\max\Bigg\{-\inf_{(\xi_1,\ldots,\xi_d) \in \phi^{-1}\big(F^{2\epsilon\|\phi\|_\lip}\big)}I_k^d(\xi_1,\ldots,\xi_d),\,\max_{i=1,\ldots,d}\limsup_{n \rightarrow \infty}\frac{\log \P\big(\|\bar{H}^{k\,(i)}_n\|_\infty>\epsilon/d \big)}{L(n)n^\alpha} \Bigg\}\nonumber
\\
&
\le
\max\Bigg\{-\inf_{(\xi_1,\ldots,\xi_d) \in \phi^{-1}\big(F^{2\epsilon\|\phi\|_\lip}\big)}I^d(\xi_1,\ldots,\xi_d),\,\limsup_{n \rightarrow \infty}\frac{ \log \P\big(\|\bar{H}^{k\,(1)}_n\|_\infty>\epsilon/d \big) }{L(n)n^\alpha}\Bigg\}.\nonumber
\end{align*}
 From Lemma~\ref{lemma:negligence-K}, we can take $k\rightarrow \infty$ to get
\begin{equation}\label{ineq:phi-Xbar-almost-upper-bound}
\limsup_{n \rightarrow \infty}\frac{\log\P\left(\phi\left(\bar {\mathbf X}_n \right)\in F\right) }{L(n)n^\alpha}
\leq
-\inf_{(\xi_1,\ldots,\xi_d) \in \phi^{-1}\left(F^{2\epsilon\|\phi\|_\lip}\right)}I^d(\xi_1,\ldots,\xi_d)
=
-\inf_{x\in F^{2\epsilon\|\phi\|_\lip}} I'(x)
\end{equation}		
 From Lemma 4.1.6 of \cite{dembo2010large},
$\lim_{\epsilon\rightarrow 0}\inf_{x  \in F^{\epsilon\|\phi\|_\lip}}I'(x)=\inf_{x \in F}I'(x)$.
Letting $\epsilon\to 0$ in \eqref{ineq:phi-Xbar-almost-upper-bound}, we arrive at the desired large deviation upper bound.

Turning to the lower bound, consider an open set $G$. Since $\phi^{-1}(G)$ is open, from Theorem~\ref{theorem:multi-d-levy-processes},
\begin{align*}
\liminf_{n \rightarrow \infty}\frac{1}{L(n)n^\alpha}\log
\P\left(\phi(\bar{\mathbf X}_n  )\in G\right)
&=
\liminf_{n \rightarrow \infty}\frac{1}{L(n)n^\alpha}\log
\P\left(\bar{\mathbf X}_n \in \phi^{-1}(G)\right)
\\
&
\ge -\inf_{(\xi_1,\ldots,\xi_d) \in \phi^{-1}\left(G\right)}I(\xi)=-\inf_{x \in G}I'(x).
\end{align*}
\end{proof}

%
%
\subsection{Sample path LDP w.r.t.\ $M_1'$ topology}\label{subsec:ldp-wrt-M1p}
Recall that
$
\bar X_n(t) \triangleq \frac 1n X(nt) - t \E X(1)
$
and 
$
\bar S_n(t) = \frac 1n \sum_{i=1}^{[nt]} (Z_i - \E Z).
$
In this section, we establish the full LDP for $\bar X_n$ and $\bar S_n$ w.r.t.\ the $M_1'$ topology.
For the definition of the $M_1'$ topology, see Appendix~\ref{section:appendix-for-M_1p}.
\begin{corollary}\label{corollary:LDP-wrt-M_1-prime}
$\bar X_n$ and $\bar S_n$ satisfy the LDP in $(\mathbb D, \mathcal T_{M_1'})$ with speed $L(n)n^\alpha$ and the good rate function $I_{M_1'}$.
$$
I_{M_1'}(\xi)
\triangleq
\begin{cases}
\sum_{t\in[0,1]}\big(\xi(t) - \xi(t-)\big)^\alpha
&
\text{if $\xi$ is a non-decreasing pure jump function with } \xi(0) \geq 0,
\\
\infty
&
\text{otherwise}.
\end{cases}.
$$
\end{corollary}
\begin{proof}
Since the proof for $\bar S_n$ is nearly identical, we only provide the proof for $\bar X_n$.
From Proposition~\ref{thm:M_1p-goodness}
we know that $I_{M_1'}$ is a good rate function.
For the LDP upper bound, suppose that $F$ is a closed set w.r.t.\ the $M_1'$ topology.
Then, it is also closed w.r.t.\ the $J_1$ topology.
From the upper bound of Theorem~\ref{theorem:1d-extended-ldp} and the fact that $I_{M_1'}(\xi) \leq I(\xi)$ for any $\xi \in \D$,
$$
\limsup_{n\to\infty} \frac{\log \P(\bar X_n \in F)}{L(n)n^\alpha}
\leq
-\lim_{\epsilon\to 0} \inf_{\xi\in F^\epsilon} I(\xi) 
\leq -\lim_{\epsilon\to 0} \inf_{\xi\in F^\epsilon} I_{M_1'}(\xi) 
= -\inf_{\xi\in F} I_{M_1'}(\xi).
$$
Turning to the lower bound, suppose that $G$ is an open set w.r.t.\ the $M_1'$ topology.
We claim that
$$\inf_{\xi \in G}I_{M_1'}(\xi) = \inf_{\xi \in G}I(\xi).$$
To show this, we only have to show that the RHS is not strictly larger than the LHS.
Suppose that $I_{M_1'}(\xi) < I(\xi)$ for some $\xi \in G$.
Since $I$ and $I_{M_1'}$ differ only if the path has a jump at either 0 or 1, this means that $\xi$ is a non-negative pure jump function of the following form:
$$
\xi = \sum_{i=1}^\infty z_i \I_{[u_i,1]},
$$
where $u_1=0$, $u_2=1$, $u_i$'s are all distinct in $(0,1)$ for $i\geq 3$ and $z_i \geq 0$ for all $i$'s.
Note that one can pick an arbitrarily small $\epsilon$ so that $\sum_{i\in\{n: u_n < \epsilon\}}z_i < \epsilon$, $\sum_{i\in\{n: u_n > 1-\epsilon\}}z_i < \epsilon$, $\epsilon \neq u_i$  for all $i\geq 2$, and $1-\epsilon \neq u_i$ for all $i\geq 2$.
For such $\epsilon$'s, if we set
$$
\xi_\epsilon \triangleq z_1 \I_{[\epsilon, 1]} + z_2 \I_{[1-\epsilon,1]} + \sum_{i=3}^\infty z_i \I_{[u_i,1]}
,
$$
then $d_{M_1'}(\xi, \xi_\epsilon) \leq \epsilon$ while $I(\xi_\epsilon) = I_{M_1'}(\xi)$.
That is, we can find an arbitrarily close element $\xi_\epsilon$ from $\xi$ w.r.t.\ the $M_1'$ metric by pushing the jump times at 0 and 1 slightly to the inside of $(0,1)$; at such an element, $I$ assumes the same value as $I_{M_1'}(\xi)$.
Since $G$ is open w.r.t.\ $M_1'$, one can choose $\epsilon$ small enough so that $\xi_\epsilon \in G$.
This proves the claim.
Now, the desired LDP lower bound is immediate from the LDP lower bound in Theorem~\ref{theorem:1d-extended-ldp} since $G$ is also an open set in the $J_1$ topology.
\end{proof}

%
%
\subsection{Nonexistence of large deviation principle in the $J_1$ topology}
Consider a compound Poisson process with arrival rate equal to 1  whose jump distribution is Weibull with the shape parameter $1/2$.
More specifically, $\bar X_n (t) \triangleq \frac1n \sum_{i=1}^{N(nt)}Z_i-t$ with $\P(Z_i\geq x) \sim \exp(-x^{\alpha})$, $\E Z_i = 1$, and $\alpha = 1/2$.
If $\bar X_n$ satisfies a full LDP w.r.t.\ the $J_1$ topology, the rate function that controls the LDP (with speed $n^\alpha$) associated with $\bar X_n$ should be of the same form as the one that controls  the extended LDP:
$$
I(\xi) =
\left\{
\begin{array}{ll}
\sum_{t\in[0,1]}\big(\xi(t) - \xi(t-)\big)^{\alpha}
&
\text{if }\xi\in \D_{\leqslant\infty},\\
\infty,
&
\text{otherwise.}
\end{array}
\right.
$$
To show that such a LDP is fundamentally impossible, we construct a closed set $A$ for which
\begin{equation}\label{ineq:the-counter-example}
\limsup_{n\to\infty} \frac{\log\P(\bar X_n\in A)}{n^\alpha} > -\inf_{\xi \in A} I(\xi).
\end{equation}
Let
$$
\varphi_{s,t} (\xi)
\triangleq
\lim_{\epsilon\to 0}
\sup_{0\vee(s-\epsilon) \leq u \leq v \leq 1\wedge (t + \epsilon)} (\xi(v) - \xi(u)).
$$
Let $A_{c;s,t} \triangleq \{\xi: \varphi_{s,t}(\xi) \geq c\}$ be (roughly speaking) the set of paths which increase at least by $c$ between time $s$ and $t$.
Then $A_{c;s,t}$ is a closed set for each $c$, $s$, and $t$.\\

\cf{
To see this, suppose that $\xi \notin A_{c;s,t}$, i.e., $\varphi_{s,t}(\xi) < a$.
Then there exists $\epsilon>0$ such that
$$\sup_{s-\epsilon \leq u \leq v \leq t + \epsilon} \xi(v) - \xi(u) < a - 2\epsilon.$$
If $d(\zeta, \xi) < \epsilon$, there exists $\lambda$ such that
$$\|\xi\circ \lambda - \zeta\| < \epsilon\qquad\text{and}\qquad\|\lambda - e \| < \epsilon.$$
For any $u,v\in [s-\epsilon, t+\epsilon]$
$$\zeta(v) - \zeta(u) \leq \xi\circ \lambda(v) - \xi\circ \lambda(u) + 2\epsilon
\leq \sup_{s-\epsilon\leq u\leq v\leq t+\epsilon} \xi(v)  - \xi(u) + 2\epsilon
< a.
$$
That is, $B(\xi; \epsilon) \subseteq A_{c;s,t}^\stcomp $.
}
Let
$$A_m \triangleq
\Big(A_{1;\frac{m+1}{m+2},\frac{m+1}{m+2}+mh_m}\Big)
\cap
\Big(A_{1;\frac{m}{m+2},\frac{m}{m+2}+mh_m}\Big)
\cap
\Big(\bigcap_{j=0}^{m-1} A_{\frac{1}{m^2};\frac{j}{m+2},\frac{j}{m+2}+mh_m}\Big)$$
where
$h_m= \frac{1}{(m+1)(m+2)}$, and let
$$
A \triangleq \bigcup_{m=1}^\infty A_m.
$$

\cf{That is,
\begin{equation}\label{def:Am}
A_m \triangleq A_{1;\frac{(m+1)(m+1)}{(m+1)(m+2)},\frac{(m+1)(m+1)+m}{(m+1)(m+2)}} \cap A_{1;\frac{(m+1)m}{(m+1)(m+2)},\frac{(m +1)m+m }{(m +1)(m+2)}} \cap \bigcap_{j=0}^{m-1} A_{\frac{1}{m^2};\frac{(m +1)j}{(m +1)(m+2)},\frac{(m +1)j+m}{(m +1)(m+2)}}.
\end{equation}
In words, paths that belong to $A_m$ increases $m$ times by $m^{-2}$, and 2 times by 1, while the two increases of size 1 get closer to each other as $m$ increases.
In this example, the times of the two big increases get closer and closer to time 1 as $m$ increases; but there is nothing special about time 1.
For example, we can make the times of the two big increases get close to 1/2 (instead of 1) as $m$ increases and proceed similarly as follows.
We pick the time 1 here for the sake of notational convenience.
}
To see that $A$ is closed, we first claim that $\zeta \in \mathbb D\setminus A$ implies the existence of $\epsilon>0$ and $N\geq 0$ such that $B(\zeta; \epsilon) \cap A_m = \emptyset$ for all $m\geq N$.
To prove this claim, suppose not.
It is straightforward to check that for each $n$, there has to be $s_n,t_n \in [1-1/n, 1)$ such that $s_n\leq t_n$ and $\zeta(t_n) - \zeta(s_n) \geq 1/2$, which in turn implies that $\zeta$ must possess infinite number of increases of size at least 1/2 in $[1-\delta, 1)$ for any $\delta>0$.
This implies that $\zeta$ cannot possess a left limit, which is contradictory to the assumption that $\zeta \in \mathbb D\setminus A$.
On the other hand, since each $A_m$ is closed, $\bigcup_{i=1}^N A_i$ is also closed, and hence, there exists $\epsilon'>0$ such that $B(\zeta; \epsilon') \cap A_m = \emptyset$ for $m=1,\ldots, N$.
Now, from the construction of $\epsilon$ and $\epsilon'$, $B(\zeta, \epsilon \vee \epsilon') \cap A = \emptyset$, proving that $A$ is closed.

Next, we show that $A$ satisfies (\ref{ineq:the-counter-example}).
First note that if $\xi$ is a pure jump function that belongs to $A_m$, $\xi$ has to possess $m$ upward jumps of size $1/m^2$ and 2 upward jumps of size $1$, and hence,
\begin{equation}\label{bnd:infI}
\inf_{\xi\in A} I(\xi) \geq \inf_m \Big(1^{1/2} + 1^{1/2} + m (1/m^2)^{1/2}\Big) = 3.
\end{equation}
On the other hand, letting $\Delta \xi(t) \triangleq \xi(t) - \xi(t-)$,
\begin{align*}
&
\P(\bar X_{(n+1)(n+2)} \in A_n )
\\
&
\geq
\prod_{j=0}^{n-1} \P\bigg(\sup_{t\in[0,1]} \Big\{\bar X_{(n+1)(n+2)}\Big({\textstyle\frac{(n+1)j + nt}{(n +1)(n+2)}}\Big)-\bar X_{(n +1)(n+2)}\Big({\textstyle\frac{(n +1)j }{(n +1)(n+2)}}\Big)\Big\}\geq \frac1{n^2} \bigg)
\\
&
\qquad
\cdot
\P\bigg(\sup_{t\in(0,1]} \Big\{\Delta\bar X_{(n +1)(n+2)}\Big({\textstyle\frac{(n +1)n + n t}{(n +1)(n+2)}}\Big)\Big\}\geq 1 \bigg)
\cdot
\P\bigg(\sup_{t\in(0,1]} \Big\{\Delta\bar X_{(n+1)(n+2)}\Big({\textstyle\frac{(n+1)(n+1) + n t}{(n +1)(n+2)}}\Big)\Big\}\geq 1 \bigg)
\\
&
=
\P\bigg(\sup_{t\in[0,1]} \Big\{\bar X_{(n +1)(n+2)}\Big({\textstyle\frac{n t}{(n +1)(n+2)}}\Big)\Big\}\geq \frac1{n^2} \bigg)^n
\cdot
\P\bigg(\sup_{t\in[0,1]} \Big\{\Delta\bar X_{(n +1)(n+2)}\Big({\textstyle\frac{n t}{(n +1)(n+2)}}\Big)\Big\}\geq 1 \bigg)^2
\\
&
=
\P\bigg(\sup_{t\in[0,1]} \Big\{X(nt)\Big\}\geq \frac{(n+1)(n+2)}{n^2} \bigg)^n
\cdot
\P\bigg(\sup_{t\in[0,1]} \Big\{\Delta X(nt)\Big\}\geq (n+1)(n+2) \bigg)^2
\\
&
\geq
\P\bigg(\sup_{t\in[0,1]} \Big\{X(nt)\Big\}\geq 6 \bigg)^n
\cdot
\P\bigg(\sup_{t\in[0,1]} \Big\{\Delta X(nt)\Big\}\geq (n+1)(n+2) \bigg)^2,
\end{align*}
and hence,
\begin{equation}\label{bnd:I-II}
\begin{aligned}
&
\limsup_{n\to\infty}\frac{\log \P(\bar X_{n} \in A )}{n^\alpha}
\geq
\limsup_{n\to\infty}\frac{\log \P(\bar X_{(n+1)(n+2)} \in A_n )}{\big((n +1)(n+2)\big)^\alpha}
\\
&
\geq
\limsup_{n\to\infty}\frac{\log\P\bigg(\sup_{t\in[0,1]} \Big\{X(n t)\Big\}\geq 6 \bigg)^n}{\big((n +1)(n+2)\big)^\alpha}  +2\limsup_{n\to\infty}\frac{ \log \P\bigg(\sup_{t\in[0,1]} \Big\{\Delta X(n t)\Big\}\geq (n+1)(n+2) \bigg)}{\big((n +1)(n+2)\big)^\alpha}
\\
&
=
\text{(I)} + \text{(II)}.
\end{aligned}
\end{equation}
Letting $p_n \triangleq \P\bigg(\sup_{t\in[0,n]} \Big\{X(t)\Big\}< 6 \bigg)$,
\begin{align}\label{bnd:I}
\text{(I)}
&
=
\limsup_{n\to\infty}\frac{\log(1-p_n)^n }{\big((n +1)(n+2)\big)^\alpha}
=
\limsup_{n\to\infty}\frac{np_n\log(1-p_n)^{1/p_n}}{\big((n +1)(n+2)\big)^\alpha}
=
\limsup_{n\to\infty}\frac{-np_n}{\big((n +1)(n+2)\big)^\alpha}
=
0
\end{align}
since $p_n\to0$ as $n\to\infty$.\\

\cf{
This is probably a standard result, but I include the proof for
\begin{equation*}\label{eq:perhaps-standard}
p_n = \P\bigg(\sup_{t\in[0,1]} \Big\{X(nt)\Big\}\leq 6 \bigg) = O(n^{-1/2})
\end{equation*}
here just to make sure.

\begin{align*}
\P\bigg( \sup_{t\in [0,1]} X(nt) \leq 6\bigg)
&
=
\P\bigg(\sup_{t\in[0,1]} \sum_{i=1}^{N(nt)} Z_i - n t \leq 6\bigg)
=
\P\bigg(\sup_{t\in[0,n]} \sum_{i=1}^{N(t)} Z_i - t \leq 6\bigg)
\\
&
=
\P\bigg(\max_{j=0,1,\ldots,N(n)} \sum_{i=1}^{j} (Z_i - \tau_i) \leq 6\bigg)
\end{align*}
where $\tau_i$'s are the holding times---i.e.,
$\tau_i \triangleq \gamma_i - \gamma_{i-1}$, $\gamma_i \triangleq \inf\{t\geq \gamma_{i-1}: N(t) > N(\gamma_{i-1})\}$, and $\gamma_0 = 0$.
Therefore,
\begin{align*}
&\P\bigg( \sup_{t\in [0,1]} X(nt) \leq 6\bigg)
\\
&
\leq
\P\bigg(\max_{j=0,1,\ldots,N(n)} \sum_{i=1}^{j} (Z_i - \tau_i) \leq 6,\ N(n) \leq (1-\epsilon)n\bigg)
+
\P\bigg(\max_{j=0,1,\ldots,N(n)} \sum_{i=1}^{j} (Z_i - \tau_i) \leq 6,\ N(n) \geq (1-\epsilon)n\bigg)
\\
&
\leq
\P\bigg(N(n)/n \leq 1-\epsilon\bigg)
+
\P\bigg(\max_{j=0,1,\ldots,\lfloor(1-\epsilon)n\rfloor} \sum_{i=1}^{j} (Z_i - \tau_i) \leq 6\bigg)
= O(n^{-1/2})
\end{align*}
since $\P\bigg(N(n)/n \leq 1-\epsilon\bigg)$ decays at a geometric rate and
$$\P\bigg(\max_{j=0,1,\ldots,\lfloor(1-\epsilon)n\rfloor} \sum_{i=1}^{j} (Z_i - \tau_i) \leq 6\bigg) < c/n$$
for some $c>0$; see, for example, Theorem 5.1.7 of \cite{lawler2010random}.
}

{
For term \text{(II)}, notice that that $\left\{Z_1 \geq (n+1)(n+2)\right\} \implies  \left\{ \sup_{t\in[0,1]} \Big\{\Delta X(nt)\Big\}\geq (n+1)(n+2)\right\}$. That is,

\begin{align}\label{eq:ver2ce}
\text{(II)}
&=
2\limsup_{n \to \infty}\frac{\log \P\bigg(\sup_{t\in[0,1]} \Big\{\Delta X(nt)\Big\}\geq (n+1)(n+2) \bigg)}{\big((n+1)(n+2)\big)^\alpha} \nonumber
\\
&
\geq
2\limsup_{n\to\infty}\frac{\log \P\bigg(Z_1\geq (n+1)(n+2) \bigg)}{\big((n+1)(n+2)\big)^\alpha}=2\limsup_{n\to\infty}\frac{\log e^{-\big((n+1)(n+2)\big)^\alpha}}{\big((n+1)(n+2)\big)^\alpha}
\nonumber
\\
&
= -2.
\end{align}
}

From (\ref{bnd:I-II}), (\ref{bnd:I}), (\ref{eq:ver2ce}),
\begin{equation}\label{bnd:limsup}
\limsup_{n\to\infty} \frac{\log \P(\bar X_{n}\in A)}{n^\alpha} \geq -2.
\end{equation}
This along with (\ref{bnd:infI}),
$$
\limsup_{n\to\infty} \frac{\log \P(\bar X_{n}\in A)}{n^\alpha} \geq -2
>
-3 \geq -\inf_{\xi \in A} I(\xi),
$$
which means that $A$ indeed is a counterexample for the desired LDP.
\cf{
In case of general $\alpha$, considering
$$
A_m
\triangleq
A_{1;\frac{(\lfloor m^\beta\rfloor+1)(m+1)}{(\lfloor m^\beta\rfloor+1)(m+2)},\frac{(\lfloor m^\beta\rfloor+1)(m+1)+\lfloor m^\beta\rfloor}{(\lfloor m^\beta\rfloor+1)(m+2)}}
\cap
A_{1;\frac{(\lfloor m^\beta\rfloor+1)m}{(\lfloor m^\beta\rfloor+1)(m+2)},\frac{(\lfloor m^\beta\rfloor +1)m+\lfloor m^\beta\rfloor }{(\lfloor m^\beta\rfloor +1)(m+2)}}
\cap
\bigcap_{j=0}^{m-1}
A_{\frac{1}{m^2};\frac{(\lfloor m^\beta\rfloor +1)j}{(\lfloor m^\beta\rfloor +1)(m+2)},\frac{(\lfloor m^\beta\rfloor +1)j+\lfloor m^\beta\rfloor}{(\lfloor m^\beta\rfloor +1)(m+2)}}.
$$
for $\beta = 1/\alpha -1$ instead of (\ref{def:Am}) gives the counter example.
}

Note that a simpler counterexample can be constructed using the peculiarity of the $J_1$ topology at the boundary of the domain; that is, jumps (of size 1, say) at time 0 are bounded away from the jumps at arbitrarily close jump times.
For example, if $\bar Y_n(t) \triangleq \frac1n\sum_{i=1}^{N(nt)} Z_i + t$ where the right tail of $Z$ is Weibull and $\E Z = -1$, then the set $B = \{x: x(t) \geq t/2 \text{ for all }t \in [0,1]\}$ gives a counterexample for the LDP.
Note that the $M_1'$ topology we used in Section~\ref{subsec:ldp-wrt-M1p} is a treatment for the same peculiarity of the $M_1$ topology at time $0$.
However, it should be clear from the above counterexample $\bar X_n$ and $A$ that the LDP is fundamentally impossible w.r.t.\ $J_1$-like topologies---i.e., the ones that do not allow merging two or more jumps to approximate a single jump at any time---and hence, there is no hope for the full LDP in the case of the $J_1$ topology.

%
%
%
%

\section{Boundary crossing with moderate jumps}
In this section, we illustrate the value of Corollary~\ref{theorem:contraction-principle}; for an application of Corollary~\ref{corollary:LDP-wrt-M_1-prime} we refer to \cite{BBRZ2}.
We consider level crossing probabilities of L\'evy processes where the jump sizes are conditioned to be moderate. Specifically, we apply 
Corollary~\ref{theorem:contraction-principle} in order to provide large-deviations estimates for
\begin{equation}
\label{moderate-prob}
\P\left(\sup_{t\in [0,1]}\bar{X}_n(t)\geq c, \sup_{t \in [0,1]}\big|\bar{X}_n(t)-\bar{X}_n(t-)\big|\leq b\right).
\end{equation}
We emphasize that this type of rare events are difficult to analyze with the tools developed previously.
In particular, the sample path LDP proved in \cite{Gantert1998} is w.r.t.\ the $L_1$ topology.
Since the closure of the sets in (\ref{moderate-prob}) w.r.t.\ the $L_1$ topology contains the zero function, the LDP upper bound would not provide any information.

Functionals like (\ref{moderate-prob}) appear in actuarial models, in case excessively large insurance claims are reinsured and therefore do not play a role in the ruin of an insurance company.
\cite{AsmussenPihlsgaard}, for example, studied various estimates of infinite-time ruin probabilities with analytic methods.
\cite{rheeblanchetzwart2016} studied the finite-time ruin probability, using probabilistic techniques in case of regularly varying L\'evy measures and confirmed that the conventional wisdom ``the principle of a single big jump'' can be extended to ``the principle of the minimal number of big jumps'' in such a context.
Here we show that a similar result---with subtle differences---can be obtained in case the L\'evy measure has a Weibull tail.

Define the function $\phi:\mathbbm{D} \to \R^2$ as
\begin{equation}
\phi(\xi)
=
\left(\phi_{1}(\xi), \phi_{2}(\xi)\right) \triangleq \left(\sup_{t \in [0,1]}\xi(t), \sup_{t \in [0,1]} |\xi(t)-\xi(t-)|\right)\nonumber.
\end{equation}
In order to apply Corollary~\ref{theorem:contraction-principle}, we will validate  that $\phi$ is Lipschitz continuous and that $I'(x,y) \triangleq \inf_{\{\xi\in  \D: \phi(\xi) = (x,y)\}}I(\xi)$ is a good rate function.

For the Lipschitz continuity of $\phi$,
we claim that each component of $\phi$ is Lipschitz continuous.
We first examine $\phi_1$.
Let $\xi, \zeta\in \D$ and suppose w.l.o.g.\ that $\sup_{t \in [0,1]}\xi(t) > \sup_{t \in [0,1]}\zeta(t) $.
For an arbitrary non-decreasing homeomorphism $\lambda:[0,1]\to[0,1]$,
\begin{align}\label{eq:phi1-is-Lipschitz}
|\phi_1(\xi)-\phi_1(\zeta)|
&
= |\sup_{t \in [0,1]}\xi(t)-\sup_{t \in [0,1]}\zeta(t)|
=
|\sup_{t \in [0,1]}\xi(t)-\sup_{t \in [0,1]}\zeta\circ\lambda(t)|
\\
&
\leq \sup_{t \in [0,1]}|\xi(t)-\zeta\circ\lambda(t)|
\leq
\sup_{t \in [0,1]}|\xi(t)-\zeta\circ\lambda(t)|\vee|\lambda(t)-t|.\nonumber
\end{align}
Taking the infimum over $\lambda$, we conclude that
$$
|\phi_1(\xi)-\phi_1(\zeta)| \leq \inf_{\lambda\in \Lambda}\sup_{t \in [0,1]}\{|\xi(t)-\zeta(\lambda(t))|\vee|\lambda(t)-t|\}=
d_{J_1}(\xi,\zeta).
$$
Therefore, $\phi_1$ is Lipschitz with the Lipschitz constant 1.

Now, in order to prove that $\phi_{2}(\xi)$ is Lipschitz,
fix two distinct paths $\xi, \zeta \in \D$ and assume w.l.o.g.\ that $\phi_2(\zeta) > \phi_2(\xi)$.
Let $c\triangleq \phi_2(\zeta) - \phi_2(\xi) > 0$, and let $t^*\in[0,1]$ be the maximum jump time of $\xi$, i.e., $\phi_{2}(\xi)=|\xi(t^*)-\xi(t^*-)|$.
For any $\epsilon>0$ there exists $\lambda^*$ so that
\begin{align}\label{eq:ineq-1}
d_{J_1}(\xi,\zeta)
&
\triangleq
\inf_{\lambda \in \Lambda}\{\|\xi-\zeta\circ \lambda\|_{\infty} \vee \|\lambda-e\|_{\infty}\}
\ge
\|\xi-\zeta\circ \lambda^{*}\|_{\infty} \vee \|\lambda^{*}-e\|_{\infty}-\epsilon.
\nonumber
\\
&
\geq
|\xi(t^*)-\zeta\circ\lambda^{*}(t^*)|\vee|\xi(t^*-)-\zeta\circ\lambda^{*}(t^*-)|-\epsilon.
\end{align}
From the general inequality $|a-b| \vee |c-d| \geq \frac12(|a-c|-|b-d|)$,
\begin{align}\label{eq:maximum}
|\xi(t^*)-\zeta\circ\lambda^{*}(t^*)|\vee|\xi(t_{1^-})-\zeta\circ\lambda^{*}(t^*-)|
&
\ge
\frac 1 2\big(|\xi(t^*)-\xi(t^*-)|-|\zeta\circ\lambda^*(t^*) - \zeta\circ\lambda^*(t^*-)|\big)
\nonumber
\\
&
=
\frac 1 2\big(\phi_2(\xi)-\phi_2(\zeta)\big)
=
c/2.
\end{align}
In view of \eqref{eq:ineq-1} and \eqref{eq:maximum},
$
d_{J_1}(\xi, \zeta) \geq \frac12 |\phi(\xi) - \phi(\zeta)|-\epsilon.
$
Since $\epsilon$ is arbitrary, we get the desired Lipschitz bound with Lipschitz constant 2.
Therefore,
$
|\phi(\xi)-\phi(\zeta)| = |\phi_1(\xi)-\phi_1(\zeta)|\vee|\phi_2(\xi)-\phi_2(\zeta)|\leq 2d_{J_1}(\xi,\zeta)
$
and hence, $\phi$ is Lipschitz with Lipschitz constant 2.

Now, we claim that $I'$ is of the following form:
\begin{equation}\label{eq:rate-function-RW-moderate-jumps}
I'(c,b)=
\begin{cases}
\left\lfloor \frac c b \right\rfloor b^\alpha +\left(c-\left\lfloor \frac c b \right\rfloor b\right)^\alpha & \text{if} \ 0<b \leq c, \\
0
& \text{if} \ b=c=0,\\
\infty
& \text{otherwise.}
\end{cases}
\end{equation}
Note first that \eqref{eq:rate-function-RW-moderate-jumps} is obvious except for the first case, and hence, we will assume that $0<b\leq c$ from now on.
Note also that $I'(c,b)=\inf\{I(\xi): \xi\in \mathbbm{D}_{\leqslant\infty}, \ \phi(\xi)=(c,b)\}$ since $I(\xi)=\infty$ if $\xi\notin \D_{\leqslant\infty}$.
Set $\mathcal{C}\triangleq \{\xi \in \mathbbm{D}_{\leqslant\infty}, (c,b)=\phi(\xi)\}$ and remember that any $\xi\in \D_{\leqslant\infty}$ admits the following representation:
\begin{equation}\label{eq:representation-non-decrea-pure-jump-function}
\xi=\sum_{i=1}^{\infty}x_i\mathbbm{1}_{[u_i,1]},
\end{equation}
where $u_i$'s are distinct in $(0,1)$ and $x_i$'s are non-negative and sorted in a decreasing order.
Consider a step function $\xi_0 \in \mathcal{C}$, with $\left\lfloor \frac c b \right\rfloor$ jumps of size $b$ and one jump of size $c-\left\lfloor \frac c b \right\rfloor b$, so that $\xi_0=\sum_{i=1}^{\left\lfloor \frac c b \right\rfloor} b \mathbbm{1}_{[u_i,1]}+(c-\left\lfloor \frac c b \right\rfloor) \mathbbm{1}_{[u_{ \left\lfloor \frac c b \right\rfloor +1},1]}.$ Clearly, $\phi(\xi_0)=(c,b)$ and  $I(\xi_0)=\left\lfloor \frac c b \right\rfloor b^\alpha +\left(c-\left\lfloor \frac c b \right\rfloor b\right)^\alpha$. Since  $\xi_0 \in \mathcal{C}$, the infimum of $I$ over $\mathcal{C}$ should be at most $I(\xi_0) $ i.e., $I(\xi_0) \geq  I'(c,b)$.

To prove that $\xi_0$ is the minimizer of $I$ over $\mathcal C$,
we will show that $I(\xi) \geq I(\xi_0)$ for any $\xi = \sum_{i=1}^\infty x_i\I_{[u_i,1]} \in \mathcal C$  by constructing $\xi'$ such that $I(\xi) \geq I(\xi')$ while $I(\xi') = I(\xi_0)$.
There has to be an integer $k$ such that $x_k'\triangleq \sum_{i=k}^\infty x_i \leq b$.
Let $\xi^1 \triangleq \sum_{i=1}^k x_i^1 \I_{[u_i,1]}$ where $x_i^1$ is the $i$\textsuperscript{th} largest element of $\{x_1,\ldots,x_{k-1},x_k'\}$.
Then, $\xi^1 \in \mathcal C$ and $I(\xi^1) \leq I(\xi)$
due to the sub-additivity of $x\mapsto x^\alpha$.
Now, given $\xi^j = \sum_{i=1}^k x_i^j\I_{[u_i,1]}$, we construct $\xi^{j+1}$ as follows.
Find the first $l$ such that $x_l^j < b$.
If $x_l^j = 0$ or $x_{l+1}^j = 0$, set $\xi^{j+1} \triangleq \xi^j$.
Otherwise, find the first $m$ such that $x_{m+1}^j = 0$ and merge the $l$\textsuperscript{th} jump and the $m$\textsuperscript{th} jump.
More specifically, set
$x_{l}^{j+1} \triangleq x_l^j + x_m^j \wedge (b-x_l^j)$,
$x_{m}^{j+1} \triangleq x_m^j - x_m^j \wedge (b-x_l^j)$,
$x_i^{j+1} \triangleq x_i^j$ for $i\neq l,m$, and
$\xi^{j+1} \triangleq \sum_{i=1}^k x_i^{j+1}\I_{[u_i,1]}$.
Note that $x_{l}^{j+1} + x_{m}^{j+1} = x_{l}^{j} + x_{m}^{j}$ while either $x_l^{j+1} = b$ or $x_m^{j+1} = 0$.
That is, compared to $\xi^j$, $\xi^{j+1}$ has either one less jump or one more jump with size $b$, while the total sum of the jump sizes and the maximum jump size remain the same.
From this observation and the concavity of $x\mapsto x^\alpha$, it is straightforward to check that $I(\xi^{j+1}) \leq I(\xi^j)$.
By iterating this procedure $k$ times, we arrive at $\xi' \triangleq \xi^k$ such that all the jump sizes of $\xi'$ are $b$, or there is only one jump whose size is not $b$.
From this, we see that $\xi^k$ has to coincide with $\xi_0$.
We conclude that
$I(\xi) \geq I(\xi^1)\geq \cdots\geq I(\xi^k) = I(\xi') = I(\xi_0)$, proving that $\xi_0$ is indeed a minimizer.

Now we check that $\Psi_{I'}(\gamma)\triangleq\{(c,b):I'(c,b) \le \gamma\}$ is compact for each $\gamma \in [0, \infty)$ so that $I'$ is a good rate function.
It is clear that $\Psi_{I'}(\gamma)$ is bounded.
To see that $\Psi_{I'}(\gamma)$ is closed,
suppose that $(c_1,b_1) \notin \Psi_{I'}(\gamma)$.
In case $0 < b_1 < c_1$, note that $I'$ can be written as
$I'(c,b) = b^\alpha\Big\{\big(c/b - \lfloor c/b \rfloor \big)^\alpha + \lfloor c/b \rfloor \Big\}$,
from which it is easy to see that $I'$ is continuous at such $(c_1,b_1)$'s.
Therefore, one can find an open ball around $(c_1,b_1)$ in such a way that it doesn't intersect with $\Psi_{I'}(\gamma)$.
By considering the cases $c_1 < b_1$, $b_1 = 0$, $b_1=c_1$ separately, the rest of the cases are straightforward to check.
We thus conclude that $I'$ is a good rate function.
Now we can apply Corollary~\ref{theorem:contraction-principle}.
Note that
$$\inf_{(x,y)\in [c,\infty)\times[0,b]} I'(x,y) = \inf_{(x,y)\in (c,\infty)\times[0,b)} I'(x,y) =
I'(c,b).$$
{That is the large deviation lower and upper bound coincide
and hence,}
\begin{align*}
&{\lim_{n\to\infty} \frac{\log\P\left(\sup_{t\in [0,1]}\bar{X}_n(t)\geq c, \sup_{t \in [0,1]}\big|\bar{X}_n(t)-\bar{X}_n(t-)\big|\leq b\right)}{L(n)n^\alpha}}
\\
&
\qquad
=
\begin{cases}
\left\lfloor \frac c b \right\rfloor b^\alpha +\left(c-\left\lfloor \frac c b \right\rfloor b\right)^\alpha & \text{if} \ 0<b \leq c, \\
0
& \text{if} \ b=c=0,\\
\infty
& \text{otherwise.}
\end{cases}
\end{align*}
From the expression of the rate function, it can be inferred that the most likely way level $c$ is reached is due to $\left\lfloor \frac c b \right\rfloor$ jumps of size $b$ and one jump of size
$\left(c-\left\lfloor \frac c b \right\rfloor b\right)$.
If we compare this with the insights obtained from the case of truncated regularly varying tails in \cite{rheeblanchetzwart2016}, we see that the total number of jumps is identical, but the size of the jumps are deterministic and non-identical, while in the regularly varying case, they are identically distributed with Pareto distribution.

\section{Proofs}\label{section:proofs}
This section provides technical proofs.

\subsection{Lower semi-continuity of $I$ and $I^d$}
Recall the definition of $I$ in \eqref{eq:the-rate-function-1d} and $I^d$ in \eqref{eq:rate-function-Id}.
\begin{lemma}\label{theorem:lower-semi-continuity}
$I$ and $I^d$ are lower semi-continuous, and hence, rate functions.
\end{lemma}
\begin{proof}
We start with $I$.
To show that the sub-level sets $\Psi_I(\gamma)$ are closed for each $\gamma<\infty$,
let $\xi$ be any given path that does not belong to $\Psi_I(\gamma)$.
We will show that there exists an $\epsilon>0$ such that $d_{J_1}(\xi,\Psi_I(\gamma)) \geq \epsilon$.
Note that $\Psi_I(\gamma)^\stcomp=(A \cap B\cap C\cap  D)^\stcomp = (A^\stcomp) \cup (A\cap B^c) \cup (A \cap B \cap C^\stcomp)\cup (A \cap B\cap C\cap D^\stcomp)$ where
\begin{align*}
&
A=\{\xi \in \mathbbm{D}: \xi(0) = 0 \text{ and } \xi(1) = \xi(1-)\},
&&
B=\{\xi \in \mathbbm{D}: \xi \text{ is non-decreasing}\},
\\
&
C=\{\xi \in \mathbbm{D}: \xi \text{ is a pure jump function}\},
&&
D=\{\xi \in \mathbbm{D}: \textstyle\sum_{t\in [0,1]}(\xi(t)-\xi(t-))^\alpha\leq \gamma\}.
\end{align*}
For $\xi \in A^\stcomp$, we will show that $d_{J_1}(\xi,\Psi_I(\gamma))\geq\delta$ where $\delta = \frac12\max\{|\xi(0)|, |\xi(1) - \xi(1-)|\}$.
Suppose not so that there exists $\zeta \in \Psi_I(\gamma)$ such that $d_{J_1}(\xi,\zeta) < \delta$.
Then $|\zeta(0)|  >|\xi(0)|-2\delta$ {and} $|\zeta(1)-\zeta(1-)| > |\xi(1)-\xi(1-)|-2\delta$.
That is, $\max\{|\zeta(0)|, |\zeta(1) - \zeta(1-)| \} > \max \{|\xi(0)|-2\delta, |\xi(1)-\xi(1-)|-2\delta\}= 0$.
Therefore, $\zeta \in A^c$, and hence, $I(\zeta)=\infty$, which contradicts the assumption that $\zeta \in \Psi_I(\gamma)$.

If $\xi \in A\cap B^\stcomp$, there are $T_s<T_t$ such that $c \triangleq \xi(T_s)-\xi(T_t)>0$.
We claim that $d_{J_1}(\xi,\zeta)\geq c/2$ if $\zeta \in \Psi_{I}(\gamma)$.
Suppose that this is not the case and there exists $\zeta \in \Psi_{I}(\gamma)$ such that $d_{J_1}(\xi,\zeta) < c/2$.
Let $\lambda$ be a non-decreasing homeomorphism $\lambda:[0,1]\to[0,1]$ such that $\|\zeta\circ \lambda - \xi\|_\infty < c/2$, in particular, $\zeta\circ\lambda(T_s) > \xi(T_s)  -c/2$ and $\zeta\circ\lambda(T_t) < \xi(T_t) + c/2$.
Subtracting the latter inequality from the former, we get
$\zeta\circ \lambda(T_s) - \zeta\circ \lambda(T_t) > \xi(T_s) - \xi(T_t) - c  = 0$.
That is, $\zeta$ is not non-decreasing, which is contradictory to the assumption $\zeta  \in \Psi_{I}(\gamma)$.
Therefore, the claim has to be the case.

If $\xi \in A\cap B\cap C^\stcomp$, there exists an interval $[T_s,T_t]$ so that $\xi$ is continuous and $c\triangleq \xi(T_t)-\xi(T_s)>0$.
Pick $\delta$ small enough so that $(c-2\delta)(2\delta)^{\alpha-1}>\gamma$.
We will show that $d_{J_1}(\xi, \Psi_I(\gamma)) \geq \delta$.
Suppose that $\zeta \in \Psi_I(\gamma)$ and $d_{J_1}(\zeta, \xi) < \delta$,
and let $\lambda$ be a non-decreasing homeomorphism such that $\|\zeta\circ \lambda-\xi\|_\infty < \delta$.
Note that this implies that each of the jump sizes of $\zeta\circ\lambda$ in $[T_s, T_t]$ has to be less than $2\delta$.
On the other hand, $\zeta\circ\lambda(T_t) \geq \xi(T_t) - \delta$ and $\zeta\circ\lambda(T_s) \leq \xi(T_s) + \delta$, which in turn implies that $\zeta\circ\lambda(T_t) - \zeta\circ\lambda(T_s) \geq c - 2\delta$.
Since $\zeta\circ\lambda$ is a non-decreasing pure jump function,
\begin{align*}
c-2\delta
&
\leq
\zeta\circ\lambda(T_t) - \zeta\circ\lambda(T_s)
= \sum_{t\in(T_s,T_t]}\big(\zeta\circ\lambda(t)-\zeta\circ\lambda(t-)\big)
\\
&
= \sum_{t\in(T_s,T_t]}\big(\zeta\circ\lambda(t)-\zeta\circ\lambda(t-)\big)^\alpha\big(\zeta\circ\lambda(t)-\zeta\circ\lambda(t-)\big)^{1-\alpha}
\\
&
\leq
\sum_{t\in(T_s,T_t]}\big(\zeta\circ\lambda(t)-\zeta\circ\lambda(t-)\big)^\alpha
(2\delta)^{1-\alpha}.
\end{align*}
That is,
$
\sum_{t\in(T_s,T_t]}\big(\zeta\circ\lambda(t)-\zeta\circ\lambda(t-)\big)^\alpha \geq (2\delta)^{\alpha-1}(c-2\delta) > \gamma
$, which is contradictory to our assumption that $\zeta \in \Psi_I(\gamma)$.
Therefore, $d_{J_1}(\xi, \Psi_I(\gamma)) \geq \delta$.

Finally, let $\xi \in A\cap B \cap C\cap D^\stcomp$. This implies that $\xi$ admits the following representation:
$\xi=\sum_{i=1}^{\infty}x_i\mathbbm{1}_{[u_i,1]}$ where $u_i$'s are all distinct in $(0,1)$ and $\sum_{i=1}^{\infty}x_i^\alpha> \gamma$.
Choose $k$ large enough and $\delta$ small enough so that $\sum_{i=1}^{k}(x_i-2\delta)^\alpha> \gamma$.
We will show that $d_{J_1}(\xi,\Psi_I(\gamma)) \geq \delta$.
Suppose that this is not the case.
That is, there exists $\zeta \in \Psi_I(\gamma)$ so that $d_{J_1}(\xi,\zeta) < \delta$.
Let $\lambda$ be a non-decreasing homeomorphism such that $\|\zeta\circ\lambda - \xi\|_\infty < \delta$.
Thus for each $i\in \{1,\ldots,k\}$,
$
\zeta\circ\lambda(u_i)-\zeta\circ\lambda(u_i-)\geq \xi(u_i)-\xi(u_i-)-2\delta=x_i-2\delta
$,
and hence,
\begin{equation}
I(\zeta)=\sum_{t \in [0,1]}(\zeta\circ\lambda(t_i)-\zeta\circ\lambda(t_i-))^\alpha \geq \sum_{i=1}^{k}(\zeta\circ\lambda(u_i)-\zeta\circ\lambda(u_i-)) \geq \sum_{i=1}^{k}(x_i-2\delta)^\alpha>\gamma, \nonumber
\end{equation}
which contradicts  the assumption that $\zeta \in \Psi_I(\gamma)$. 

$I^{d}$ is lower semi-continuous since it is a sum of lower semi-continuous functions. 
\end{proof}

\subsection{Proof of Proposition~\ref{theorem:extended-ldp-from-decomposition}}
\begin{proof}[Proof of Proposition~\ref{theorem:extended-ldp-from-decomposition}]
We start with the extended large deviation upper bound.
For any measurable set $A$,
\begin{align}
\P\left({X}_n \in A\right)
&
=
\P\left(X_n \in A, \, d(X_n, Y_n^k) \leq \epsilon\right)
+\P\left(X_n \in A, \, d(X_n, Y_n^k)>\epsilon\right)\nonumber\\
&
\le
\underbrace{\P\big(Y_n^{k} \in A^{\epsilon}\big)}_{\triangleq \text{(I)}}
+\underbrace{\P\big(d(X_n, Y_n^k)>\epsilon\big)}_{\triangleq \text{(II)}}.
\end{align}
From the principle of the largest term and (i),
\begin{align}
\limsup_{n\rightarrow\infty}\frac{\log \P( X_n \in A)}{a_n}
&\le
\max \left\{ -\inf_{x\in A^{2\epsilon}}I_k(x),\,\limsup_{n\rightarrow\infty}\frac{1}{a_n}\log\mathbf{P}\big(d(X_n, Y_n^k)>\epsilon\big) \right\}.\nonumber
\end{align}
Now letting $k\to\infty$ and then $\epsilon\to0$, (ii) and (iv) lead to
\begin{equation*}
\limsup_{n \rightarrow \infty}\frac{1}{a_n}\log \P\left({X}_n\in A\right) \le -\lim_{\epsilon \rightarrow 0}\inf_{x\in A^\epsilon}I(x),
\end{equation*}
which is the upper bound of the extended LDP.

Turning to the lower bound, note that the lower bound is trivial if $\inf_{x\in A^\circ}I(x)=\infty$.
Therefore, we focus on the case $\inf_{x\in A^\circ}I(x) < \infty$.
Consider an arbitrary but fixed $\delta\in(0,1)$.
In view of (iii) and (iv), one can pick $\epsilon>0$ and $k\geq 1$ in such a way that
\begin{equation}\label{ineq:choice-of-epsilon-and-k-1}
-\inf_{x \in A^\circ}I(x) \le -\inf_{x \in A^{-\epsilon}}I_k(x)+\delta
\quad\qquad\text{and}\qquad\quad
\limsup_{n\to\infty}\frac{\log\P\big(d(X_n, Y_n^k)>\epsilon\big)}{a_n} \le -\inf_{x\in A^\circ} I(x) - 1.
\end{equation}
Hence
\begin{equation}\label{ineq:choice-of-epsilon-and-k-2}
\limsup_{n\to\infty}\frac{\log\P\big(d(X_n, Y_n^k)>\epsilon\big)}{a_n} \le -\inf_{x \in A^{-\epsilon}}I_k(x)+\delta-1.
\end{equation}
Now, from (\ref{ineq:choice-of-epsilon-and-k-2}) and the lower bound of the assumed extended LDP for $Y_n^k$, one can easily verify that
\begin{equation}\label{claim-in-proof-proposition-21}
\frac{\P\big(d(X_n, Y_n^k)>\epsilon\big)}{\P\big(Y_n^k \in A^{-\epsilon}\big)}\to 0
\end{equation}
as $n\to\infty$.
Using \eqref{claim-in-proof-proposition-21} and the first inequality of \eqref{ineq:choice-of-epsilon-and-k-1},
\begin{align*}
\liminf_{n\rightarrow \infty}\frac{1}{a_n}\log\P\left({X}_n \in A\right)
&
\ge
\liminf_{n\rightarrow \infty}\frac{1}{a_n}\log\P\left(Y_n^k \in A^{-\epsilon}, d(X_n, Y_n^k)\leq\epsilon\right)
\\
&
\ge
\liminf_{n\rightarrow \infty}\frac{1}{a_n}\log\Big(\P\left(Y_n^k \in A^{-\epsilon})-\P(d(X_n, Y_n^k)>\epsilon\right)\Big)
\\
&
=
\liminf_{n\rightarrow \infty}\frac{1}{a_n}\log\left(\P\big(Y_n^k \in A^{-\epsilon}\big)\left(1-\frac{\P\big(d(X_n, Y_n^k)>\epsilon\big)}{\P\big(Y_n^k \in A^{-\epsilon}\big)}\right)\right)
\\
&
=
\liminf_{n\rightarrow \infty}\frac{1}{a_n}\log\P\left(Y^{k}_n \in A^{-\epsilon}\right)
\ge
-\inf_{x \in A^{-\epsilon}}I_k(x)
\ge
-\inf_{x \in A} I(x)-\delta.
\end{align*}
Since $\delta$ was arbitrary in $(0,1)$, the lower bound is proved by letting $\delta \to 0$.
\end{proof}
\subsection{Proof of Lemma~\ref{lemma:equivalence-J}}
We prove Lemma~\ref{lemma:equivalence-J} in several steps.
Before we proceed, we introduce some notation and a distributional representation of the compound Poisson processes $Y_n$.
The following representation for the time-scaled compound Poisson process is a straightforward modification of the distributional representation on page 305 of \cite{LRR}; see also exercise 5.4 on page 163 of \cite{Resnick07}:
\begin{align*}
\int_{x\geq1} xN([0,n\cdot]\times dx)
\stackrel{\mathcal D}{=}
\sum_{l=1}^{\tilde N_n} Q_n^\gets (\Gamma_l)\I_{[U_l,1]}(\cdot),
\end{align*}
where
 $\Gamma_l = E_1 + E_2 + ... + E_l$; $E_i$'s are i.i.d.\ and standard exponential random variables; $U_l$'s are i.i.d.\ and uniform variables in $[0,1]$; $\tilde N_n  = N_n\big([0,1]\times [1,\infty)\big)$;
$
N_n = \sum_{l=1}^\infty \delta_{(U_l, Q_n^{\gets}(\Gamma_l))},
$
where $\delta_{(x,y)}$ is the Dirac measure concentrated on $(x,y)$;
$
Q_n(x) \triangleq  n\nu[x,\infty)$, and
$
Q_n^{\gets}(y) \triangleq \inf\{s>0: n\nu[s,\infty)< y\}
$.
It should be noted that $\tilde N_n$ is the largest $l$  such that $\Gamma_l \leq n\nu_1$, where $\nu_1 \triangleq \nu[1,\infty)$, and hence, $\tilde N_n \sim \text{Poisson}(n\nu_1)$.
Recall the definition of $\bar {J}_n^k$---the process which keeps (up to) the $k$ biggest $Z_i$'s among $Z_1,\ldots, Z_{N(n)}$. 
From this and the observation that $Q_n^\gets(\Gamma_l)$ is decreasing in $l$, we conclude that $\bar J_n^k$ has the following distributional representation:
\begin{align*}
\bar J_n^k(\cdot)
\stackrel{\mathcal D}{=}
\underbrace{\frac{1}{n} \sum_{i=1}^{k} Q_n^\gets(\Gamma_i)\I_{[U_i,1]}{(\cdot)}}_{\triangleq \hat{J}_n^{\leqslant k}{(\cdot)}}
-
\underbrace{\frac{1}{n}\I\{\tilde N_n<k\}\sum_{i=\tilde N_n+1}^{k}Q_n^\gets(\Gamma_i)\I_{[U_i,1]}{(\cdot)}}_{\triangleq \check{J}_n^{\leqslant k}{(\cdot)}}.
\end{align*}
Roughly speaking, $(Q_n^\gets(\Gamma_1)/n,\ldots,Q_n^\gets(\Gamma_k)/n)$ represents the $k$ largest jump sizes of $\bar Y_n$, and $\hat J_n^{\leqslant k}$ can be regarded as the process obtained by keeping only $k$ largest jumps of $\bar Y_n$ while disregarding the rest.
Lemma~\ref{lemma:ldp-for-k-biggest-jump-sizes} and Corollary~\ref{lemma:ldp-for-k-biggest-jump-sizes-and-their-times} prove an LDP for $(Q_n^\gets(\Gamma_1)/n,\ldots,Q_n^\gets(\Gamma_k)/n,U_1, \ldots, U_k)$.
Consequently, Lemma~\ref{lemma:sample-path-ldp-for-J-hat} yields a sample path LDP for $\hat J_n^{\leqslant k}$.
Finally, Lemma~\ref{lemma:equivalence-J} is proved by showing that $\bar J_n^k$ satisfies the same LDP as the one satisfied by $\hat J_n^{\leqslant k}$.

\begin{lemma}\label{lemma:ldp-for-k-biggest-jump-sizes}
$\Big(Q_n^\gets(\Gamma_1)/n,Q_n^\gets(\Gamma_2)/n,....,Q_n^\gets(\Gamma_k)/n\Big)$ satisfies a large deviation principle in $\R_{+}^k$ with normalization $L(n)n^\alpha,$ and with good rate function
\begin{equation}\label{def-I-check-k}
\check I_k(x_1,...x_k)=
\begin{cases}
\sum_{i=1}^{k} x_i^\alpha & \mathrm{if}\ x_1 \ge x_2 \ge \dots \ge x_k \ge 0\\
\infty, & \mathrm{o.w.}
\end{cases}.
\end{equation}
	
\end{lemma}

\begin{proof}
It is straightforward to check that $\check{I}_k$ is a good rate function.
For each $f \in \mathcal{C}_b(\R_{+}^k)$, let	
\begin{equation}\label{eq:brycslemma}
\Lambda_f^*\triangleq \lim_{n\rightarrow\infty}\frac{1}{L(n)n^\alpha}\log\left({\E  e^{L(n)n^\alpha f\big(Q_n^\gets(\Gamma_1)/n,\,Q_n^{\gets}(\Gamma_2)/n,\ldots,\,Q_n^\gets(\Gamma_k)/n\big)}}
\right).
\end{equation}	
Applying Bryc's inverse Varadhan lemma (see e.g.\ Theorem~4.4.13 of \citealp{dembo2010large}), we can show that $(Q_n^\gets(\Gamma_1)/n,\ldots,Q_n^\gets(\Gamma_k)/n)$ satisfies a large deviation principle with speed $L(n)n^\alpha$ and  good rate function $\check I_k(x)$ if
\begin{equation}\label{eq:inverse-varadhan-condition}
\Lambda_f^*= \sup_{x \in \R_{+}^k}\{f(x)-\check I_k(x)\}
\end{equation}
for every $f \in \mathcal{C}_b(\R_{+}^k)$.

To prove \eqref{eq:inverse-varadhan-condition}, fix $f \in C_b(\R_+^k)$ and let $M$ be a constant  such that $|f(x)|\le M$ for all $x\in \R_+^k$.
We first claim that the supremum of $\Lambda_f \triangleq f-\check I_k$ is finite and attained.
Pick a constant $R$ so that $R^\alpha > 2M$.
Since $\Lambda_f$ is upper semi-continuous on $[0,R]^k$, which is compact, there exists a maximizer $\hat x \triangleq (\hat x_1,\dots,\hat x_k)$ of $\Lambda_f$ on $[0,R]^k$.
Since
\begin{equation*}
\sup_{x \in [0,R]^k}\{ f(x)-\check I_k(x)\}
\geq 
\sup_{x \in [0,R]^k}f(x)
\geq -M
\end{equation*}
and
\begin{equation*}
\sup_{x \in \R_+^k\setminus [0,R]^k}\{ f(x)-\check I_k(x)\}
<
\sup_{x \in \R_+^k\setminus [0,R]^k}\{ f(x)- 2M\} \leq -M,
\end{equation*}
$\hat x$ is, in fact, a global maximizer.
Now, it is enough to prove that
\begin{equation}\label{ineq:bryc-liminf-limsup}
\Lambda_f(\hat x)
\leq
\liminf_{n\to\infty} \frac{1}{L(n)n^\alpha}\log \Upsilon_{f}(n)
\qquad\text{ and }\qquad
\limsup_{n\to\infty} \frac{1}{L(n)n^\alpha}\log \Upsilon_{f}(n)
\leq \Lambda_f(\hat x),
\end{equation}
where
\begin{equation*}
\Upsilon_{f}(n)
\triangleq \int_{\R_{+}^k}e^{L(n)n^\alpha f\big(Q_n^\gets(y_1)/n,\ldots,Q_n^\gets(y_1+\cdots+y_k)/n\big)}e^{-\sum_{i=1}^{k}y_i}\,dy_1\ldots dy_k.
\end{equation*}
We start with the lower bound---i.e., the first inequality of \eqref{ineq:bryc-liminf-limsup}.
Fix an arbitrary $\epsilon>0$.
Since $\Lambda_f$ is continuous on $A_\infty\triangleq \{(x_1,\ldots,x_k)\in\R_+^k: x_1\geq \cdots \geq x_k\}$, there exists $\delta>0 $  such that $x\in B(\hat x; 2\sqrt k\delta)\cap A_\infty$ implies $\Lambda_f(x) \geq \Lambda_f(\hat x)-\epsilon$.
Since $\prod_{j=1}^k [\hat x_j+\delta, \hat x_j+2\delta] \subseteq B(\hat x; 2\sqrt k \delta)$ and $Q_n^\gets(\cdot)$ is non-increasing,
$Q_n^\gets\big(\sum_{i=1}^j y_i\big)/n \in [\hat x_j+\delta, \hat x_j+2\delta]$  for all $j=1,\ldots, k$ implies 
$(Q_n^\gets(y_1)/n,\ldots, Q_n^\gets (y_1+\cdots+ y_k)/n) \in B(\hat x; 2\sqrt k \delta)$, and hence,
\begin{equation}\label{eq:lbcontinuity}
\Lambda_f\left( Q_n^\gets(y_1)/n,...,Q_n^\gets(y_1+\cdots+y_k)/n \right) \ge \Lambda_f(\hat x) - \epsilon.
\end{equation}
That is, if we define $D_n^j (=D_n^{y_1,\ldots,y_{j-1}})$ as
$$
D_n^j \triangleq \{y_j\in \mathbb R_+:Q_n^\gets\big(\textstyle\sum_{i=1}^{j}y_i \big)/n \in [\hat x_j+\delta,\hat x_j+2\delta]\},
$$
then
\eqref{eq:lbcontinuity} holds for $(y_1,\ldots,y_k)$'s such that $y_j\in D_n^j$ for $j=1,\ldots, k$. Therefore,
\cf{
From the properties of the general inverse we can deduce that if,
\begin{equation}\label{eq:boundsgeneralinverse}
Q_n^{\gets}(x_1+...x_j)/n \in [\hat x_j+\delta,\hat x_j+2\delta]
\end{equation}
then,
\begin{equation}\label{eq:summationlbound}
Q_n\big(n(\hat x_j+2\delta)\big)\le x_1+...+x_j \le Q_n\big(n(\hat x_j+\delta)\big).
\end{equation}
}
\begin{align}
\Upsilon_{f}(n)
&
=
\int_{\R_+^k}e^{L(n)n^\alpha\Lambda_f\big(Q_n^\gets(y_1)/n, \ldots,Q_n^\gets(y_1+\cdots+y_k)/n\big)+L(n)\sum_{i=1}^{k}Q_n^\gets(\sum_{j=1}^i y_j)^\alpha-\sum_{i=1}^{k}y_i}\,dy_1\ldots dy_k
\nonumber
\\
&
\ge
\int_{D_n^1}\dots\int_{D_n^k}e^{L(n)n^\alpha \Lambda_f\big(Q_n^\gets(y_1)/n, \ldots,Q_n^\gets(y_1+\cdots+y_k)/n\big)+L(n)\sum_{i=1}^{k}Q_n^\gets(\sum_{j=1}^i y_j)^\alpha-\sum_{i=1}^{k}y_i}dy_k\ldots dy_1
\nonumber
\\
&
\ge
\int_{D_n^1}\dots\int_{D_n^{k}}e^{L(n)n^{\alpha}\big(\Lambda_f(\hat x_1,\ldots, \hat x_k)-\epsilon\big)}e^{L(n)\sum_{i=1}^{k}Q_n^\gets(\sum_{j=1}^i y_j)^\alpha-\sum_{i=1}^{k}y_i}dy_k\ldots dy_1
\nonumber
\\
&
\ge
\int_{D_n^1}\dots\int_{D_n^k}e^{L(n)n^{\alpha}\big(\Lambda_f(\hat x_1,...,\hat x_k)-\epsilon\big)}e^{L(n)\sum_{i=1}^{k}\big(n(\hat x_i+\delta)\big)^\alpha-\sum_{i=1}^{k}y_i}dy_k\ldots dy_1
\nonumber
\\
&
=
\underbrace{e^{L(n)n^{\alpha}\big(\Lambda_f(\hat x_1,...,\hat x_k)-\epsilon\big)}}_{\triangleq \text{(I)}_n}\
\underbrace{e^{L(n)\sum_{i=1}^{k}\big(n(\hat x_i+\delta)\big)^\alpha}}_{\triangleq \text{(II)}_n}
\underbrace{\int_{D_n^1}\dots\int_{D_n^k}e^{-\sum_{i=1}^{k}y_i}dy_k\ldots dy_1}_{\triangleq \text{(III)}_n},
\label{ineq:upsilon}
\end{align}
where the first equality is obtained by adding and subtracting $L(n)\sum_{i=1}^k Q_n^\gets(\sum_{j=1}^i y_j)^\alpha$ to the exponent of the integrand.
Note that by the construction of $D_n^j$'s,
$$Q_n\big( n(\hat x_j +2 \delta)\big) \leq y_1+\cdots + y_j \leq Q_n\big(n(\hat x_j + \delta)\big)$$
on the domain of the integral in (III)$_n$, and hence,
\begin{equation}
\text{(III)}_n
\geq e^{-Q_n(n(\hat x_k+\delta))}\prod_{i=1}^{k}\Big(Q_n\big(n(\hat x_i+\delta)\big)-Q_n\big(n(\hat x_i+2\delta)\big)\Big).\label{ineq:integrals}
\end{equation}
Since $Q_n\big(n(\hat x_k + \delta)\big) \to 0$ and ${L\left(n(\hat x_i+\delta)\right)n^{\alpha}(\hat x_i +\delta)^\alpha-L\left(n(\hat x_i +2\delta)\right)n^{\alpha}(\hat x_i +2\delta)^\alpha}\to -\infty$ for each $i$,
\begin{align}
&
\liminf_{n\rightarrow\infty}\frac{1}{L(n) n^{\alpha}}\log \text{(III)}_n
\nonumber
\\
&
\geq
\liminf_{n\rightarrow\infty}\frac{1}{L(n)n^\alpha}\Big(-Q_n\big(n(\hat x_k + \delta)\big)\Big)
+\sum_{i=1}^k\liminf_{n\rightarrow\infty}\frac{1}{L(n)n^\alpha}\log \Big(Q_n\big(n(\hat x_i+\delta)\big)-Q_n\big(n(\hat x_i +2\delta)\big)\Big)
\nonumber
\\
&=\sum_{i=1}^k\liminf_{n\rightarrow\infty}\frac{1}{L(n)n^\alpha}\log \left( ne^{-L\left(n(\hat x_i+\delta)\right)n^{\alpha}(\hat x_i+\delta)^\alpha}\left(1-e^{L\left(n(\hat x_i+\delta)\right)n^{\alpha}(\hat x_i+\delta)^\alpha-L\left(n(\hat x_i +2\delta)\right)n^{\alpha}(\hat x_i+2\delta)^\alpha}\right)\right)
\nonumber
\\
&
= \sum_{i=1}^k\liminf_{n\to\infty}\left(\frac{-L\left(n(\hat x_i+\delta)\right)n^{\alpha}(\hat x_i+\delta)^\alpha}{L(n)n^\alpha}
+
\frac{\log\left(1-e^{L\left(n(\hat x_i+\delta)\right)n^{\alpha}(\hat x_i+\delta)^\alpha-L\left(n(\hat x_i+2\delta)\right)n^{\alpha}(\hat x_i+2\delta)^\alpha}\right)}{L(n)n^\alpha}\right)
\nonumber
\\
&
=
-\sum_{i=1}^k(\hat x_i+\delta)^\alpha.\label{eq:III}
\end{align}
From this along with
\begin{align*}
\liminf_{n\rightarrow\infty}\frac{1}{n^{\alpha}L(n)}\log \text{(I)}_n
&
=
\liminf_{n\rightarrow\infty}\frac{1}{n^{\alpha}L(n)}\log \Big(e^{n^{\alpha}L(n)\big(\Lambda_f(\hat x_1,\dots,\hat x_k)-\epsilon\big)}\Big)
=
\Lambda_f(\hat x_1,\dots,\hat x_k)-\epsilon
\end{align*}
and
\begin{equation*}
\liminf_{n\rightarrow\infty}\frac{1}{n^{\alpha}L(n)}\log \text{(II)}_n
=
\liminf_{n\rightarrow\infty}\frac{1}{n^{\alpha}L(n)}\log \big(e^{L(n)\sum_{i=1}^{k}(n(\hat x_i +\delta))^\alpha} \big)
= \sum_{i=1}^{k}(\hat x_i+\delta)^\alpha,
\end{equation*}
we arrive at
\begin{equation}
\Lambda_f(\hat x) - \epsilon \leq \liminf_{n\to\infty} \frac{1}{L(n)n^\alpha}\log \Upsilon_f(n).
\end{equation}
Letting $\epsilon\to 0$, we obtain the lower bound of \eqref{ineq:bryc-liminf-limsup}.

Turning to the upper bound, consider
\begin{equation*}
D_{R,n}\triangleq \{(y_1,y_2,\ldots,y_k): Q_n^\gets(y_1)/n\le R\},
\end{equation*}
and decompose $\Upsilon_f(n)$ into two parts:
\begin{align*}
\Upsilon_f(n)
&
=
\int_{D_{R,n}}e^{L(n)n^{\alpha}f\big(Q_n^\gets(x_1)/n,\ldots ,Q_n^\gets(x_1+\cdots+x_k)/n\big)}e^{-\sum_{i=1}^{k}x_i}dx_1 \dots dx_k\nonumber
\\
&
\qquad
+
\int_{D_{R,n}^c}e^{L(n)n^{\alpha}f\big(Q_n^\gets(x_1)/n,\ldots ,Q_n^\gets(x_1+\cdots+x_k)/n\big)}e^{-\sum_{i=1}^{k}x_i}dx_1\dots dx_k.
\end{align*}
We first evaluate the integral over $D_{R,n}^c$.
Since $|f|\le M$,
\begin{align}
&
\int_{D_{R,n}^c}e^{L(n)n^{\alpha}f\big(Q_n^\gets(x_1)/n,\ldots ,Q_n^\gets(x_1+\cdots+x_k)/n\big)}e^{-\sum_{i=1}^{k}x_i}dx_1\dots dx_k
\nonumber
\\
&
=
\int e^{L(n)n^{\alpha}f\big(Q_n^\gets(x_1)/n,\ldots ,Q_n^\gets(x_1+\cdots+x_k)/n\big)}e^{-\sum_{i=1}^{k}x_i}\one{Q_n^\gets(x_1)/n > R }dx_1\dots dx_k
\nonumber
\\
&
=
\int e^{L(n)n^{\alpha}f\big(Q_n^\gets(x_1)/n,\ldots ,Q_n^\gets(x_1+\cdots+x_k)/n\big)}e^{-\sum_{i=1}^{k}x_i}\one{x_1\leq Q_n(nR)}dx_1 \dots dx_k
\nonumber
\\
&
\le
\int e^{L(n)n^{\alpha}M}e^{-\sum_{i=1}^{k}x_i}\mathbbm{1}_{\{x_1 \le Q_n(nR)\} }dx_1 \dots dx_k
\le
e^{L(n)n^{\alpha}M}\big(1-e^{-Q_n(nR)}\big)
\nonumber
\\
&
\le
e^{L(n)n^{\alpha}M}Q_n(nR)
.
\label{ineq:integral-over-DRn-complement}
\end{align}
Turning to the integral over $D_{R,n}$, fix $\epsilon>0$ and pick $\{\check x^{(1)},\ldots,\check x^{(m)}\}\subset \R_+^k$ in such a way that \allowbreak $\left\{\prod_{j=1}^k[\check x^{(l)}_j-\epsilon,\check x^{(l)}_j+\epsilon]\right\}_{l=1,\ldots,m}$ covers $A_R$. Set
\begin{equation*}
H_{R,n,l}\triangleq\Big\{(y_1,\ldots,y_k)\in \R_+^k: Q_n^\gets(y_1)/n \in \big[\check x_1^{(l)}-\epsilon,\check x_1^{(l)}+\epsilon\big],\ldots,Q_n^\gets(y_1+\ldots+y_k)/n \in \big[\check x_k^{(l)}-\epsilon,\check x_k^{(l)}+\epsilon\big]\Big\}.
\end{equation*}
Then $D_{R,n}\subseteq\bigcup_{l=1}^{m}H_{R,n,l}$, and hence,
\begin{align}\label{ineq:integral-over-DRn}
&
\int_{D_{R,n}}e^{L(n)n^{\alpha}f\big(Q_n^\gets(y_1)/n,\ldots,Q_n^\gets(y_1+\cdots +y_k)/n\big)}e^{-\sum_{i=1}^{k}y_i}dy_1\cdots dy_k\nonumber
\\
&
\le
\sum_{l=1}^{m}\int_{H_{R,n,l}}e^{L(n)n^{\alpha}f\big(Q_n^\gets(y_1)/n,...,Q_n^\gets(y_1+\ldots+y_k)/n\big)}e^{-\sum_{i=1}^{k}y_i}dy_1 \cdots dy_k\nonumber
\\
&
=
\sum_{l=1}^{m}\int_{H_{R,n,l}}e^{L(n)n^{\alpha}\Lambda_f\big(Q_n^\gets(y_1)/n,...,Q_n^\gets(y_1+\cdots+y_k)/n\big)}e^{L(n)\sum_{i=1}^{k}Q_n^\gets(\sum_{j=1}^i y_j)^\alpha-\sum_{i=1}^{k}y_i}dy_1..dy_k\nonumber
\\
&
\leq
\sum_{l=1}^{m}\int_{H_{R,n,l}}e^{L(n)n^{\alpha}\Lambda_f(\hat x_1,\hat x_2,\dots,\hat x_k)}e^{L(n)\sum_{i=1}^{k}Q_n^\gets(\sum_{j=1}^i y_j)^\alpha-\sum_{i=1}^{k}y_i}\,dy_1\cdots dy_k \nonumber \\
&
=
\sum_{l=1}^{m}\underbrace{e^{L(n)n^{\alpha}\Lambda_f(\hat x_1 ,\hat x_2 ,...\hat x_k )}\int_{ H_{R,n,l}}e^{L(n)\sum_{i=1}^{k}Q_n^\gets(\sum_{j=1}^i y_j)^\alpha-\sum_{i=1}^{k}y_i}\,dy_1\cdots dy_k }_{\triangleq \text{H}(R,n,l)},
\end{align}
where the first equality is obtained by adding and subtracting $L(n)\sum_{i=1}^kQ_n^\gets(\sum_{j=1}^i y_j)^\alpha$ to the exponent of the integrand.
Since
\begin{equation*}
Q_n^\gets({\textstyle\sum_{j=1}^i} y_j)/n \in \big[\check x_i^{(l)}-\epsilon,\,\check x_i^{(l)}+\epsilon\big]
\implies
Q_n\big(n(\check x_i^{(l)}+\epsilon)\big)\le {\textstyle \sum_{j=1}^{i}}y_j \le Q_n\big(n(\check x_i^{(l)}-\epsilon)\big),
\end{equation*}
we can bound the integral in \eqref{ineq:integral-over-DRn} as follows:
\begin{align}\label{ineq:integral-over-HRnl}
&
\int_{{H_{R,n,l}}}e^{L(n)\sum_{i=1}^{k}Q_n^\gets({\sum_{j=1}^i} y_j)^\alpha-\sum_{i=1}^{k}y_i}\,dy_1\cdots ,dy_k\nonumber\\
&
\leq
\int_{{H_{R,n,l}}}e^{L(n)\sum_{i=1}^{k}\big(n(\check x_i^{(l)}-\epsilon)\big)^\alpha-\sum_{i=1}^{k}y_i}\,dy_1 \dots \,dy_k\nonumber\\
&
\leq
\int_{{H_{R,n,l}}}e^{L(n)\sum_{i=1}^{k}\big(n(\check x_i^{(l)}-\epsilon)\big)^\alpha-Q_n\big(n(\check x_k^{(l)}+\epsilon)\big)}\,dy_1 \dots \,dy_k\nonumber\\
&
=
e^{L(n)\sum_{i=1}^{k}\big(n(\check x_i^{(l)}-\epsilon)\big)^\alpha-Q_n\big(n(\check x_k^{(l)}+\epsilon)\big)}\int_{{H_{R,n,l}}}\,dy_1 \dots \,dy_k\nonumber\\
&
=
e^{L(n)n^\alpha \sum_{i=1}^{k}(\check x_i^{(l)}-\epsilon)^\alpha-Q_n\big(n(\check x_k^{(l)}+\epsilon)\big)}\prod_{i=1}^{k}\Big(Q_n\big(n(\check x_i^{(l)}-\epsilon)\big)-Q_n\big(n(\check x_i^{(l)}+\epsilon)\big)\Big).
\end{align}
With \eqref{ineq:integral-over-DRn} and \eqref{ineq:integral-over-HRnl}, a straightforward calculation as in the lower bound leads to
\begin{align*}
&\limsup_{n\to\infty}\frac{1}{L(n)n^\alpha} \log \mathrm{H}(R,n,l)
\\
&
\leq
\limsup_{n\to\infty}\frac{1}{L(n)n^\alpha}
\log \left(e^{L(n)n^{\alpha}\Lambda_f(\hat x_1 ,\hat x_2 ,...\hat x_k )}\right)
+
\limsup_{n\to\infty}\frac{1}{L(n)n^\alpha}
\log \left(e^{L(n)n^\alpha\sum_{i=1}^{k}(\hat x_i^{(l)}-\epsilon)^\alpha-Q_n\big(n(\hat x_k^{(l)}+\epsilon)\big)}\right)
\\
&
\qquad+\sum_{i=1}^k
\limsup_{n\to\infty}\frac{1}{L(n)n^\alpha}
\log\Big(Q_n\big(n(\hat x_i^{(l)}-\epsilon)\big)-Q_n\big(n(\hat x_i^{(l)}+\epsilon)\big)\Big)
\\
&
=
\Lambda_f(\hat x_1,\ldots,\hat x_k).
\end{align*}
Therefore,
\begin{align*}
&\limsup_{n\to\infty}\frac{1}{L(n)n^\alpha} \log \Upsilon_f(n)
\\
&
=
\limsup_{n\to\infty}\frac{1}{L(n)n^\alpha}\log \left(e^{L(n)n^{\alpha}M}Q_n(nR)\right)
\vee
\max\left\{ \limsup_{n\to\infty}\frac{1}{L(n)n^\alpha} \log \mathrm H(R,n,l)\right\}
\\
&
\leq
(M-R^\alpha) \vee \Lambda_f(\hat x_1,\ldots,\hat x_k)
=
(M-R^\alpha) \vee \sup_{x\in \R_+^k} \{f(x) - \check I_k(x)\}.
\end{align*}
Since $R$ was arbitrary, we can send $R\to\infty$ to arrive at the desired upper bound of \eqref{ineq:bryc-liminf-limsup}.
\end{proof}

The following corollary is immediate from Lemma~\ref{lemma:ldp-for-k-biggest-jump-sizes} and Theorem 4.14 of \cite{ganesh2004big}.

\begin{corollary}
\label{lemma:ldp-for-k-biggest-jump-sizes-and-their-times}
$(Q_n^\gets(\Gamma_1)/n,\dots,Q_n^\gets(\Gamma_k)/n,U_1,\dots,U_k)$ satisfies a large deviation principle in  $\R_+^k\times[0,1]^k$ with speed $L(n)n^\alpha$ and the good rate function
\begin{equation}
\hat I_k(x_1,\dots,x_k,u_1,\dots,u_k)\triangleq
\begin{cases}
\sum_{i=1}^{k} x_i^\alpha & \text{if }\ x_1 \ge x_2 \ge \dots \ge x_k \ \text{ and }\ u_1,\ldots, u_k \in [0,1],\\
\infty, & otherwise.
\end{cases}
\end{equation}
\end{corollary}

Recall that $ \hat{J}_n^{\leqslant k} = \frac{1}{n} \sum_{i=1}^{k} Q_n^\gets(\Gamma_i)\I_{[U_i,1]}$ and the rate function $I_k$ defined in \eqref{eq:rate-function-for-Ik}.
We next prove a sample path LDP for $\hat{J}_n^{\leqslant k}$.
\begin{lemma}\label{lemma:sample-path-ldp-for-J-hat}
$\hat{J}_n^{\leqslant k}$ satisfies the LDP in $(\mathbb D, \mathcal T_{J_1})$ with speed $L(n)n^\alpha$ and rate function $I_k$.
\end{lemma}
\begin{proof}
First, we note that $I_k$ is indeed a rate function since the sublevel sets of $I_k$ equal the intersection between the sublevel sets of $I$ and a closed set $\D_{\leqslant k}$, and $I$ is a rate function (Lemma~\ref{theorem:lower-semi-continuity}).

Next, we prove the LDP in $\mathbb D_{\leqslant k}$ w.r.t.\ the relative topology induced by $\mathcal T_{J_1}$.
(Note that $I_k$ is a rate function in $\D_{\leqslant k}$ as well.)
Set $T_k(x,u) \triangleq \sum_{i=1}^k x_i \I_{[u_i,1]}$.
Since
$$
\inf_{(x,u)\in T_k^{-1}(\xi)} \hat I_k(x,u) = I_k(\xi)
$$
for $\xi\in \mathbb D_{\leqslant k}$,
the LDP in $\mathbb D_{\leqslant k}$ is established once we show that
for any closed set $F\subseteq \mathbb D_{\leqslant k}$,
\begin{equation}\label{ineq:sample-path-ldp-J-upper-bound}
\limsup_{n\to\infty} \frac{1}{L(n)n^\alpha}\log
\P \left( \hat{J}_n^{\leqslant k} \in F \right)
\leq
-\inf_{(x,u)\in T_k^{-1}(F)}\hat I_k(x,u),
\end{equation}
and for any open set $G\subseteq \mathbb D_{\leqslant k}$,
\begin{equation}\label{ineq:sample-path-ldp-J-lower-bound}
-\inf_{(x,u)\in T_k^{-1}(G)}\hat I_k(x,u)
\leq
\liminf_{n\to\infty} \frac{1}{L(n)n^\alpha}\log
\P \left( \hat{J}_n^{\leqslant k} \in G \right).
\end{equation}
We start with the upper bound. Note that
\begin{align*}
&
\limsup_{n\to\infty} \frac{1}{L(n)n^\alpha}\log
\P \left( \hat{J}_n^{\leqslant k} \in F \right)
\\
&
=
\limsup_{n\to\infty} \frac{1}{L(n)n^\alpha}\log
\P \Big( \big(Q_n^\gets(\Gamma_1),\ldots,Q_n^\gets(\Gamma_k), U_1,\ldots,U_k\big) \in T_{k}^{-1} \left( F \right) \Big)
\\
&
\leq
\limsup_{n\to\infty} \frac{1}{L(n)n^\alpha}\log
\P \left( \big(Q_n^\gets(\Gamma_1),\ldots,Q_n^\gets(\Gamma_k), U_1,\ldots,U_k\big) \in T_{k}^{-1} \left( F \right)^-\,\right)
\\
&\leq
-\inf_{(x_1,\ldots, x_k,u_1,\ldots, u_k)\in   T_{k}^{-1} \left( F \right)^-} \hat I_k(x_1,\ldots,x_k,u_1,\ldots,u_k).
\end{align*}
In view of \eqref{ineq:sample-path-ldp-J-upper-bound}, it is therefore enough for the upper bound to show that
\begin{equation}\nonumber
\inf_{(x,u)\in T_{k}^{-1} \left(F \right)}\hat I_k(x,u)\leq \inf_{(x,u)\in T_{k}^{-1} \left(F \right)^-}\hat I_k(x,u).
\end{equation}
To prove this, we proceed with proof by contradiction.
Suppose that
\begin{equation}\label{ineq:inf-equal-inf-over-closer-contradiction}
c \triangleq \inf_{(x,u)\in T_{k}^{-1} \left(F \right)} \hat I_k(x,u) > \inf_{(x,u)\in T_{k}^{-1} \left(F \right)^-}\hat I_k(x,u).
\end{equation}
Pick an $\epsilon>0$ in such a way that  $\inf_{(x,u) \in T^{-1}_k(F)^-}\hat I_k(x,u)< c-2\epsilon $.
Then there exists $(x^*,u^*)\in T_k^{-1}(F)^-$
 such that $\hat I_k(x^*,u^*) < c-2\epsilon$.
Let $\bar I_k(x_1,\ldots,x_k,u_1,\ldots,u_k) \triangleq \sum_{i=1}^k x_i^\alpha$.
Since $\bar I_k$ is continuous, one can find $(x',u') = (x_1',\ldots,x_k', u_1',\ldots,u_k')\in T_k^{-1}(F)$ sufficiently close to $(x^*,u^*)$ so that $\bar I_k(x',u') < c-\epsilon$.
Note that for any permutation $p:\{1,\ldots, k\}\to\{1,\ldots,k\}$, $(x'',u'')\triangleq (x_{p(1)}',\ldots, x_{p(k)}',u_{p(1)}',\ldots,u_{p(k)}')$ also belongs to $T_k^{-1}(F)$ and $\bar I_k(x'',u'') = \bar I_k(x',u')$ due to the symmetric structure of $T_k$ and $\bar I_k$.
If we pick $p$ so that $x'_{p(1)}\geq \cdots\geq x'_{p(k)}$, then $\hat I_k(x'',u'') = \bar I_k(x',u') < c-\epsilon \leq \inf_{(x,u)\in T_k^{-1}(F)}\hat I_k(x,u)$, which contradicts to $(x'',u'')\in T_k^{-1}(F)$.
Therefore, \eqref{ineq:inf-equal-inf-over-closer-contradiction} cannot be the case, which proves the upper bound.

Turning to the lower bound, consider an open set $G\subseteq \mathbb D_{\leqslant k}$.
\begin{align*}
&
\liminf_{n\to\infty} \frac{1}{L(n)n^\alpha}\log
\P \left( \hat{J}_n^{\leqslant k} \in G \right)
\\
&
=
\liminf_{n\to\infty} \frac{1}{L(n)n^\alpha}\log
\P \Big( \big(Q_n^\gets(\Gamma_1),\ldots,Q_n^\gets(\Gamma_k), U_1,\ldots,U_k\big) \in T_{k}^{-1} \left( G \right) \Big)
\\
&
\geq
\liminf_{n\to\infty} \frac{1}{L(n)n^\alpha}\log
\P \left( \big(Q_n^\gets(\Gamma_1),\ldots,Q_n^\gets(\Gamma_k), U_1,\ldots,U_k\big) \in  T_{k}^{-1} \left( G \right)^\circ\,\right)
\\
&\geq
-\inf_{(x_1,\ldots, x_k,u_1,\ldots, u_k)\in   T_{k}^{-1} \left( G \right)^\circ}\hat I_k(x_1,\ldots,x_k,u_1,\ldots,u_k).
\end{align*}
In view of \eqref{ineq:sample-path-ldp-J-lower-bound}, we are done if we prove that
\begin{equation}\label{ineq:sample-path-ldp-J-inf-over-interior-leq-inf}
\inf_{(x,u) \in T_k^{-1}(G)^{\circ}}\hat I_k(x,u)\leq \inf_{(x,u) \in T_k^{-1}(G)}\hat I_k(x,u).
\end{equation}
Let $(x,u)$ be an arbitrary point in $T_k^{-1}(G)$ so that $T_k(x,u) \in G$.
We will show that there exists $(x^*, u^*)\in T_k^{-1}(G)^\circ$ such that $I_k(x^*,u^*) \leq I_k(x,u)$.
Note first that if $u_i\in \{0,1\}$ for some $i$, then $x_i$ has to be 0 since $G \subseteq \mathbb D_{\leqslant k}$.
This means that we can replace $u_i$ with an arbitrary number in $(0,1)$ without changing the value of $I_k$ and $T_k$.
Therefore, we assume w.l.o.g.\ that $u_i>0$ for each $i=1,\ldots,k$.
Now, suppose that $u_i = u_j$ for some $i\neq j$.
Then one can find $(x', u')$ such that $T_k(x',u')=T_k(x,u)$ by setting
$$
(x',u')
\triangleq
(
x_1,
\ldots,
\underbrace{x_i+x_j}_{\mathclap{i\text{\textsuperscript{th} coordinate}\qquad}},
\ldots,
\underbrace{0_{\color{white}j}}_{\mathclap{\quad j^\text{th}\text{ coordinate}}},
\ldots,
x_k,
u_1,
\ldots,
\underbrace{u_{i\color{white}j}}_{\mathclap{ k+i\text{\textsuperscript{th} coordinate}\qquad}},
\ldots,
\underbrace{u_j'}_{\mathrlap{k+j^\text{th} \text{coordinate}}\qquad},
\ldots,
u_k
),
$$
where $u_j'$ is an arbitrary number in $(0,1)$; in particular, we can choose $u_j'$ so that $u_j'\neq u_l$ for $l=1,\ldots,k$.
It is easy to see that $\bar I_k(x',u') \leq \hat I_k(x,u)$.
Now one can permute the coordinates of $(x',u')$ as in the upper bound to find $(x'', u'')$ such that $T_k(x'',u'') = T_k(x,u)$ and $\hat I_k(x'',u'') \leq \hat I_k(x,u)$.
Iterating this procedure until there is no $i\neq j$ for which $u_i=u_j$, we can find $(x^*, u^*)$ such that $T_k(x^*, u^*) = T_k(x,u)$, $u_i^*$'s are all distinct in $(0,1)$, and $ I_k(x^*,u^*)\leq I_k(x,u)$.
Note that since $T_k$ is continuous at $(x^*, u^*)$, $T_k(x^*,u^*)\in G$, and $G$ is open, we conclude that $(x^*, u^*)\in T_k^{-1}(G)^\circ$.
Therefore,
$$\inf_{(x,u)\in T_k^{-1}(G)^\circ} I_k(x,u) \leq I_k(x,u).$$
Since $(x,u)$ was arbitrarily chosen in $T_k^{-1}(G)$, \eqref{ineq:sample-path-ldp-J-inf-over-interior-leq-inf} is proved.
Along with the upper bound, this proves the LDP in $\mathbb D_{\leqslant k}$.
Finally, since $\mathbb D_{\leqslant k}$ is a closed subset of $\mathbb D$, $\P(\hat J_n^{\leqslant k}\notin \mathbb D_{\leqslant k}) = 0$, and $I_k = \infty$ on $\mathbb D\setminus \mathbb D_{\leqslant k}$, Lemma 4.1.5 of \cite{dembo2010large} applies, proving the desired LDP in $\mathbb D$.
\end{proof}

Now we are ready to prove Lemma~\ref{lemma:equivalence-J}.

\begin{proof}[Proof of Lemma~\ref{lemma:equivalence-J}]
Recall that
\begin{align*}
\bar J_n^k
\stackrel{\mathcal D}{=}
\underbrace{\frac{1}{n} \sum_{i=1}^{k} Q_n^\gets(\Gamma_i)\I_{[U_i,1]}}_{= \hat{J}_n^{\leqslant k}}-\underbrace{\frac{1}{n}\I\{\tilde N_n<k\}\sum_{i=\tilde N_n+1}^{k}Q_n^\gets(\Gamma_i)\I_{[U_i,1]}}_{= \check{J}_n^{\leqslant k}}.
\end{align*}
Let $F$ be a closed set and note that
\begin{align*}
\P( \bar{J}_n^k \in F )
&
=
\P( \hat{J}_n^{\leqslant k}-\check{J}_n^{\leqslant k} \in F )
\leq
\P\big(\hat{J}_n^{\leqslant k}-\check J_n^{\leqslant k}\in F , \I\{N(n) < k\}=0\big)+\P\big(\I\{N(n)<k\} \neq 0\big)
\\
&
\leq \P( \hat J_n^{\leqslant k} \in F)+\P(N(n)<k).
\end{align*}
From Lemma~\ref{lemma:sample-path-ldp-for-J-hat},
\begin{align*}
\limsup_{n \rightarrow \infty}\frac{\log \P(\bar{J}_n^k\in F)}{L(n)n^\alpha}
&
\leq
\limsup_{n \rightarrow \infty}\frac{\log \P(\hat {J}_n^{\leqslant k}\in F)}{L(n)n^\alpha}
\vee
\limsup_{n \rightarrow \infty}\frac{\log \P(N(n) < k)}{L(n)n^\alpha}
\\
&
\leq
-\inf_{\xi \in F}I_k(\xi),
\end{align*}
since $\limsup_{n\rightarrow \infty}\frac{1}{L(n)n^\alpha}\log\P(N(n)<k)= -\infty$.

Turning to the lower bound, let $G$ be an open set.
Since the lower bound is trivial in case $\inf_{x\in G}I_k(x)=\infty$, we focus on the case $\inf_{x\in G}I_k(x)<\infty$.
In this case,
\begin{align*}
\liminf_{n\rightarrow \infty}\frac{\log\P(\bar{J}_n^k \in G) }{L(n)n^\alpha}
&\geq
\liminf_{n\rightarrow \infty}\frac{\log\P(\bar{J}_n^k \in G,\,N(n) \geq k)}{L(n)n^\alpha}
=
\liminf_{n\rightarrow \infty}\frac{\log\P(\hat{J}_n^{\leqslant k} \in G,\,N(n)\geq k)}{L(n)n^\alpha}
\\
&
\geq
\liminf_{n\rightarrow  \infty}\frac{1}{L(n)n^\alpha}\log\left(\P(\hat{J}_n^{\leqslant k} \in G)-\P(N(n)< k)\right)
\\
&=
\liminf_{n\rightarrow \infty}\frac{1}{L(n)n^\alpha}\log\left(\P(\hat{J}_n^{\leqslant k} \in G)\left(1-\frac{\P(N(n)< k)}{\P(\hat{J}_n^{\leqslant k}\in G)}\right)\right)
\\
&=
\liminf_{n\rightarrow \infty}\frac{1}{L(n)n^\alpha}\left\{\log\left(\P(\hat{J}_n^{\leqslant k} \in G)\right)
+\log\left(1-\frac{\P(N(n) < k)}{\P(\hat{J}_n^{\leqslant k}\in G)}\right)\right\}
\\
&
=
\liminf_{n\rightarrow \infty}\frac{1}{L(n)n^\alpha}\log\P(\hat{J}_n^{\leqslant k} \in G)\geq -\inf_{\xi \in G}I_k(\xi).
\end{align*}
The last equality holds since
\begin{equation}\label{eq:lowerbound-because-ratio-zero}
\lim_{n\to\infty}\frac{\P(N(n) < k)}{\P(\hat{J}_n^{\leqslant k} \in G)}
=
\lim_{n\to\infty}\left\{\text{exp}\left(\frac{\log \P(N(n) < k)}{L(n)n^\alpha}-\frac{\log\P(\hat{J}_n^{\leqslant k}\in G)}{L(n)n^\alpha}\right)\right\}^{L(n)n^\alpha} = 0,
\end{equation}
which in turn follows from
$$
\limsup_{n\to\infty}\frac{1}{L(n)n^\alpha}\log \P(N(n)< k) = -\infty
$$
and
$$
\limsup_{n\to\infty}\frac{-1}{L(n)n^\alpha}\log\P(\hat{J}_n^{\leqslant k} \in G) \leq \inf_{x\in G}I_k(x)<\infty.
$$

\end{proof}

%
%
\subsection{Proof of Lemma~\ref{lemma:Jk-approximates-J}}
\begin{proof}[Proof of Lemma~\ref{lemma:Jk-approximates-J}]
Since the inequality is obvious if $\inf_{\xi \in G} I(\xi) = \infty$, we assume that $\inf_{\xi \in G} I(\xi)  < \infty$.
Then, there exists a $\xi_0 \in G$ such that $I(\xi_0) \le \inf_{\xi \in G}I(\xi)+\delta$.
Since $G$ is open, we can pick $\epsilon >0 $ such that  $B_{J_1}(\xi_0;2\epsilon) \subseteq G$ so that $B_{J_1}(\xi_0; \epsilon) \subseteq G^{-\epsilon}$.
Note that since $I(\xi_0) < \infty$, $\xi_0$ has the representation
$\xi_0 = \sum_{i=1}^\infty x_i\I_{[u_i,1]}$
where $x_i\geq 0$ for all $i=1,2,\ldots$, and $u_i$'s all distinct in $(0,1)$.
Note also that since $I(\xi_0) = \sum_{i=1}^\infty x_i^\alpha < \infty$ with $\alpha<1$, $\sum_{i=1}^\infty x_i$ has to be finite as well.
There exists $K$ such that $k\geq K$ implies $\sum_{i=k+1}^\infty x_i <\epsilon$.
For these $\epsilon$ and $K$, we claim that \eqref{inf-G-epsilon-delta} holds.
For any given $k\geq K$, let $\xi_1 \triangleq \sum_{i=1}^k x_i \I_{[u_i,1]}$, then $I_k(\xi_1) \leq I(\xi_0)$ while $d_{J_1}(\xi_0, \xi_1) \leq \|\xi_0 - \xi_1\|_\infty \leq \sum_{i=k+1}^\infty x_i < \epsilon$.
That is, $\xi_1 \in B_{J_1}(\xi_0; \epsilon) \subseteq G^{-\epsilon}$.
Therefore, $\inf_{\xi\in G^{-\epsilon}}I_k(\xi) \leq I(\xi_1) \leq I(\xi_0) \leq \inf_{\xi\in G} I(\xi) + \delta$.
\end{proof}

%
%
\subsection{Proof of Lemma~\ref{lemma:negligence-K}}
In our proof of Lemma~\ref{lemma:negligence-K}, the following lemmas (Lemma~\ref{lemma:negligence-positive} and Lemma~\ref{lemma:negligence-negative}) play key roles.
\begin{lemma}\label{lemma:negligence-positive}
For each $\epsilon>\delta>0$,
\begin{align}\label{ineq:concentration}
&\limsup_{n\to\infty}\frac1{L(n)n^\alpha}\log  \P\left(\max_{1\leq j \leq 2n}\sum_{i=1}^{j}\big(Z_i\mathbbm{1}_{\{Z_i \le n\delta \}}-\E Z\big)> n\epsilon\right)
\le -(\epsilon/3)^\alpha(\epsilon/\delta)^{1-\alpha}.
\end{align}
\end{lemma}

\begin{proof}
We refine an argument developed in \cite{jelenkovic2003large}.
Note that for any $s>0$ such that $1/s \leq n\delta$,
\begin{equation}\label{eq:decomposition-of-truncated-exponential-moment}
\mathbf{E}e^{sZ \mathbbm{1}_{\{Z \le n\delta \}}}
= \mathbf{E}e^{sZ \mathbbm{1}_{\{Z \le n\delta \}}}\mathbbm{1}_{\{Z \geq \frac 1 s \}}
+\mathbf{E}e^{sZ \mathbbm{1}_{\{Z \le n\delta \}}}\mathbbm{1}_{\{Z < \frac 1 s \}}
=(I) + (II),
\end{equation}
and\begin{align}
(I)
&=
\int_{[1/s,n\delta]}e^{sy} \d\P( Z\leq y) + \int_{(n\delta,\infty)} \d \P(Z \leq y)
\nonumber
\\
&=
\Big[e^{sy}\P(Z\leq y)\Big]^{(n\delta)+}_{(1/s)-} - s\int_{[1/s,n\delta]} e^{sy} \P(Z\leq y) \d y + \P(Z > n\delta)
\nonumber
\\
&=
e^{sn\delta}\P(Z\leq n\delta)-e\P(Z< 1/s) -  s\int_{[1/s,n\delta]} e^{sy}  \d y + s\int_{[1/s,n\delta]} e^{sy} \P(Z> y) \d y + \P(Z > n\delta)
\nonumber
\\
&=
e^{sn\delta}\P(Z\leq n\delta)-e\P(Z< 1/s) -  e^{sn\delta} + e + s\int_{[1/s,n\delta]} e^{sy} \P(Z> y) \d y + \P(Z > n\delta)
\nonumber
\\
&=
-e^{sn\delta}\P(Z> n\delta)+e\P(Z\geq 1/s) + s\int_{[1/s,n\delta]} e^{sy} \P(Z> y) \d y + \P(Z > n\delta)
\nonumber
\\
&\leq
s\int_{[1/s,n\delta]} e^{sy} \P(Z> y) \d y+ e\P(Z\geq 1/s) + \P(Z > n\delta)
\nonumber
\\
&\leq
s\int_{[1/s,n\delta]} e^{sy} \P(Z> y) \d y+ s^2(e+1) \E Z^2,
\label{ineq:I}
\end{align}
where the last inequality is from $\P(Z\geq n\delta) \leq \P(Z\geq 1/s) \leq s^2 \,\E Z^2$;
while
\begin{equation}\label{ineq:II}
(II)
\leq \int_0^{1/s} e^{sy} \d \P(Z \leq y)
\leq \int_0^{1/s} (1+sy+(sy)^2) \d \P(Z \leq y)
\leq 1+s\,\mathbf{E}Z +s^2\,\mathbf{E} Z^2.
\end{equation}
Therefore, from \eqref{eq:decomposition-of-truncated-exponential-moment}, \eqref{ineq:I} and \eqref{ineq:II}, if $1/s \leq n\delta$ and $s$ is sufficiently small,
\begin{align}\label{eq:trunc1}
\mathbf{E} e^{sZ \mathbbm{1}_{\{Z \le n\delta \}}}
&
\le s\int_{\frac 1 s }^{n\delta}e^{sy}\P\left(Z>y\right)dy
+1+s\E Z+s^2(e+2)\E Z^2 \nonumber \\
&
= s\int_{\frac 1 s }^{n\delta}e^{sy-q(y)}dy
+1+s\mathbf{E}Z +s^2(e+2)\E Z^2\nonumber \\
&
\le sn\delta e^{sn\delta-q(n\delta)}+1+s\mathbf{E} Z +s^2(e+2)\mathbf{E}Z^2,
\end{align}
where $q(x) \triangleq -\log \P(X > x) = L(x)x^\alpha$, and the last inequality is from the fact that $e^{sy-q(y)}$ is increasing over $[1/s,n\delta]$ due to the assumption that $L(y)y^{\alpha-1}$ is non-increasing for sufficiently large $y$'s.
Now, from the Markov inequality,
\begin{align}
&
\mathbf{P}\left(\sum_{i=1}^{j}\big(Z_i\mathbbm{1}_{\{Z_i \le n\delta \}}-\E Z\big) > n\epsilon\right)
\nonumber
\\
&
\le
\mathbf{P}\left(\exp\left(s\sum_{i=1}^{j}Z_i\mathbbm{1}_{\{Z_i \le n\delta \}}\right)>\exp\big(s(n\epsilon + j\E Z)\big)\right)
\nonumber
\\
&
\le
\exp\left\{-s\big(n\epsilon+j\mathbf{E}Z\big)+j\log\left(\E e^{s Z\mathbbm 1_{\{Z\leq n\delta\}}}\right)\right\}\nonumber
\\
&
\le
\exp\left\{-s(n\epsilon+j\mathbf{E}Z)+j\left(
sn\delta e^{sn\delta-q(n\delta)}+s\mathbf{E} Z +s^2(e+2)\mathbf{E}Z^2
 \right)\right\}\nonumber\\
&
=
\exp\left\{-sn\epsilon+jsn\delta e^{sn\delta-q(n\delta)}+js^2(e+2)\mathbf{E}Z^2
\right\}
\nonumber\\
&
\leq \exp\left\{-sn\epsilon+2n^2 s\delta e^{sn\delta-q(n\delta)}+2ns^2(e+2)\mathbf{E}Z^2
\right\}\label{eq:auxnot}
\end{align}
for $j \leq 2n$, where the third inequality is from \eqref{eq:trunc1} and the generic inequality $\log (x+1) \le x$.
Fix $\gamma\in (0, (\epsilon/\delta)^{1-\alpha})$ and and set $s =\frac{\gamma q(n\epsilon)}{n\epsilon}$.
Note that $1/s\to \infty$ as $n\to \infty$, while $1/s \leq n\delta$ for sufficiently large $n$.
From now on, we only consider sufficiently large $n$'s such that $1/s < n\delta$ and $s$ is sufficiently small so that \eqref{eq:trunc1} and \eqref{eq:auxnot} are valid.
To establish an upper bound for \eqref{eq:auxnot}, we next examine $e^{sn\delta-q(n\delta)}$.
Note that 
$q(n\epsilon) = q(n\delta) \frac{L(n\epsilon)}{L(n\delta)}(\delta/\epsilon)^{-\alpha}$,
and hence,
\begin{equation*}
sn\delta-q(n\delta)
=
\frac{\gamma q(n\epsilon)}{n\epsilon} n\delta-q(n\delta)
=
-q(n\delta)\left(1-\gamma\frac{L(n\epsilon)}{L(n\delta)}(\delta/\epsilon)^{1-\alpha} \right),
\end{equation*}
and
\begin{equation}\label{eq:tconc1}
e^{sn\delta-q(n\delta)} \leq e^{-q(n\delta)\left(1-\gamma\frac{L(n\epsilon)}{L(n\delta)}(\delta/\epsilon)^{1-\alpha}\right)}.
\end{equation}
Plugging this $s \big(=\frac{\gamma q(n\epsilon)}{n\epsilon}\big)$ into \eqref{eq:auxnot} along with \eqref{eq:tconc1},
\begin{align*}
&
\max_{0\leq j \leq 2n}\mathbf{P}\left(\sum_{i=1}^{j}\big(Z_i\mathbbm{1}_{\{Z_i \le n\delta \}}-\E Z\big) > n\epsilon\right)
\\
&
\leq
\exp\left\{
-\gamma q(n\epsilon) +
\frac{2 \gamma \delta  n q(n\epsilon)}{\epsilon} e^{-q(n\delta)\left(1-\gamma\frac{L(n\epsilon)}{L(n\delta)}(\delta/\epsilon)^{1-\alpha}\right)}
+ \frac{2\gamma^2 (e+2)\E Z^2 }{\epsilon^2} \frac{q(n\epsilon)^2}{n}
\right\}.
\end{align*}
Since
$$\limsup_{n\to\infty}\frac1{L(n)n^\alpha}\frac{2 \gamma \delta  n q(n\epsilon)}{\epsilon} e^{-q(n\delta)\left(1-\gamma\frac{L(n\epsilon)}{L(n\delta)}(\delta/\epsilon)^{1-\alpha}\right)} = 0,$$
and
$$\limsup_{n\to\infty} \frac1{L(n)n^\alpha}\frac{2\gamma^2 (e+2)\E Z^2 }{\epsilon^2} \frac{q(n\epsilon)^2}{n} = 0,$$
we conclude that
$$
\limsup_{n\to\infty} \frac1{L(n)n^\alpha} \log \max_{0\le j \le 2n}\mathbf{P}\left(\sum_{i=1}^{j}\big(Z_i\mathbbm{1}_{\{Z_i \le n\delta \}}-\E Z\big) > n\epsilon\right) = \limsup_{n\to\infty} \frac{-\gamma q(n\epsilon)}{L(n)n^\alpha} = -\epsilon^\alpha \gamma.
$$
From Etemadi's inequality,
\begin{align*}
&\limsup_{n\to\infty} \frac1{L(n)n^\alpha} \log
\mathbf{P}\left(\max_{0\leq j\leq 2n}\sum_{i=1}^{j}\big(Z_i\mathbbm{1}_{\{Z_i \le n\delta \}}-\E Z\big) > 3n\epsilon\right)
\\
&\leq
\limsup_{n\to\infty} \frac1{L(n)n^\alpha} \log \left\{
3\max_{0\leq j\leq 2n}\P\left(\sum_{i=1}^{j}\big(Z_i\mathbbm{1}_{\{Z_i \le n\delta \}}-\E Z\big) > n\epsilon\right) \right\}
=
-\epsilon^\alpha \gamma.
\end{align*}
Since this is true for arbitrary $\gamma$'s such that $\gamma\in (0, (\epsilon/\delta)^{1-\alpha})$, we arrive at the conclusion of the lemma.

\end{proof}

\begin{lemma}\label{lemma:negligence-negative}
For every $\epsilon, \, \delta >0 $,
\begin{equation*}
\limsup_{n\rightarrow \infty}\frac{1}{L(n)n^\alpha}\log\P\left(\sup_{1 \le j \le 2n}\sum_{i=1}^{j}\left(\mathbf{E}Z-{Z_i}{\mathbbm{1}_{\{Z_i\le n\delta\}}}\right)> n\epsilon\right)
=-\infty.
\end{equation*}
\end{lemma}
\begin{proof}
Note first that there is $n_0$ such that $\mathbf{E}\left({Z_i}{\mathbbm{1}_{\{Z_i > n\delta\}}}\right) \le \frac{\epsilon}{3}$ for $n\geq n_0$.
For $n\geq n_0$ and $j \leq 2n$,
\begin{align*}
\mathbf{P}\left(\sum_{i=1}^{j}\left(\mathbf{E}{Z}-{Z_i}{\mathbbm{1}_{\{Z_i\le n\delta\}}}\right)>{n\epsilon}\right)
&
=
\mathbf{P}\left(\sum_{i=1}^{j}\left(\mathbf{E}{Z}\mathbbm{1}_{\{Z\le n\delta\}}-{Z_i}{\mathbbm{1}_{\{Z_i\le n\delta\}}}\right)>n\epsilon - j\E Z\mathbbm 1_{\{Z > n\delta\}}\right)
\\
&
\leq
\mathbf{P}\left(\sum_{i=1}^{j}\left(\mathbf{E}{Z}\mathbbm{1}_{\{Z\le n\delta\}}-{Z_i}{\mathbbm{1}_{\{Z_i\le n\delta\}}}\right)>n\epsilon - j\epsilon/3\right)
\\
&
\leq
\mathbf{P}\left(\sum_{i=1}^{j}\left(\mathbf{E}{Z}\mathbbm{1}_{\{Z\le n\delta\}}-{Z_i}{\mathbbm{1}_{\{Z_i\le n\delta\}}}\right)>\frac{n\epsilon}3\right).
\end{align*}
Let $Y_i^{(n)}\triangleq \mathbf{E}({Z_i}{\mathbbm{1}_{\{Z_i\le n\delta\}}})-{Z_i}{\mathbbm{1}_{\{Z_i\le n\delta\}}}$. Recall the definition of $Z$ in Section~\ref{subsection:extended-ldp-1-d-levy-processes} and note that it is bounded from below. Furthermore, $\E Y_i^{(n)} = 0$, $\var Y_i^{(n)}\leq \E Z^2$, and $Y_i^{(n)}\le \mathbf{E}Z$ almost surely.
From Bennet's inequality,
\begin{align}\label{eq:step2}
& \mathbf{P}\left( \sum_{i=1}^{j}\left(\mathbf{E}Z\mathbbm{1}_{\{Z_i\le n\delta\}}-{Z_i}{\mathbbm{1}_{\{Z_i\le n\delta\}}}\right)>\frac{n\epsilon}{3}\right)\nonumber
\\
&		
\le
\exp\left[-\frac{j\var Y^{(n)}}{(\E Z)^2}\left\{\left(1+\frac{n\epsilon\,\E Z}{3j\,\var Y^{(n)}}\right)\log \left( 1+\frac{n\epsilon\,\E Z}{3j\,\var Y^{(n)}} \right) -\left(\frac{n\epsilon\,\E Z}{3j\,\var Y^{(n)}}\right) \right\}\right]
\nonumber
\\
&		
\le
\exp\left[-\frac{j\var Y^{(n)}}{(\E Z)^2}\left\{\left(\frac{n\epsilon\,\E Z}{3j\,\var Y^{(n)}}\right)\log \left( 1+\frac{n\epsilon\,\E Z}{3j\,\var Y^{(n)}} \right) -\left(\frac{n\epsilon\,\E Z}{3j\,\var Y^{(n)}}\right) \right\}\right]
\nonumber
\\
&
\le
\exp\left[-\left\{\left(\frac{n\epsilon}{3\,\mathbf{E}Z}\right)\log \left( 1+\frac{n\epsilon\,\mathbf{E}Z}{3j\,\var Y^{(n)}} \right) -\left(\frac{n\epsilon}{3\,\mathbf{E}Z}\right) \right\}\right]
\nonumber
\\
&
\le
\exp\left[-n\left\{\left(\frac{\epsilon}{3\,\mathbf{E}Z}\right)\log \left( 1+\frac{\epsilon\,\mathbf{E}Z}{6\,\E Z^2} \right) -\left(\frac{\epsilon}{3\,\mathbf{E}Z}\right) \right\}\right]
\nonumber
\end{align}
for $j \leq 2 n$.
Therefore, for $n\geq n_0$ and $j\leq 2n$,
\begin{equation*}
\mathbf{P}\left(\sum_{i=1}^{j}\left(\mathbf{E}{Z}-{Z_i}{\mathbbm{1}_{\{Z_i\le n\delta\}}}\right)>{n\epsilon}\right)
\le
\exp\left[-n\left\{\left(\frac{\epsilon}{3\,\mathbf{E}Z}\right)\log \left( 1+\frac{\epsilon\,\mathbf{E}Z}{6\,\E Z^2} \right) -\left(\frac{\epsilon}{3\,\mathbf{E}Z}\right) \right\}\right].
\end{equation*}
Now, from Etemadi's inequality,
\begin{align*}
\limsup_{n\rightarrow \infty}&\frac{1}{L(n)n^\alpha}\log\P\left(\sup_{1 \le j \le 2n}\sum_{i=1}^{j}\left(\mathbf{E}Z-{Z_i}{\mathbbm{1}_{\{Z_i\le n\delta\}}}\right)> 3n\epsilon\right)
\\
&
\le
\limsup_{n\rightarrow \infty}\frac{1}{L(n)n^\alpha}\log \left\{3\max_{1\le j \le 2n}\P\left(\sum_{i=1}^{j}\left(\mathbf{E}Z-{Z_i}{\mathbbm{1}_{\{Z_i\le n\delta\}}}\right)> n\epsilon\right)
\right\}
\\
&
\le
\limsup_{n\rightarrow \infty}\frac{1}{L(n)n^\alpha}\log \left\{3
\exp\left[-n\left\{\left(\frac{\epsilon}{3\,\mathbf{E}Z}\right)\log \left( 1+\frac{\epsilon\,\mathbf{E}Z}{6\,\E Z^2} \right) -\left(\frac{\epsilon}{3\,\mathbf{E}Z}\right) \right\}\right]\right\}
=-\infty.
\end{align*}
Replacing $\epsilon$ with $\epsilon/3$, we arrive at the conclusion of the lemma.
\end{proof}

Now we are ready to prove Lemma~\ref{lemma:negligence-K}.
\begin{proof}[Proof of Lemma~\ref{lemma:negligence-K}]
\begin{align}\label{eq:bbbb}
&
\P\left(\|\bar H_n^k\|_\infty > \epsilon\right)
\nonumber
\\
&
\le
\P\left(\|\bar H_n^k\|_\infty > \epsilon, N(nt) \geq k\right)
+ \P\left(\|\bar H_n^k\|_\infty > \epsilon, N(nt) < k\right)
\nonumber
\\
&
\le
\P\left(\|\bar H_n^k\|_\infty > \epsilon,\, N(nt) \geq k, \, Z_{R_n^{-1}(k)} \leq n\delta \right)
\\
&\qquad\qquad\qquad\qquad
+\P\left(\|\bar H_n^k\|_\infty > \epsilon,\,N(nt) \geq k,\, Z_{R_n^{-1}(k)} > n\delta \right)
+ \P\left(N(nt) < k\right)
\nonumber
\\
&
\le
\P\left(\|\bar H_n^k\|_\infty > \epsilon, \,N(nt) \geq k,\, Z_{R_n^{-1}(k)} \leq n\delta \right)
\\
&\qquad\qquad\qquad\qquad
+\P\left(N(nt) \geq k,\,Z_{R_n^{-1}(k)} > n\delta \right)
+ \P\left(N(nt) < k\right)
.\nonumber
\end{align}
An explicit upper bound for the second term can be obtained:
\begin{align*}
\P\left(N(nt) \geq k,\,Z_{R_n^{-1}(k)} > n\delta \right)
&
\leq \mathbf{P} \left( Q_n^\gets(\Gamma_k) > n\delta\right) \leq {\mathbf{P} \left( Q_n^\gets(\Gamma_k) \geq n\delta\right)}\nonumber
= \mathbf{P}\left(\Gamma_k\le Q_n(n\delta)\right)
\\
&
= \int_{0}^{Q_n(n\delta)}\frac{1}{k!}t^{k-1}e^{-t}dt
= \int_{0}^{nv[n\delta,\infty)}\frac{1}{k!}t^{k-1}e^{-t}dt\nonumber
\leq \int_{0}^{nv(n\delta,\infty)}t^{k-1}dt
\\
&
= \frac{1}{k}n^{k}e^{-kL(n\delta)n^\alpha\delta^\alpha}.
\end{align*}
Therefore,
\begin{equation}\label{eq:log-asymptotics-for-QnGammak}
\limsup_{n\to\infty} \frac1{L(n)n^\alpha}\log \P \left( Q_n^\gets(\Gamma_k) \ { \geq } \ n\delta\right)
\leq
-k\delta^\alpha.
\end{equation}
Turning to the first term of \eqref{eq:bbbb}, we consider the following decomposition:
\begin{align*}
&\mathbf{P}\left(\|\bar H_n^k\|_\infty > \epsilon,\,N(nt) \geq k,\,Z_{R_n^{-1}(k)} \le n\delta\right)
\\
&
=
\underbrace{\P\left(N(nt) \geq k,\,Z_{R_n^{-1}(k)} \le n\delta,\sup_{t\in[0,1]}\bar H_n^k(t) > \epsilon\right)}_{\triangleq \text{(i)}}
+
\underbrace{\P\left(N(nt) \geq k,\,Z_{R_n^{-1}(k)} \le n\delta,\sup_{t\in[0,1]}-\bar H_n^k(t) > \epsilon\right)}_{\triangleq \text{(ii)}}.
\end{align*}
Since $Z_{R_n^{-1}(k)}\leq n\delta$ implies $\one{R_n(i)>k} \leq \one{Z_i \leq n\delta}$,
\begin{align*}
\text{(i)}&
\leq
\P\left( \sup_{t\in[0,1]} \sum_{i=1}^{N(nt)} \big(Z_i \one{R_n(i) >k} - \E Z\big) > n\epsilon,\,N(nt) \geq k,\,Z_{R_n^{-1}(k)} \leq n\delta \right)
\\
&
\leq
\P\left( \sup_{t\in[0,1]} \sum_{i=1}^{N(nt)} \big(Z_i \one{Z_i \leq n\delta} - \E Z\big) > n\epsilon \right)
=
\P\left( \sup_{0\leq j\leq N(n)} \sum_{i=1}^{j} \big(Z_i \one{Z_i \leq n\delta} - \E Z\big) > n\epsilon \right)
\\
&
\leq
\P\left( \sup_{0\leq j\leq 2n} \sum_{i=1}^{j} \big(Z_i \one{Z_i \leq n\delta} - \E Z\big) > n\epsilon, N(n) < 2n \right)
+ \P\big(N(n) \geq 2n\big)
\\
&
\leq
\P\left( \sup_{0\leq j\leq 2n} \sum_{i=1}^{j} \big(Z_i \one{Z_i \leq n\delta} - \E Z\big) > n\epsilon\right)
 + \P\big(N(n) \geq 2n\big).
\end{align*}
 From Lemma~\ref{lemma:negligence-positive} and the fact that the second term decays at an exponential rate,
\begin{align}
&\limsup_{n\to\infty}\frac1{L(n)n^\alpha}
\P\left(Z_{R_n^{-1}(k)} \le n\delta,\sup_{t\in[0,1]}\bar H_n^k(t) > \epsilon\right)
\leq -(\epsilon/3)^\alpha (\epsilon/\delta)^{1-\alpha}. \label{ineq:log-asymptotics-for-Xnk-positive}
\end{align}
Turning to (ii),
\begin{align*}
\text{(ii)}
&
\leq
\P\left(\sup_{t\in[0,1]} \sum_{i=1}^{N(nt)}\left(\E Z - Z_i\one{R_n(i)>k}\right)> n\epsilon\right)
\\
&
=
\P\left(\sup_{t\in[0,1]} \sum_{i=1}^{N(nt)}\Big(\E Z - Z_i\one{Z_i\leq n\delta} + Z_i\left(\one{Z_i\leq n\delta} - \one{R_n(i)>k}\right)\Big)> n\epsilon\right)
\\
&
\leq
\P\left(\sup_{t\in[0,1]} \sum_{i=1}^{N(nt)}\Big(\E Z - Z_i\one{Z_i\leq n\delta} + Z_i\I_{\{Z_i\leq n\delta\}\cap\{R_n(i)\leq k\}}\Big)> n\epsilon\right)
\\
&
\leq
\P\left(\sup_{t\in[0,1]} \sum_{i=1}^{N(nt)}\Big(\E Z - Z_i\one{Z_i\leq n\delta}\Big) + kn\delta> n\epsilon\right)
\\
&
=
\P\left(\sup_{t\in[0,1]} \sum_{i=1}^{N(nt)}\Big(\E Z - Z_i\one{Z_i\leq n\delta}\Big) > n(\epsilon-k\delta) \right)
\\
&
\leq
\P\left(\sup_{0\leq j \leq 2n} \sum_{i=1}^{j}\Big(\E Z - Z_i\one{Z_i\leq n\delta}\Big) > n(\epsilon-k\delta), N(nt) < 2n \right) + \P(N(nt) \geq 2n)
\\
&
\leq
\P\left(\sup_{0\leq j \leq 2n} \sum_{i=1}^{j}\Big(\E Z - Z_i\one{Z_i\leq n\delta}\Big) > n(\epsilon-k\delta)\right) + \P(N(nt) \geq 2n).
\end{align*}
Applying Lemma~\ref{lemma:negligence-negative} to the first term and noticing that the second term vanishes at an exponential rate, we conclude that for $\delta$ and $k$ such that $k\delta < \epsilon$
\begin{equation}\label{eq:log-asymptotics-for-Xnk-negative}
\limsup_{n\to\infty} \frac1{L(n)n^\alpha} \log \P\left(Z_{R_n^{-1}(k)} \le n\delta,\sup_{t\in[0,1]}-\bar H_n^k(t) > \epsilon\right) = -\infty.
\end{equation}
From \eqref{ineq:log-asymptotics-for-Xnk-positive} and \eqref{eq:log-asymptotics-for-Xnk-negative},
\begin{equation}
\limsup_{n\to\infty}\frac1{L(n)n^\alpha}
\log\P\left(Z_{R_n^{-1}(k)} \le n\delta,\|\bar H_n^k\|_\infty > \epsilon\right)
\leq -(\epsilon/3)^\alpha (\epsilon/\delta)^{1-\alpha}.
\end{equation}
This together with \eqref{eq:bbbb} and \eqref{eq:log-asymptotics-for-QnGammak},
\begin{equation*}
\limsup_{n\to\infty}\frac1{L(n)n^\alpha}
\P\left(\|\bar H_n^k\|_\infty > \epsilon\right) \leq \max\{-(\epsilon/3)^\alpha(\epsilon/\delta)^{1-\alpha},\ -k\delta^\alpha\}
\end{equation*}
for any $\delta$ and $k$ such that $k\delta < \epsilon$. Choosing, for example, $\delta = \frac{\epsilon}{2k}$ and letting $k\to \infty$, we arrive at the conclusion of the lemma.
\end{proof}

%
%

\subsection{Proof of Lemma~\ref{lemma:rw-lem1}, \ref{lemma:rw-lem3}, and \ref{lemma:rw-lem2}}
\label{sec:proof-for-rw-lemmas}
\begin{proof}[Proof of Lemma~\ref{lemma:rw-lem1}]
We follow a similar program as in the proof of Lemma~\ref{lemma:equivalence-J}.
Recall that $\tilde Q^\gets(x) = \inf\{s>0: \P(Z\geq s) < y\}$ and $V_{(1)},\ldots, V_{(n-1)}$ are the order statistics of $n-1$ i.i.d.\ Uniform(0,1) random variables $V_1,\ldots, V_{n-1}$.
We first claim that $(\tilde Q^\gets(V_{(1)})/n, \ldots, \tilde Q^\gets(V_{(k)})/n)$ satisfies the LDP with speed $L(n)n^\alpha$ and the good rate function $\check I_k$ defined in \eqref{def-I-check-k}.
Let $f$ be a bounded continuous function such that $|f(x)|< M$, $x\in \R_+^k$ for some $M\in \R$. 
We want to prove that
\begin{align*}
&\lim_{n\to\infty} \frac{1}{L(n)n^\alpha} \log\E \exp\big\{L(n)n^\alpha f\big(\tilde Q^\gets(V_{(1)})/n,\ldots, \tilde Q^\gets(V_{(k)})/n\big)\big\}
= \sup_{x}\Lambda_f( x),
\end{align*}
where $\Lambda_f = f - \check I_k$ to invoke inverse Varadhan lemma and establish the LDP for $\big(\tilde Q^\gets(V_{(1)})/n,\ldots, \allowbreak\tilde Q^\gets(V_{(k)})/n\big)$.
Recall that in the proof of Lemma~\ref{lemma:ldp-for-k-biggest-jump-sizes}, we have shown that the supremum of $f(x) - \check I_k(x)$ over $\R_+^k$ is attained. 
Let $\hat x$ denote one of the optimizers that attain the supremum.  
Then, due to the form of $\check I_k$, for any given $\epsilon>0$, we can find $\delta>0$ and $\check x = (\check x_1,\ldots, \check x_k)$ such that $\check x_i  \geq \check x_{i+1} + \delta$ for $i=1,\ldots, k-1$ and $x \in \prod_{i=1}^k [\check x_i, \check x_i+\delta]$ implies 
$$\check I_k(x) \geq \check I_k (\hat x)-\epsilon
\qquad\qquad
\text{and}
\qquad\qquad
f( x) - \check I_k(x) \geq f(\hat x) - \check I_k(\hat x) - \epsilon.$$
Therefore, if we set 
$A_n(\delta) \triangleq\{(y_1,\ldots,y_k): \tilde Q^\gets(y_i)/n \in [\check x_i, \check x_i+\delta],\ i=1,\ldots,k+1\}$, then  
$y\in A_n(\delta)$ implies 
$$\check I_k(\tilde Q^\gets(y_1)/n,\ldots,\tilde Q^\gets(y_k)/n) \geq \check I_k(\hat x) -\epsilon$$ 
and
$$f(\tilde Q^\gets(y_1)/n,\ldots,\tilde Q^\gets(y_k)/n) -  \check I_k(\tilde Q^\gets(y_1)/n,\ldots,\tilde Q^\gets(y_k)/n) \geq f(\hat x) - \check I_k(\hat x) - \epsilon, $$ 
and hence, 
$$f(\tilde Q^\gets(y_1)/n,\ldots,\tilde Q^\gets(y_k)/n) \geq f(\hat x) - 2\epsilon.$$
Note also that $\tilde Q^\gets(y) < x$ if and only if $\P(Z\geq x) < y$, and hence, $\P(Z\geq n(\check x_i+\delta)) <  y_i \leq \P(Z \geq n\check x_i)$ implies $ \tilde Q^\gets(y_i)/n \in [\check x_i, \check x_i+\delta]$.
We have that
$$ \I_{\{\tilde Q^\gets(y_i)/n\in [\check x_i, \check x_i + \delta],\ i=1,\ldots, k\}} 
\geq \I_{\{\P(Z\geq n(\check x_i+\delta)) <  y_i \leq \P(Z \geq n\check x_i),\ i=1,\ldots, k\}},$$
and hence, for $y_{k+1}\geq \P(Z\geq n\check x_i)$,
\begin{align*}
&\int_0^{y_{k+1}}\cdots \int_0^{y_2}  \I_{\{\tilde Q^\gets(y_i)/n\in [\check x_i, \check x_i + \delta],\ i=1,\ldots, k\}}\ dy_1\cdots dy_k 
\\
&\geq
\int_0^{y_{k+1}}\cdots \int_0^{y_2}  \I_{\{\P(Z\geq n(\check x_i+\delta)) <  y_i \leq \P(Z \geq n\check x_i),\ i=1,\ldots, k\}}\ dy_1\cdots dy_k  
\\
&=
\int_0^1\cdots \int_0^1  \I_{\{\P(Z\geq n(\check x_i+\delta)) <  y_i \leq \P(Z \geq n\check x_i),\ i=1,\ldots, k\}}\ dy_1\cdots dy_k  
\\
&=
\prod_{i=1}^k \big(\P(Z\geq n\check x_i) - \P(Z \geq n\check x_i+n\delta)\big).
\end{align*}
Therefore,
\begin{align*}
&\E \exp\big\{L(n)n^\alpha f\big(\tilde Q^\gets(V_{(1)})/n,\ldots, \tilde Q^\gets(V_{(k)})/n\big)\big\}
\\
&\geq
\E \exp\big\{L(n)n^\alpha f\big(\tilde Q^\gets(V_{(1)})/n,\ldots, \tilde Q^\gets(V_{(k)})/n\big)\big\}\I_{\{(V_{(1)},\ldots,V_{(k)})\in A_n(\delta)\}}
\\
&
= 
\int_0^{1}\int_0^{y_{n-1}}\cdots \int_0^{y_2}e^{L(n)n^\alpha f(\tilde Q^\gets(y_1)/n,\ldots,\tilde Q^\gets(y_k)/n)} {n!} \I_{\{(y_1,\ldots, y_k) \in A_n(\delta)\}}\ dy_1\cdots dy_{n-2} dy_{n-1}
\\
&
\geq
{n!}e^{L(n)n^\alpha (f(\hat x)-2\epsilon)} \int_0^{1}\int_0^{y_{n-1}}\cdots \int_0^{y_2}  \I_{\{\tilde Q^\gets(y_i)/n\in [\check x_i, \check x_i + \delta],\ i=1,\ldots, k\}}\ dy_1\cdots dy_{n-2} dy_{n-1}
\\
&
\geq
{n!}e^{L(n)n^\alpha (f(\hat x)-2\epsilon)} 
\prod_{i=1}^k \big(\P(Z\geq n\check x_i) - \P(Z \geq n\check x_i+n\delta)\big) \int_0^{1}\int_0^{y_{n-1}}\cdots\int_{\P(Z\geq n\check x_i)}^{y_{k+2}}  dy_{k+1}\cdots dy_{n-2} dy_{n-1}
\\
&
=
{n!}e^{L(n)n^\alpha (f(\hat x)-2\epsilon)} 
\prod_{i=1}^k \big(\P(Z \geq n\check x_i) - \P(Z \geq n\check x_i+n\delta)\big)
\frac1{(n-k-1)!}\big(1-\P(Z\geq n \check x_i)\big)^{n-k-1}.
\end{align*}
Since 
$$
\liminf_{n\to\infty}\frac1{L(n)n^\alpha}\log\prod_{i=1}^k \big(\P(Z \geq n\check x_i) - \P(Z \geq n\check x_i+n\delta)\big) = 
-\sum_{i=1}^k\hat x^\alpha = -\check I_k(\hat x)
$$
and
$$
\liminf_{n\to\infty}\frac1{L(n)n^\alpha}\log \big(1-\P(Z\geq n \check x_i)\big)^{n-k-1} = 0,
$$
we get
\begin{align*}
\liminf_{n\to\infty}\frac1{L(n)n^\alpha}\log\E e^{L(n)n^\alpha f\big(\tilde Q^\gets(V_{(1)})/n,\ldots, \tilde Q^\gets(V_{(k)})/n\big)}
&\geq f(\hat x)  - 2\epsilon - \check I_k(\hat x)
\\
&= \sup_{x\in \R_+^k} \{f(x) - \check I_k(x)\} - 2\epsilon.
\end{align*}
Letting $\epsilon \to 0$, we arrive at the lower bound.

Turning to the upper bound,
\begin{align*}
&\E \exp\big\{L(n)n^\alpha f\big(\tilde Q^\gets(V_{(1)})/n,\ldots, \tilde Q^\gets(V_{(k)})/n\big)\big\}
\\
&=
\E \exp\big\{L(n)n^\alpha f\big(\tilde Q^\gets(V_{(1)})/n,\ldots, \tilde Q^\gets(V_{(k)})/n\big)\big\}
\I_{\{\tilde Q^\gets(V_{(1)})/n> R\}}
\\
&\quad\qquad+
\E \exp\big\{L(n)n^\alpha f\big(\tilde Q^\gets(V_{(1)})/n,\ldots, \tilde Q^\gets(V_{(k)})/n\big)\big\}
\I_{\{\tilde Q^\gets(V_{(1)})/n\leq R\}}.
\end{align*}
For the first term, note that
\begin{align}
&\E \exp\big\{L(n)n^\alpha f\big(\tilde Q^\gets(V_{(1)})/n,\ldots, \tilde Q^\gets(V_{(k)})/n\big)\big\}
\I_{\{\tilde Q^\gets(V_{(1)})/n> R\}}
\nonumber
\\
&
\leq
\E \exp\{L(n)n^\alpha M\}\I\{\tilde Q^\gets(V_{(1)})/n > R\}
\leq
\exp\{L(n)n^\alpha M\}\,\P(V_{(1)} \leq \P(Z\geq nR))
\nonumber
\\
&
=
\exp\{L(n)n^\alpha M\}\,
\Big(1-\big(1-\P(Z\geq nR)\big)^{n-1}\Big)
\nonumber
\\
&=
\exp\{L(n)n^\alpha M\}\,
\Big(1-\big(1-\exp\{-L(nR)(nR)^\alpha\}\big)^{n-1}\Big).
\label{left-with-straightforward-algebra}
\end{align}
Also, from the generic inequality $1-\exp(-z) \leq z$, 
\begin{align*}
1-(1-1/x)^y 
&=  1-\{ (1-1/x)^x\}^{y/x} = 1-\exp \log \{(1-1/x)^x\}^{y/x}
= 1-\exp \{(y/x)\log (1-1/x)^x\}
\\
&\leq
(y/x)\log (1-1/x)^{-x}
\end{align*}
for any $x, y>0$.
Setting $x=\exp(L(nR)(nR)^\alpha)$ and $y = n-1$, we get
\begin{align*}
&1-\big(1-\exp\{-L(nR)(nR)^\alpha\}\big)^{n-1}
\\
&
\leq  (n-1)\exp\{-L(nR)(nR)^\alpha\} \log\Big(1-1/{\exp\{L(nR)(nR)^\alpha\}}\Big)^{-\exp\{L(nR)(nR)^\alpha\}}.
\end{align*}
Substituting this into \eqref{left-with-straightforward-algebra},  we arrive at the upper bound for the first term:
\begin{align*}
&\limsup_{n\to\infty} \frac{1}{L(n)n^\alpha} \log\E \exp\big\{L(n)n^\alpha f\big(\tilde Q^\gets(V_{(1)})/n,\ldots, \tilde Q^\gets(V_{(k)})/n\big)\big\}
\I_{\{\tilde Q^\gets(V_{(1)})/n> R\}}
\leq M-R^\alpha.
\end{align*}
For the second term,
fix $\epsilon>0$ and pick $\{\check x^{(1)},\ldots,\check x^{(m)}\}\subset \R_+^k$ in such a way that 
$$\bigg\{\prod_{j=1}^k[\check x^{(l)}_j-\epsilon,\check x^{(l)}_j+\epsilon]\bigg\}_{l=1,\ldots,m}$$ 
covers $\{(x_1,\ldots,x_k): R \geq x_1 \geq  x_2 \geq \cdots \geq x_k \geq 0\}$ and $\check x_1^{(l)} \geq \check x_2^{(l)} \geq \cdots \geq \check x_k^{(l)}\geq 0$ for $l=1,\ldots, m$. 
Set
\begin{equation*}
A_{n,l}(R)\triangleq\Big\{(y_1,\ldots,y_k)\in \R_+^k: y_1\leq \cdots \leq y_k, Q^\gets(y_j)/n \in \big[\check x_j^{(l)}-\epsilon,\check x_j^{(l)}+\epsilon\big],\ j=1,\ldots, k\Big\}.
\end{equation*}
Note that 
$y_1\leq \cdots \leq y_k\ \&\ Q^\gets(y_1)/n \leq R$ implies $R \geq Q^\gets(y_1)/n \geq Q^\gets(y_2)/n \geq \cdots \geq Q^\gets(y_k)/n$, which, in turn, implies 
$Q^\gets(y_j)/n \in \big[\check x_j^{(l)}-\epsilon,\check x_j^{(l)}+\epsilon\big],\ j=1,\ldots, k$ for some $l\in \{1,\ldots,m\}$.
Therefore, 
$
\Big\{(y_1,\ldots,y_k)\in \R_+^k: y_1 \leq \cdots \leq y_k,\ Q^\gets(y_1)/n \leq R\Big\}
\subseteq\bigcup_{l=1}^{m}A_{n,l}(R)$, and hence,
\begin{align*}
&\E \exp\big\{L(n)n^\alpha f\big(\tilde Q^\gets(V_{(1)})/n,\ldots, \tilde Q^\gets(V_{(k)})/n\big)\big\}
\I_{\{\tilde Q^\gets(V_{(1)})\leq R\}}
\\
&
\leq
\sum_{i=1}^m
\E \exp\big\{L(n)n^\alpha f\big(\tilde Q^\gets(V_{(1)})/n,\ldots, \tilde Q^\gets(V_{(k)})/n\big)\big\}
\I_{\{(V_{(1)},\ldots,V_{(k)})\in A_{n,l}(R)\}}.
\end{align*}
Note that
\begin{align*}
&\E \exp\big\{L(n)n^\alpha f\big(\tilde Q^\gets(V_{(1)})/n,\ldots, \tilde Q^\gets(V_{(k)})/n\big)\big\}
\I_{\{(V_{(1)},\ldots,V_{(k)})\in A_{n,l}(R)\}}
\\&
=
\E e^{L(n)n^\alpha \Lambda_f\big(\tilde Q^\gets(V_{(1)})/n,\ldots, \tilde Q^\gets(V_{(k)})/n\big)}
e^{L(n)n^\alpha \check I_k\big(\tilde Q^\gets(V_{(1)})/n,\ldots, \tilde Q^\gets(V_{(k)})/n\big)}
\I_{\{(V_{(1)},\ldots,V_{(k)})\in A_{n,l}(R)\}}
\\
&\leq
e^{L(n)n^\alpha \Lambda_f(\hat x_1,\ldots,\hat x_k)}\E 
e^{L(n)n^\alpha \check I_k\big(\tilde Q^\gets(V_{(1)})/n,\ldots, \tilde Q^\gets(V_{(k)})/n\big)}
\I_{\{(V_{(1)},\ldots,V_{(k)})\in A_{n,l}(R)\}}
\\
&\leq
e^{L(n)n^\alpha \Lambda_f(\hat x_1,\ldots,\hat x_k)} 
e^{L(n)n^\alpha \check I_k\big(\check x_1^{(l)}+\epsilon, \ldots, \check x_k^{(l)}+\epsilon\big)}
\E\I_{\{(V_{(1)},\ldots,V_{(k)})\in A_{n,l}(R)\}}
\end{align*}
and
\begin{align*}
&\E\I_{\{(V_{(1)},\ldots,V_{(k)})\in A_{n,l}(R)\}}
\\
&\int_0^1\int_0^{y_{n-1}}\cdots \int_0^{y_2} {(n-1)!} \I_{\{(y_1,\ldots, y_k) \in A_{n,l}(R)\}}\ dy_1\cdots dy_{n-2} dy_{n-1}
\\
&
\geq
(n-1)!\int_0^{1}\int_0^{y_{n-1}}\cdots \int_0^{y_2}  \I_{\{\tilde Q^\gets(y_i)/n\in [\check x_i^{(l)}-\epsilon, \check x_i^{(l)} + \epsilon],\ i=1,\ldots, k\}}\ dy_1\cdots dy_{n-2} dy_{n-1}
\\
&
\leq
(n-1)!
\prod_{i=1}^k \big(\P(Z\geq n\check x_i^{(l)}-n\epsilon) - \P(Z\geq n\check x_i^{(l)}+n\epsilon)\big) \int_0^1
\int_0^{y_{n-1}}\cdots\int_0^{y_{k+2}}  dy_{k+1}\cdots dy_{n-2} dy_{n-1}
\\
&
=
{(n-1)!}
\prod_{i=1}^k \big(\P(Z\geq n\check x_i^{(l)}-n\epsilon) - \P(Z\geq n\check x_i^{(l)}+n\epsilon)\big) 
\frac{1}{(n-k-1)!}.
\\
&
\leq
n^k
\prod_{i=1}^k \big(\P(Z\geq n\check x_i^{(l)}-n\epsilon) - \P(Z\geq n\check x_i^{(l)}+n\epsilon)\big).
\end{align*}
Therefore,
\begin{align*}
&\E \exp\big\{L(n)n^\alpha f\big(\tilde Q^\gets(V_{(1)})/n,\ldots, \tilde Q^\gets(V_{(k)})/n\big)\big\}
\I_{\{\tilde Q^\gets(V_{(1)})\leq R\}}
\\
&\leq
n^ke^{L(n)n^\alpha \Lambda_f(\hat x_1,\ldots,\hat x_k)} 
\sum_{i=1}^l
e^{L(n)n^\alpha \check I_k\big(\check x_1^{(l)}+\epsilon, \ldots, \check x_k^{(l)}+\epsilon\big)}
\prod_{i=1}^k \big(\P(Z\geq n\check x_i^{(l)}-n\epsilon) - \P(Z\geq n\check x_i^{(l)}+n\epsilon)\big).
\end{align*}
Note that 
\begin{align*}
&\limsup_{n\to\infty}\frac1{L(n)n^\alpha}\log\prod_{i=1}^k \big(\P(Z\geq n\check x_i^{(l)}-n\epsilon) - \P(Z\geq n\check x_i^{(l)}+n\epsilon)\big) 
\\
&= 
-\sum_{i=1}^k(\check x_i^{(l)}-\epsilon)_+^\alpha 
= -\check I_k\big((\check x_1^{(l)}-\epsilon,\ldots,\check x_k^{(l)}-\epsilon)_+\big),
\end{align*} where $(y)_+$ denotes $\max \{y, 0\}$ and $(y_1,\ldots, y_k)_+$ denotes $((y_1)_+,\ldots,(y_k)_+)$.
This along with the principle of the largest term, 
\begin{align*}
&\limsup_{n\to\infty} \frac{1}{L(n)n^\alpha} \log\E \exp\big\{L(n)n^\alpha f\big(\tilde Q^\gets(V_{(1)})/n,\ldots, \tilde Q^\gets(V_{(k)})/n\big)\big\}
\I_{\{\tilde Q^\gets(V_{(1)})\leq R\}}
\\
&\leq 
\max_{l=1,\ldots,m} \Big(\Lambda_f(\hat x_1,\ldots, \hat x_k) 
+ \check I_k\big(\check x_1^{(l)}+\epsilon, \ldots, \check x_k^{(l)}+\epsilon\big)
- \check I_k\big((\check x_1^{(l)}-\epsilon, \ldots, \check x_k^{(l)}-\epsilon)_+\big)\Big)
\\
&\leq 
\max_{l=1,\ldots,m} \Big(\Lambda_f(\hat x_1,\ldots, \hat x_k) 
+ k\epsilon^\alpha
\Big).
\end{align*}
Sending $\epsilon\to0$, we get
\begin{align*}
&\limsup_{n\to\infty} \frac{1}{L(n)n^\alpha} \log\E \exp\big\{L(n)n^\alpha f\big(\tilde Q^\gets(V_{(1)})/n,\ldots, \tilde Q^\gets(V_{(k)})/n\big)\big\}
\I_{\{\tilde Q^\gets(V_{(1)})\leq R\}}
\\
&\leq 
\Lambda_f(\hat x_1,\ldots, \hat x_k) .
\end{align*}
Now, combining with the bound for the first term, and sending $R\to\infty$, 
we get the upper bound:
\begin{align*}
&\limsup_{n\to\infty} \frac{1}{L(n)n^\alpha} \log\E \exp\big\{L(n)n^\alpha f\big(\tilde Q^\gets(V_{(1)})/n,\ldots, \tilde Q^\gets(V_{(k)})/n\big)\big\}
\\
&\leq 
\max\{\Lambda_f(\hat x_1,\ldots, \hat x_k) , M-R^\alpha\}
\to \Lambda_f(\hat x_1,\ldots, \hat x_k).
\end{align*}
Together with the lower bound, we get
\begin{align*}
&\lim_{n\to\infty} \frac{1}{L(n)n^\alpha} \log\E \exp\big\{L(n)n^\alpha f\big(\tilde Q^\gets(V_{(1)})/n,\ldots, \tilde Q^\gets(V_{(k)})/n\big)\big\}
= \Lambda_f(\hat x_1,\ldots, \hat x_k),
\end{align*}
which in turn allows us to apply Bryc's inverse Varadhan Lemma to prove that $\big(\tilde Q^\gets(V_{(1)})/n,\ldots,\allowbreak\tilde Q^\gets(V_{(k)})/n\big)$ satisfies the LDP with the rate function $\check I_k$.
From Theorem 4.14 of \cite{ganesh2004big}, we see that 
$\big(\tilde Q^\gets(V_{(1)})/n,\ldots,\allowbreak\tilde Q^\gets(V_{(k)})/n, Z/n\big)$ satisfies the LDP with the rate function $\check I_k'$ given by
\begin{equation}\label{def-I-check-k-prime}
\check I_k'(x_1,\ldots,x_{k+1})=
\begin{cases}
\sum_{i=1}^{k+1} x_i^\alpha & \mathrm{if}\ x_1 \ge x_2 \ge \dots \ge x_k \ge 0 \text{ and } x_{k+1} \geq 0\\
\infty, & \mathrm{o.w.}
\end{cases}.
\end{equation}
Proceeding with essentially the same argument as in Corollary~\ref{lemma:ldp-for-k-biggest-jump-sizes-and-their-times} and Lemma~\ref{lemma:sample-path-ldp-for-J-hat}---except for considering a mapping $\tilde T_k:(x_1,\ldots,x_{k+1},u_1,\ldots,u_k) \mapsto \sum_{i=1}^k x_i \I_{[u_i,1]} + x_{k+1} \I_{\{1\}}$ instead of the mapping $T_k:(x_1,\ldots,x_{k},u_1,\ldots,u_k) \mapsto \sum_{i=1}^k x_i \I_{[u_i,1]}$) and $\tilde \D_{\leqslant k}$ instead of $\D_{\leqslant k}$---we  conclude that 
$\tilde J_n^k(t) = \frac1n \sum_{i=1}^k \tilde Q^\gets (V_{(i)}) \I_{[U_i,1]}(t)+\frac1n Z \I_{\{1\}}(t) $ satisfies the LDP with speed $L(n)n^\alpha$ and the rate function $\tilde I_k$  in \eqref{eq:rate-function-for-Ik-tilde}.
\end{proof}

\begin{proof}[Proof of Lemma~\ref{lemma:rw-lem3}] 
The proof is essentially identical to Lemma~\ref{lemma:Jk-approximates-J}, and hence, omitted. 
\end{proof}

\begin{proof}[Proof of Lemma~\ref{lemma:rw-lem2}]
Let 
$$
\check H_n^k(t) 
\triangleq 
\tilde H_n^k(t)  + \frac1n \E Z \I_{\{1\}}(t)
=
\frac1n \sum_{i=k+1}^{n-1}  \tilde Q^\gets (V_{(i)}) \I_{[U_i,1]}(t) - \frac1n \sum_{i=1}^{n-1} \E Z  \I_{[U_i,1]}(t)
.
$$
Since 
$\P(\|\tilde H_n^k\|_\infty \geq \epsilon) 
\leq \P(\|\tilde H_n^k\|_\infty \geq \epsilon/2) + \P(\|\frac 1n \E Z \sum_{i=1}^{n-1}\I_{[U_i,1]} + \frac1n \E Z \I_{\{1\}} \|_\infty \geq \epsilon/2)$ and
$ \P(\|\frac 1n \E Z \sum_{i=1}^{n-1}\I_{[U_i,1]}(t) + \frac1n \E Z \I_{\{1\}} \|_\infty >\epsilon/2) = 0$ for large enough $n$, we only need to prove that 
$$\lim_{k\to\infty}\limsup_{n\to\infty} \frac{1}{L(n)n^\alpha} \log\P(\| \check H_n^k\|_\infty > \epsilon) = -\infty.$$ 
To show this, we fix an arbitrary $\delta \in(0, \epsilon/k)$ and consider the following decomposition:
\begin{align*}
&\P(
\|
\check H_n^k\|_\infty > \epsilon )
\leq
 \P(
\|
\check H_n^k 
\|_\infty > \epsilon ,\  \tilde Q^\gets(V_{(k)}) < n\delta ) 
 + \P\left(
 \tilde Q^\gets(V_{(k)}) \geq n\delta \right).
\end{align*}
We first bound the second term. 
Since the density of the $k$\textsuperscript{th} order statistic of the uniform distribution on $[0,1]$ is $n{n-1\choose k-1} x^{k-1}(1-x)^{n-k}$,
\begin{align*}
\P\left(  \tilde Q^\gets(V_{(k)}) \geq n\delta \right) 
&= \P\left( V_{(k)} \leq  \P(Z\geq n\delta) \right)
\leq
\int_0^{ \P(Z\geq n\delta)} n  {n-1 \choose k-1} x ^{k-1} dx 
\\
&
=  {n \choose k} \big( \P(Z\geq n\delta)\big)^k
=  {n \choose k} \exp (-k L(n\delta) (n\delta)^\alpha)
\end{align*}
and hence, $\limsup_{n\to\infty} \frac{1}{L(n)n^\alpha} \log \P( \tilde Q_n^\gets (V_{(k)}) > n\delta) \leq -k\delta^\alpha$.
For the first term,
\begin{align*}
& \P(
\|
 \check H_n^k 
\|_\infty > \epsilon ,\  \tilde Q^\gets(V_{(k)}) < n\delta ) 
\\
&= \P\Big(
\sup_{t\in[0,1]}
 \check H_n^k (t)
 > \epsilon ,\  \tilde Q^\gets(V_{(k)}) \leq n\delta \Big) 
+
\P\Big(
\sup_{t\in[0,1]}
- \check H_n^k 
 > \epsilon ,\  \tilde Q^\gets(V_{(k)}) \leq n\delta \Big) 
 \\
& \leq 
\P\Big( \max_{1\leq j \leq n-1} \sum_{i=1}^j (Z_i\I_{\{Z_i \leq n\delta\}} - \E Z) > n\epsilon\Big)  
+
\P\Big( \max_{1\leq j \leq n-1} \sum_{i=1}^j (\E Z - Z_i\I_{\{Z_i \leq n\delta\}})+kn\delta > n\epsilon\Big) .
\end{align*}
Note that from Lemma~\ref{lemma:negligence-positive}, 
$$
\limsup_{n\to\infty} \frac{1}{L(n)n^\alpha}\P\Big( \max_{1\leq j \leq n-1} \sum_{i=1}^j (Z_i\I_{\{Z_i \leq n\delta\}} - \E Z) > n\epsilon\Big)  \leq -(\epsilon/3)^\alpha (\epsilon/\delta)^{1-\alpha}
$$
and from \ref{lemma:negligence-negative}, since $\delta < \epsilon/k$,
$$
\limsup_{n\to\infty} \frac{1}{L(n)n^\alpha}\P\Big( \max_{1\leq j \leq n-1} \sum_{i=1}^j (\E Z - Z_i\I_{\{Z_i \leq n\delta\}})+kn\delta > n\epsilon\Big) = -\infty.
$$
Therefore, 
$$
\limsup_{n\to\infty} \frac{1}{L(n)n^\alpha} \P(\| \tilde H_n^k\|_\infty>\epsilon, Q_n^\gets(V_{(k)}) \leq n\delta) \leq \max \{ -(\epsilon/3)^\alpha (\epsilon/\delta)^{1-\alpha}, -\infty\} =  -(\epsilon/3)^\alpha (\epsilon/\delta)^{1-\alpha}.
$$
Applying the principle of the maximum term once again,
$$
\lim_{k\to\infty}\limsup_{n\to\infty} \frac{1}{L(n)n^\alpha} \P(\| \tilde H_n^k\|_\infty>\epsilon) \leq \lim_{k\to\infty}\max \{ -(\epsilon/3)^\alpha (\epsilon/\delta)^{1-\alpha}, -k\delta^\alpha\} = - (\epsilon/3)^\alpha (\epsilon/\delta)^{1-\alpha}.
$$
Since $\delta$ can be chosen arbitrarily small, 
$$
\lim_{k\to\infty}\limsup_{n\to\infty} \frac{1}{L(n)n^\alpha} \P(\| \tilde H_n^k\|_\infty>\epsilon) = -\infty.
$$


\end{proof}

%
%
\subsection{Proof of Theorem~\ref{theorem:multi-d-levy-processes}}
\label{sec:proof-of-theorem-2-3}
We follow a similar program as in Section~\ref{subsection:extended-ldp-1-d-levy-processes} and the earlier subsections of this section.
Let $\bar Q_{n}^{(i)}(j)\triangleq Q_{n}^{\gets}(\Gamma_j^{(i)})/n$ where $Q_{n}^{\gets}(t) = \inf\{s> 0: n\nu[s,\infty) < t\}$ and $\Gamma^{(i)}_l = E^{(i)}_1 + \cdots + E^{(i)}_l$ where $E^{(i)}_j$'s are independent standard exponential random variables.
Let $U^{(i)}_j$ be independent uniform random variables in [0,1] and
 $Z_n^{(i)} \triangleq \big(\bar Q_{n}^{(i)}(1),\ldots,\bar Q_{n}^{(i)}(k),\allowbreak U_1^{(i)},\ldots,U_{k}^{(i)}\big)$.
The following corollary is an immediate consequence of Corollary~\ref{lemma:ldp-for-k-biggest-jump-sizes-and-their-times} and Theorem 4.14 of \cite{ganesh2004big}.
\begin{corollary}
$
\big(Z_n^{(1)},\ldots,Z_n^{(d)}\big)$ satisfies the LDP in $\prod_{i=1}^{d}\big(\R^{k}_+\times[0,1]^{k}\big)$ with the rate function $\hat I_{k}^d(z^{(1)},\ldots,z^{(d)}) \triangleq \sum_{j=1}^d \hat I_{k}(z^{(j)})$
where $z^{(j)}=(x^{(j)}_1,\dots,x^{(j)}_{k},u^{(j)}_1,\dots,u^{(j)}_{k})$ for each $j \in \{1,\dots,d\}$.
\end{corollary}

Let $\hat J_{n}^{\leqslant k\,(i)} \triangleq \sum_{j=1}^{k} \bar Q_{n}^{(i)}(j)\I_{[U_j^{(i)},1]}$.
\begin{lemma}\label{lemma:multi-Jk-process}
$\big(\hat J_n^{\leqslant k\,(1)},\ldots,\hat J_n^{\leqslant k\,(d)}\big)$ satisfies the LDP in $\prod_{i=1}^{d}\mathbbm{D}\left([0,1],\R\right)$ with speed $L(n)n^\alpha$ and the rate function
\begin{align*}
I_{k}^d(\xi_1,...\xi_d)
&
\triangleq
\sum_{i=1}^d I_{k}(\xi_i)
=
\begin{cases}
\sum_{i=1}^{d}\sum_{t: \xi_i(t) \neq \xi_i(t-)}\left(\xi_i(t)-\xi_i(t-)\right)^{\alpha}
& \text{if }\ \xi_i \in \mathbbm{D}_{\leqslant k}\ \text{ for } \ i=1,\ldots,d,
\\
\infty, & otherwise.
\end{cases}
\end{align*}

\end{lemma}

\begin{proof}
Since $I_{k_i}$ is lower semi-continuous in $\prod_{i=1}^d \D([0,1],\R)$ for each $i$, $I_{k_1,\ldots,k_d}$ is a sum of lower semi-continuous functions, and hence, is lower semi-continuous itself.
The rest of the proof for the LDP upper bound and the lower bounds mirrors that of the one dimensional case (Lemma~\ref{lemma:sample-path-ldp-for-J-hat}) closely, and hence, omitted.
\end{proof}

\begin{proof}[Proof of Lemma~\ref{lemma:multi-d-equivalence-J}]
Again, we consider the same distributional relation for each coordinate as in the 1-dimensional case:
\begin{align*}
\bar J_n^{k\,(i)}
\stackrel{\mathcal D}{=}
\underbrace{\frac{1}{n} \sum_{j=1}^{k} Q_n^{(i)}(j)\I_{[U_j,1]}}_{= \hat{J}_n^{\leqslant k\,(i)}}-\underbrace{\frac{1}{n}\I\{\tilde N_n^{(i)}<k\}\sum_{j=\tilde N_n^{(i)}+1}^{k}Q_n^{(i)}(j)\I_{[U_j^{(i)},1]}}_{= \check{J}_n^{\leqslant k\,(i)}}.
\end{align*}
Note that this distributional equality holds jointly w.r.t.\ $i=1,\ldots,d$ due to the assumed independence.
Let $F$ be a closed set and write
\begin{align*}
&\P( (\bar{J}_n^{k\,(1)},\ldots,\bar{J}_n^{k\,(d)}) \in F )
\\
&
\leq
\P\Big((\hat{J}_n^{\leqslant k\,(1)},\ldots,\hat{J}_n^{\leqslant k\,(d)})\in F ,\ {\textstyle \sum_{i=1}^d}\I\{\tilde N_n^{(i)} < k\}=0\Big)
+\sum_{i=1}^d\P\big(\I\{N_n^{(i)}<k\} \neq 0\big)
\\
&
\leq
\P\Big((\hat{J}_n^{\leqslant k\,(1)},\ldots,\hat{J}_n^{\leqslant k\,(d)})\in F \Big)+\sum_{i=1}^d\P\big(\I\{N_n^{(i)}<k\} \neq 0\big).
\end{align*}
From Lemma~\ref{lemma:multi-Jk-process} and the principle of the largest term,
\begin{align*}
&\limsup_{n \rightarrow \infty}\frac{\log \P\big( (\bar{J}_n^{k\,(1)},\ldots,\bar{J}_n^{k\,(d)}) \in F \big)
}{L(n)n^\alpha}
\\
&
\leq
\limsup_{n \rightarrow \infty}\frac{\log \P\big((\hat{J}_n^{\leqslant k\,(1)},\ldots,\hat{J}_n^{\leqslant k\,(d)})\in F\big)}{L(n)n^\alpha}
\vee\max_{i=1,\ldots,d}
\limsup_{n \rightarrow \infty}\frac{\log \P\big(\tilde N_n^{(i)} < k\big)}{L(n)n^\alpha}
\\
&
\leq
-\inf_{(\xi_1,\ldots,\xi_d) \in F}I_k^d(\xi_1,\ldots,\xi_d).
\end{align*}
Turning to the lower bound, let $G$ be an open set.
Since the lower bound is trivial in case $\inf_{x\in G}I_k(x)=\infty$, we focus on the case $\inf_{x\in G}I_k(x)<\infty$.
In this case, using a reasoning similar to the one leading to \eqref{eq:lowerbound-because-ratio-zero},
\begin{align*}
&\liminf_{n\rightarrow \infty}\frac{\log\P((\bar{J}_n^{k\,(1)},\ldots,\bar{J}_n^{k\,(d)}) \in G) }{L(n)n^\alpha}
\\
&\geq
\liminf_{n\rightarrow \infty}\frac{\log\P\big((\bar{J}_n^{k\,(1)},\ldots,\bar{J}_n^{k\,(d)}) \in G,\,\sum_{i=1}^d \I\{\tilde N_n^{(i)} \geq k\}=0\big)}{L(n)n^\alpha}
\\
&
=
\liminf_{n\rightarrow \infty}\frac{\log\P\big((\hat{J}_n^{\leqslant k\,(1)},\ldots,\hat{J}_n^{\leqslant k\,(d)}) \in G,\,\sum_{i=1}^d \I\{\tilde N_n^{(i)}\geq k\} = 0\big)}{L(n)n^\alpha}
\\
&
\geq
\liminf_{n\rightarrow  \infty}\frac{1}{L(n)n^\alpha}\log\left(\P\big((\hat{J}_n^{\leqslant k\,(1)},\ldots,\hat{J}_n^{\leqslant k\,(d)}) \in G\big)-d\P(\tilde N_n^{(1)}< k)\right)
\\
&
=
\liminf_{n\rightarrow \infty}\frac{1}{L(n)n^\alpha}\log\P\big((\hat{J}_n^{\leqslant k\,(1)},\ldots,\hat{J}_n^{\leqslant k\,(d)}) \in G\big)
\\
&
\geq
-\inf_{(\xi_1,\ldots,\xi_d) \in G}I_k^d(\xi_1,\ldots,\xi_d).
\end{align*}

\end{proof}

The proof of Lemma~\ref{eq:extspldplbt} is completely analogous to the one-dimensional case, and therefore omitted.

\appendix

%
%
%
%

\section{$M_1'$ topology and goodness of the rate function}\label{section:appendix-for-M_1p}
Let $\tilde {\mathbb D}[0,1]$ be the space of functions from $[0,1]$ to $\R$ such that the left limit exists at each $t\in (0,1]$, the right limit exists at each $t\in[0,1)$, and
\begin{equation}\label{condition:jump-consistency}
\xi(t) \in [\xi(t-)\wedge\xi(t+), \,\xi(t-)\vee\xi(t+)]
\end{equation}
for each $t\in[0,1]$ where we interpret $\xi(0-)$ as $0$ and $\xi(1+)$ as $\xi(1)$.
\begin{definition}
For $\xi\in \tilde \D$, define the extended completed graph $\Gamma'(\xi)$ of $\xi$ as
$$
\Gamma'(\xi) \triangleq \{(u,t)\in \R\times [0,1]: u\in [\xi(t-)\wedge \xi(t+),\ \xi(t-)\vee \xi(t+)]\},
$$
where $\xi(0-) \triangleq 0$ and $\xi(1+) \triangleq \xi(1)$.
Define an order on the graph $\Gamma'(\xi)$ by setting $(u_1,t_1) < (u_2,t_2)$ if either
\begin{itemize}
\item $t_1 < t_2$; or
\item $t_1 = t_2$ and $|\xi(t_1-)-u_1| < |\xi(t_2-) - u_2|$.
\end{itemize}
We call a continuous nondecreasing function $(u,t) =\big((u(s), t(s)),\,s \in [0,1]\big)$ from $[0,1]$ to $\R\times[0,1]$ a parametrization of $\Gamma'(\xi)$---or a parametrization of $\xi$---if $\Gamma'(\xi) = \{(u(s), t(s)): s\in [0,1]\}$.
\end{definition}

\begin{definition}
Define the $M_1'$ metric on $\mathbb D$ as follows
$$
d_{M_1'}(\xi,\zeta)
\triangleq
\inf_{\substack{(u,t) \in \Gamma'(\xi)\\(v,r) \in \Gamma'(\zeta)}}
\{
\|u-v\|_\infty + \|t-r\|_\infty
\}.
$$
\end{definition}

Let $\mathbb D^\uparrow \triangleq \{\xi\in \mathbb D:\ \xi \text{ is nondecreasing and }\xi(0) \geq 0\}$.
\begin{proposition}\label{thm:convergence}
Suppose that $\hat\xi_0\in\tilde \D$ with $\hat \xi_0(0) \geq 0$ and $\xi_n \in \mathbb D^\uparrow$ for each $n\geq 1$.
If
$T \triangleq \{t\in[0,1]:\xi_n(t)\to \hat\xi_0(t)\}$ is dense on $[0,1]$ and $1\in T$%
, then $\xi_n\stackrel{M_1'}{\to}\xi_0\in \D^\uparrow$ where
$
\xi_0(t) \triangleq \lim_{s\downarrow t}\hat \xi_0(s)
$
for $t\in[0,1)$ and $\xi_0(1) \triangleq \hat \xi_0(1)$.
\end{proposition}
\begin{proof}
It is easy to check that $\hat\xi_0$ has to be non-negative and non-decreasing, and for such $\hat\xi_0$, $\xi_0$ should be in $\D^\uparrow$.
Let $(x,t)$ be a parametrization of $\Gamma'(\hat\xi_0)$, and let $\epsilon>0$ be given.
Note that $\Gamma'(\xi_0)$ and $\Gamma'(\hat\xi_0)$ coincide.
Therefore, the proposition is proved if we show that there exists an integer $N_0$ such that for each $n\geq N_0$, $\Gamma'(\xi_n)$ can be parametrized by some $(y, r)$  such that
\begin{equation}\label{eq:want-to-construct}
\|x-y\|_\infty + \|t-r\|_\infty \leq \epsilon.
\end{equation}
We start with making an observation that one can always construct a finite number of points $S=\{s_i\}_{i=0,1,\ldots,m} \subseteq [0,1]$ such that
\begin{itemize}
\item[(S1)] $0 = s_0<s_1<\cdots<s_m=1$;

\item[(S2)] $t(s_i) - t(s_{i-1}) < \epsilon/4$ for $i=1,\ldots,m$;

\item[(S3)] $x(s_i) - x(s_{i-1}) < \epsilon/8$ for $i=1,\ldots,m$;

\item[(S4)] if $t(s_{k-1}) < t(s_k) < t(s_{k+1})$ then $t(s_k)\in T$;

\item[(S5)] if $t(s_{k-1}) < t(s_k) = t(s_{k+1})$, then $t(s_{k-1})\in T$; if, in addition, $k-1>0$, then $t(s_{k-2}) < t(s_{k-1})$;

\item[(S6)] if $t(s_{k-1}) = t(s_{k}) < t(s_{k+1})$, then $t(s_{k+1})\in T$; if, in addition, $k+1<m$, then $t(s_{k+1})<t(s_{k+2})$.

\end{itemize}
One way to construct such a set is to start with $S$ such that (S1), (S2), and (S3) are satisfied.
This is always possible because $x$ and $t$ are continuous and non-decreasing.
Suppose that (S4) is violated for some three consecutive points in $S$, say $s_{k-1}$, $s_k$, $s_{k+1}$.
We argue that it is always possible to eliminate this violation by either adding an additional point $\hat s_k$ or moving $s_k$ slightly.
More specifically, if there exists $\hat s_k\in (s_{k-1}, s_{k+1})\setminus \{s_k\}$ such that $t(\hat s_k) = t(s_k)$, add $\hat s_k$ to $S$.
If there is no such $\hat s_k$, $t(\cdot)$ has to be strictly increasing at $s_k$, and hence, from the continuity of $x$ and $t$ along with the fact that $T$ is dense, we can deduce that there has to be $\tilde s_k\in(s_{k-1}, s_{k+1})$ such that $t(\tilde s_k)\in T$ and $|t(\tilde s_k) - t(s_k)|$ and $|x(\tilde s_k) - x(s_k)|$ are small enough so that (S2) and (S3) are still satisfied when we replace $s_k$ with $\tilde s_k$ in $S$.
Iterating this procedure, we can construct $S$ so that (S1)-(S4) are satisfied.
Now turning to (S5),
suppose that it is violated for three consecutive points $s_{k-1}$, $s_k$, $s_{k+1}$ in $S$.
Since $T$ is dense and $t$ is continuous, one can find $\hat s_k$ between $s_{k-1}$ and $s_k$ such that $t(s_{k-1})< t(\hat s_k) < t(s_{k})$ and $t(\hat s_k) \in T$.
Note that after adding $\hat s_k$ to $S$, (S2), (S3), and (S4) should still hold while the number of triplets that violate (S5) is reduced by one.
Repeating this procedure for each triplet that violates (S5), one can construct a new $S$ which satisfies (S1)-(S5).
One can also check that the same procedure for the triplets that violate (S6) can reduce the number of triplets that violate (S6) while not introducing any new violation for (S2), (S3), (S4), and (S5).
Therefore, $S$ can be augmented so that the resulting finite set satisfies (S6) as well.
Set $\hat S\triangleq\{s_i\in S:\, t(s_i) \in T, \ t(s_{i-1})<t(s_{i}) \text{ in case }i>0, \ t(s_i)<t(s_{i+1}) \text{ in case }i<m\}$
\cf{
Note that
(S4) implies $s_k\in \check S$,
(S5) implies $s_{k-1}\in \check S$,
(S6) implies $s_{k+1}\in \check S$.
}%
and
let $N_0$ be such that $n \geq N_0$ implies $|\xi_n(t(s_i)) - \hat\xi_0 (t(s_i))| < \epsilon /8$ for all $s_i\in \hat S$.
Now we will fix $n\geq N_0$ and proceed to showing that we can re-parametrize an arbitrary parametrization $(y',r')$  of $\Gamma(\xi_n)$ to obtain a new parametrization $(y,r)$ such that (\ref{eq:want-to-construct}) is satisfied.
Let $(y', r')$ be an arbitrary parametrization  of $\Gamma(\xi_n)$.
For each $i$ such that $s_i\in \hat S$, let $s_i'\triangleq \max\{s\geq 0:  r'(s) = t(s_i)\}$
so that $r'(s'_i) = t(s_i)$ and $\xi_n(r'(s'_i)) =y'(s'_i)$.
For $i$'s such that $s_i\in S\setminus \hat S$,
note that there are three possible cases: $t(s_i) \in (0,1)$, $t(s_i) = 0$, and $t(s_i) = 1$.
Since the other cases can be handled in similar (but simpler) manners, we focus on the case $t(s_i) \in (0,1)$. 
In this case, one can check that there exist $k$ and $j$ such that $k\leq i \leq k+j$, $t(s_{k-1}) < t(s_{k}) = t(s_{k+j}) < t(s_{k+j+1})$, and $s_{k-1}, s_{k+j+1}\in \hat S$.
Here we assume that $k>1$; the case $k=1$ is essentially identical but simpler---hence omitted.
Note that from the monotonicity of $\hat \xi_0$ and \eqref{condition:jump-consistency},
$$
x(s_{k-2})\leq \hat\xi_0(t(s_{k-2})+)\leq \hat\xi_0(t(s_{k-1})-)\leq \hat \xi_0(t(s_{k-1})) \leq\hat\xi_0(t(s_{k-1})+)  \leq \hat\xi_0(t(s_k)-) \leq x(s_k),
$$
i.e., $\hat \xi_0(t(s_{k-1})) \in [x(s_{k-2}), x(s_k)]$,
which along with (S3) implies
$|\hat \xi_0(t(s_{k-1})) - x(s_{k-1})| < \epsilon/8$.
From this, (S5), and the constructions of $s'_{k-1}$ and $N_0$,
\begin{align*}
|y'(s'_{k-1}) - x(s_{k-1})|
&
=
|\xi_n(r'(s'_{k-1}))-x(s_{k-1})|
\\
&
= |\xi_n(r'(s'_{k-1}))- \hat\xi_0(t(s_{k-1}))| + |\hat\xi_0(t(s_{k-1}))-x(s_{k-1})|
\\
&
=|\xi_n(t(s_{k-1}))- \hat\xi_0(t(s_{k-1}))| + |\hat\xi_0(t(s_{k-1}))-x(s_{k-1})|
< \epsilon / 4.
\end{align*}
Following the same line of reasoning, we can show that
\cf{
\begin{align*}
x(s_{k+j})
\leq \hat\xi_0(t(s_{k+j})+)
&
\leq \hat\xi_0(t(s_{k+j+1})-)
\\
&\leq \hat \xi_0(t(s_{k+j+1}))
\\
&
\leq\hat\xi_0(t(s_{k+j+1})+)  \leq \hat\xi_0(t(s_{k+j+2})-) \leq x(s_{k+j+2}),
\end{align*}
i.e.,
$\hat \xi_0(t(s_{k+j+1})) \in [x(s_{k+j}), x(s_{k+j+2})]$, which implies
$|\hat \xi_0(t(s_{k+j+1})) - x(s_{k+j+1})| < \epsilon/8$, and hence
\begin{align*}
|y'(s'_{k+j+1}) - x(s_{k+j+1})|
&
=
|\xi_n(r'(s'_{k+j+1}))-x(s_{k+j+1})|
\\
&
= |\xi_n(r'(s'_{k+j+1}))- \hat\xi_0(t(s_{k+j+1}))| + |\hat\xi_0(t(s_{k+j+1}))-x(s_{k+j+1})|
\\
&
=|\xi_n(t(s_{k+j+1}))- \hat\xi_0(t(s_{k+j+1}))| + |\hat\xi_0(t(s_{k+j+1}))-x(s_{k+j+1})|
< \epsilon / 4.
\end{align*}
}%
$|y'(s'_{k+j+1}) - x(s_{k+j+1})| < \epsilon / 4$.
Noting that both $x$ and $y'$ are nondecreasing, there have to exist
$s_{k}', s_{k+1}', \ldots,  s_{k+j}'$
such that
$s_{k-1}'<s_{k}'< \cdots< s_{k+j}'<s_{k+j+1}'$
and
$|y'(s_l')-x(s_l)|< \epsilon/4$
for $l=k,k+1,\ldots,k+j$.
Note also that from (S2),
\begin{align*}
t(s_{l})-\epsilon/4 
&= t(s_{k})-\epsilon/4 < t(s_{k-1}) = r'(s'_{k-1}) \leq r'(s'_l) \leq r'(s'_{k+j+1}) = t(s_{k+j+1})
< t(s_{k+j})+\epsilon/4 
\\
&
= t(s_{l})+\epsilon/4,
\end{align*}
and hence,
$|r'(s'_l) - t(s_l)|< \epsilon/4$ for $l=k,\ldots,k+j$ as well.
Repeating this procedure for the $i$'s for which $s'_i$ is not designated until there is no such $i$'s are left, we can construct $s_1',\ldots, s_m'$ in such a way that
\begin{align*}
&|y'(s_i')-x(s_i)|< \epsilon/4
&
\text{and}
&
&
|r'(s_i') - t(s_i)|< \epsilon/4
\end{align*}
for all $i$'s.
Now, define a (piecewise linear) map $\lambda:[0,1]\to[0,1]$ by setting $\lambda(s_i) = s_i'$ at each $s_i$'s and interpolating $(s_i, s_i')$'s in between.
Then, $y\triangleq y'\circ \lambda$ and $r\triangleq r'\circ \lambda$ consist a parametrization $(y,r)$ of $\Gamma(\xi_n)$ such that $|x(s_i) - y(s_i)|<\epsilon/4$  and $|t(s_i) - r(s_i)| < \epsilon/4$ for each $i=1,\ldots, m$.
Due to the monotonicity of $x$, $y$, $t$, and $r$ along with (S2) and (S3), we conclude that $\|y-x\|_\infty< \epsilon/2$ and $\|t-r\|_\infty<\epsilon/2$, proving (\ref{eq:want-to-construct}).
\end{proof}

\begin{proposition}\label{thm:relative_compactness}
Let $K$ be a subset of $\mathbb D^\uparrow$.
If
	$M \triangleq \sup_{\xi\in K} \|\xi\|_\infty < \infty$
then $K$ is relatively compact w.r.t.\ the $M_1'$ topology.
\end{proposition}
\begin{proof}
Let $\{\xi_n\}_{n=1,2,\ldots}$ be a sequence in $K$. We prove that there exists a subsequence $\{\xi_{n_k}\}_{k=1,2,\ldots}$ and $\xi_0 \in \mathbb D$ such that $\xi_{n_k} \stackrel{M_1'}{\to} \xi_0$ as $k\to\infty$.
Let $T\triangleq \{t_n\}_{n=1,2,\ldots}$ be a dense subset of $[0,1]$ such that $1\in T$.
By the assumption, $\sup_{n=1,2,\ldots}|\xi_n(t_1)| < M$, and hence there is a subsequence $\{n_k^{(1)}\}_{k=1,2,\ldots}$ of $\{1,2,\ldots\}$ such that $\xi_{n_k^{(1)}}(t_1)$ converges to a real number $x_1 \in [-M, M]$.
For each $i\geq 1$, given $\{n_k^{(i)}\}$, one can find a further subsequence $\{n_k^{(i+1)}\}_{k=1,2,\ldots}$ of $\{n_k^{(i)}\}_{k=1,2,\ldots}$ in such a way that $\xi_{n_k^{(i+1)}}(t_{i+1})$ converges to a real number $x_{i+1}$.
Let $n_k \triangleq n_k^{(k)}$ for each $k=1,2,\ldots$.
Then, $\xi_{n_k}(t_i) \to x_i$ as $k\to\infty$ for each $i=1,2,\ldots$.
Define a function $\hat\xi_0:T\to\R$ on $T$ so that $\hat\xi_0(t_i) = x_i$.
We claim that $\hat \xi_0$ has left limit everywhere; more precisely, we claim that for each $s\in (0,1]$, if a sequence $\{s_n\}\subseteq T\cap [0,s)$ is such that $s_n\to s$ as $n\to\infty$, then $\hat\xi_0(s_n)$ converges as $n\to\infty$.
(With a similar argument, one can show that $\hat\xi_0$ has right limit everywhere---i.e., for each $s\in [0,1)$, if a sequence $\{s_n\}\subseteq T\cap (s,1]$ is such that $s_n\to s$ as $n\to\infty$, then $\hat\xi_0(s_n)$ converges as $n\to\infty$.)
To prove this claim, we proceed with proof by contradiction; suppose that the conclusion of the claim is not true---i.e., $\hat\xi_0(s_n)$ is not convergent.
Then, there exist a $\epsilon>0$ and a subsequence $r_n$ of $s_n$ such that
\begin{equation}\label{ineq:xi-hat-naught-are-far}
|\hat\xi_0(r_{n+1}) - \hat\xi_0(r_n)| > \epsilon.
\end{equation}
	\cf{
	In general, if a sequence $\{y_n\}$ of real numbers do not converge, then there has to be a positive constant $\epsilon>0$ and a subsequence $\{z_n\}$ such that $|z_n-z_{n+1}| > \epsilon$ for all $n\geq 1$.\\
	
	To see this, note that $z_n$ is not a Cauchy sequence, and hence, there has to be an $\epsilon>0$ and a subsequence $y'_n$ of $y_n$ such that $|y'_{2n-1}-y'_{2n}| > 2\epsilon$ for $n=1,2,\ldots$.
	Now we construct $z_n$ as follows.
	Let $z_1 = y'_2$.
	\begin{itemize}
	\item if $|y'_2-y'_3| > \epsilon$
	then set $z_2 = y'_3$ and $z_3 = y'_4$;
	\item if $|y'_2-y'_3|\leq \epsilon$,
	then, by the construction of $\{y'_n\}$ and the triangular inequality, $|y'_2 - y'_4| > \epsilon$; therefore, set $z_2 = y'_4$.
	\end{itemize}
	Repeat.
	}
Note that since $\hat\xi_0$ is a pointwise limit of nondecreasing functions $\{\xi_{n_k}\}$ (restricted on $T$),
\begin{itemize}
\item
$\hat \xi_0$ is also nondecreasing on $T$, \hfill(monotonicity)
\item
$\sup_{t\in T}|\hat \xi_0(t)| < M$. \hfill(boundedness)
\end{itemize}
However, these two are contradictory to each other since the monotonicity together with (\ref{ineq:xi-hat-naught-are-far}) implies $\hat \xi_0(r_{N_0+j})> \hat \xi_0(r_{N_0}) + j\epsilon$, which leads to the contradiction to the boundedness for a large enough $j$.
This proves the claim.

Note that the above claim means that $\hat \xi_0$ has both left  and right limit at each point of $T\cap(0,1)$, and due to the monotonicity, the function value has to be between the left limit and the right limit.
Since $T$ is dense in $[0,1]$, we can extend $\hat \xi$ from $T$ to $[0,1]$ by setting $\hat \xi_0(t) \triangleq \lim_{\substack{t_i\to t\\ t_i > t}} \hat \xi_0(t_i)$ for $t\in [0,1]\setminus T$.
Note that such $\hat\xi_0$ is an element of $\tilde \D$ and satisfies the conditions of Proposition~\ref{thm:convergence}.
We therefore conclude that $\xi_{n_k}\to \xi_0\in \D^\uparrow$ in $M_1'$ as $k\to \infty$, where
$
\xi_0(t) \triangleq \lim_{s\downarrow t}\hat \xi_0(s)
$
for $t\in[0,1)$ and $\xi_0(1) \triangleq \hat \xi_0(1)$.
This proves that $K$ is indeed relatively compact.
\end{proof}
Recall that our rate function for one-sided compound Poisson processes is as follows:
$$
I_{M_1'}(\xi) =
\begin{cases}
\sum_{t\in[0,1]} \big(\xi(t) - \xi(t-)\big)^\alpha
& \text{if $\xi$ is a non-decreasing pure jump function with $\xi(0)\geq 0$},
\\
\infty
& \text{otherwise.}
\end{cases}
$$
\begin{proposition}\label{thm:M_1p-goodness}
$I_{M_1'}$ is a good rate function w.r.t.\ the $M_1'$ topology.
\end{proposition}
\begin{proof}
In view of Proposition~\ref{thm:relative_compactness}, it is enough to show that the sublevel sets of $I_{M_1'}$ are closed.
Let $a$ be an arbitrary finite constant, and consider the sublevel set $\Psi_{I_{M_1'}}(a)\triangleq \{\xi\in\mathbb D: I_{M_1'}(\xi)\leq a\}$.
Let $\xi^c\in \mathbb D$ be any given path that does not belong to $\Psi_{I_{M_1'}}(a)$.
We will show that there exists $\epsilon>0$ such that $d_{M_1'}\big(\xi^c, \Psi_{I_{M_1'}}(a)\big) \geq \epsilon$.
Note that $\Psi_{I_{M_1'}}(a)^\stcomp = A\cup B\cup C\cup D$ where
\begin{align*}
A
&=\{\xi\in\mathbb D: \xi(0) < 0\},
\\
B
&=\{\xi\in\mathbb D: \xi \text{ is not a non-decreasing function}\},
\\
C
&=\{\xi\in\mathbb D: \xi \text{ is non-decreasing but not a pure jump function}\},
\\
D
&=\{\xi\in \mathbb D: \xi \text{ is a non-decreasing pure jump function with }\xi(0)\geq 0 \text{ and } \sum_{t\in[0,1]}\big(\xi(t)-\xi(t-)\big)^\alpha > a\}.
\end{align*}
In each case, we will show that $\xi^c$ is bounded away from $\Psi_{I_{M_1'}}(a)$.
In case $\xi^c \in A$, note that for any parametrization $(x,t)$ of $\xi^c$, there has to be $s^*\in[0,1]$ such that $x(s^*) = \xi^c(0)<0$.
On the other hand, $y(s) \geq 0$ for all $s\in[0,1]$ for any parametrization $(y,r)$ of $\zeta$ such that $\zeta \in \Psi_I(a)$, and hence, $\|x-y\|_\infty\geq \xi^c(0)$.
Therefore,
$$
d_{M_1'}(\xi^c, \zeta) \geq \inf_{\substack{(x,t)\in\Gamma(\xi^c)\\(y,r)\in \Gamma(\zeta)}}\|x- y\|_\infty \geq |\xi^c(0)|.
$$
Since $\zeta$ was an arbitrary element of $\Psi_{I_{M_1'}}(a)$, we conclude that $d_{M_1'}(\xi^c, \Psi_I(a)) \geq \epsilon$ with $\epsilon = |\xi^c(0)|$.

Using a similar argument, it is straightforward to show that any $\xi^c \in B$ is bounded away from $\Psi_I(a)^c$.

If $\xi^c \in C$, there has to be $T_s$ and $T_t$ such that $0\leq T_s < T_t\leq 1$, $\xi^c$ is continuous on $[T_s, T_t]$, and $c\triangleq  \xi^c(T_t)-\xi^c(T_s) > 0$.
Pick a small enough $\epsilon\in(0,1)$ so that
\begin{equation}\label{eq:M_1-prime-rate-function-choice-of-epsilon}
(4\epsilon)^{\alpha-1}(c-5\epsilon) > a.
\end{equation}
Note that since $\xi^c$ is uniformly continuous on $[T_s,T_t]$, there exists $\delta>0$ such that $|\xi^c(t) - \xi^c(s)|<\epsilon$ if $|t-s|\leq \delta$.
In particular, we pick $\delta$ so that $\delta < \epsilon$ and $T_s + \delta < T_t - \delta$.
We claim that
$$
d_{M_1'}(\Psi_{I_{M_1'}}(a),\xi^c) \geq \delta.
$$
Suppose not.
That is, there exists $\zeta \in \Psi_{I_{M_1'}}(a)$ such that $d_{M_1'}(\zeta, \xi^c) < \delta$.
Let $(x,t)\in \Gamma(\xi^c)$ and $(y,r)\in\Gamma(\zeta)$ be the parametrizations of $\xi^c$ and $\zeta$, respectively, such that $\|x-y\|_\infty + \|t - r\|_\infty < \delta$.
Since $I_{M_1'}(\zeta) \leq a <\infty$, one can find a finite set $K\subseteq \{t \in [0,1]: \zeta(t)-\zeta(t-) > 0\}$ of jump times of $\zeta$ in such a way that $\sum_{t\notin K} \big(\zeta(t)-\zeta(t-)\big)^\alpha < \epsilon$.
Note that since $\epsilon \in (0,1)$, this implies that $\sum_{t\notin K} \big(\zeta(t)-\zeta(t-)\big) < \epsilon$.
Let $T_1,\ldots,T_k$ denote (the totality of) the jump times of $\zeta$ in $K\cap (T_s + \delta, T_t - \delta]$, and let $T_0 \triangleq T_s + \delta$ and $T_{k+1} \triangleq T_t-\delta$.
That is, $\{T_1,\ldots, T_k\} = K\cap (T_s + \delta, T_t - \delta] = K\cap (T_0, T_{k+1}]$.
Note that
\begin{itemize}
\item There exist $s_0$ and $s_{k+1}$ in $[0,1]$ such that
$$
y(s_{0}) = \zeta(T_{0}), \qquad
r(s_{0}) = T_{0}, \qquad
y(s_{k+1}) = \zeta(T_{k+1}), \qquad
r(s_{k+1}) = T_{k+1}.
$$
\item For each $i=1,\ldots, k$, there exists $s_i^+$ and $s_i^-$ such that
$$
r(s_i^+) = r(s_i^-) = T_i,\qquad
y(s_i^+) = \zeta(T_i),\qquad
y(s_i^-) = \zeta(T_i-).
$$
\end{itemize}
Since $t(s_{k+1}) \in [r(s_{k+1})-\delta, r(s_{k+1}) + \delta] \subseteq [T_s,T_t]$, and $\xi^c$ is continuous on $[T_s,T_t]$ and non-decreasing,
\begin{align*}
y(s_{k+1}) 
&\geq x(s_{k+1}) - \delta = \xi^c(t(s_{k+1})) - \delta \geq \xi^c(r(s_{k+1})-\delta)-\delta
= \xi^c(T_{k+1} - \delta)-\delta 
\\
&
\geq \xi^c(T_{k+1}) - \epsilon - \delta 
\geq \xi^c(T_{k+1}) - 2\epsilon
.
\end{align*}
Similarly,
$$
y(s_0) \leq x(s_{0}) + \delta = \xi^c(t(s_{0})) + \delta \leq \xi^c(r(s_{0})+\delta)+\delta
= \xi^c(T_0 + \delta)+\delta \leq \xi^c(T_0) + \epsilon + \delta \leq \xi^c(T_0) + 2\epsilon
.
$$
Subtracting the two equations,
\begin{equation*}
y(s_{k+1}) - y(s_0) \geq \xi^c(T_{k+1}) - \xi^c(T_0) -4\epsilon = c - 4\epsilon.
\end{equation*}
Note that
\begin{align}\label{eq:M_1-prime-rate-function-on-the-one-hand}
\sum_{i=1}^k \big(\zeta(T_i)-\zeta(T_i-)\big)
&
=
\zeta(T_{k+1}) - \zeta(T_0) - \sum_{t\in (T_0,T_{k+1}]\cap K^\stcomp} (\zeta(t) - \zeta(t-)) \geq \zeta(T_{k+1}) - \zeta(T_0) - \epsilon
\nonumber
\\
&
=y(s_{k+1}) - y(s_0) -\epsilon \geq c-5\epsilon.
\end{align}
On the other hand,
\begin{align*}\label{eq:M_1-prime-rate-function-on-the-other-hand}
y(s_i^+) - y(s_i^-)
&
\leq
(x(s_i^+) + \delta) - (x(s_i^-)-\delta) = x(s_i^+) - x(s_i^-) + 2\delta
\leq \xi^c(t(s_i^+)) - \xi^c(t(s_i^-)) + 2\delta
\\
&
\leq
\xi^c(r(s_i^+)+\delta) - \xi^c(r(s_i^-) - \delta) + 2\delta
\leq
\xi^c(T_i+\delta) - \xi^c(T_i - \delta) + 2\delta
\leq 2\epsilon + 2\delta \leq 4 \epsilon.
\end{align*}
That is, $(\zeta(T_i) - \zeta(T_i-))^{\alpha-1} = (y(s_i^+)-y(s_i^-))^{\alpha-1} \geq (4\epsilon)^{\alpha -1}$.
Combining this with \eqref{eq:M_1-prime-rate-function-on-the-one-hand},
\begin{align*}
I_{M_1'}(\zeta) \geq \sum_{i=1}^k \big(\zeta(T_i)-\zeta(T_i-)\big)^\alpha
=
\sum_{i=1}^k \big(\zeta(T_i)-\zeta(T_i-)\big)\big(\zeta(T_i)-\zeta(T_i-)\big)^{\alpha -1}
\geq
(c-5\epsilon)(4\epsilon)^{\alpha-1}
>a,
\end{align*}
which is contradictory to the assumption that $\zeta \in \Psi_{I_{M_1'}}(a)$.
Therefore, the claim that $\xi^c$ is bounded away from $\Psi_{I_{M_1'}}(a)$ by $\delta$ is proved.

Finally, suppose that $\xi^c\in D$.
That is, there exists $\{(z_i,u_i)\in \R_+ \times [0,1]\}_{i=1,\ldots}$ such that $\xi^c = \sum_{i=1}^\infty z_i \I_{[u_i,1]}$ where $u_i$'s are all distinct and $\sum_{i=1}^\infty z_i^\alpha  > a$.
Pick sufficiently large $k$ and sufficiently small $\delta>0$ such that $\sum_{i=1}^k (z_i - 2\delta)^\alpha> a$
and $u_{i+1} - u_i > 2\delta$ for $i=1,\ldots, k-1$.
We claim that $d_{M_1'}(\zeta, \xi^c) \geq \delta$ for any $\zeta \in \Psi_{I_{M_1'}}(a)$.
Suppose not and there is $\zeta \in \Psi_{I_{M_1'}}(a)$ such that $\|x-y\|_\infty + \|t-r\|_\infty < \delta$ for some parametrizations $(x,t)\in \Gamma(\xi^c)$ and $(y,r)\in \Gamma(\zeta)$.
Note first that there are $s_i^+$'s and $s_i^-$'s for each $i=1,\ldots, k$ such that
$t(s_i^-) = t(s_i^+) = u_i$, $x(s_i^-) = \xi^c(u_i-)$, and $x(s_i^+) = \xi^c(u_i)$.
Since $y(s_i^+) \geq x(s_i^+) - \delta = \xi^c(u_i) - \delta$
and
$y(s_i^-) \leq x(s_i^-)+ \delta = \xi^c(u_i-) + \delta$,
$$
\zeta(r(s_i^+)) - \zeta(r(s_i^-)) \geq y(s_i^+) - y(s_i^-) \geq \xi^c (u_i) - \xi^c(u_i^-) - 2\delta = z_i - 2\delta.
$$
Note that by construction, $r(s_i^+) < t(s_i^+) +\delta  = u_i + \delta < u_{i+1} - \delta = t(s_{i+1}^-) - \delta < r(s_{i+1}^-)$ for each $i=1,\ldots, k-1$, and hence, along with the subadditivity of $x\mapsto x^\alpha$,
$$
I_{M_1'}(\zeta) = \sum_{t\in[0,1]} (\zeta(t) - \zeta(t-)\big)^\alpha \geq \sum_{i=1}^k [\zeta(r(s_i^+)) - \zeta(r(s_i^-)) ]^\alpha\geq \sum_{i=1}^k (z_i - 2\delta)^\alpha > a,
$$
which is contradictory to the assumption $\zeta\in \Psi_{I_{M_1'}}(a)$.
\end{proof}


%
%
%
%


\bibliographystyle{imsart-nameyear}
\bibliography{BBRZ-REF.bib}

\end{document}